\documentclass[reqno,11pt,a4paper]{amsart}
\usepackage[a4paper,left=30mm,right=30mm,top=30mm,bottom=30mm,marginpar=20mm]{geometry}
\usepackage{amsmath}
\usepackage{amssymb}
\usepackage{amsthm}
\usepackage{mathrsfs}
\usepackage{setspace}
\usepackage{xcolor}
\usepackage{pict2e}
\usepackage{graphicx}
\usepackage{enumitem}

\setlist[enumerate]{font=\normalfont}

\usepackage{cite}
\usepackage[colorlinks=true, linkcolor=blue, urlcolor=blue, citecolor=green]{hyperref}

\newcommand{\wkto}{\mathrel{\smash{{\overset{\raisebox{-0.2ex}{$\scriptstyle\ast$}}{\rightharpoonup}}}}}

\newcommand{\subext}[1]{\scalebox{1.2}{$\wedge$}_{#1}\mkern2mu}

\newcommand{\mvert}{\mathrel{
		\begingroup
		\begin{picture}(7,7)
			\roundcap\roundjoin
			\linethickness{0.6pt}
			\polyline(0,7)(0,0)(7,0)
		\end{picture}
		\endgroup
}}

\newcommand{\Crm}{\mathrm{C}}

\newcommand{\Irm}{\mathrm{I}}

\newcommand{\Lrm}{\mathrm{L}}

\newcommand{\Nrm}{\mathrm{N}}

\newcommand{\Wrm}{\mathrm{W}}

\newcommand{\Bcal}{\mathcal{B}}

\newcommand{\Ecal}{\mathcal{E}}

\newcommand{\Mcal}{\mathcal{M}}

\newcommand{\Qcal}{\mathcal{Q}}

\newcommand{\Tcal}{\mathcal{T}}

\newcommand{\Wcal}{\mathcal{W}}

\newcommand{\Fbf}{\mathbf{F}}

\newcommand{\Mbf}{\mathbf{M}}

\newcommand{\Ebb}{\mathbb{E}}

\DeclareMathOperator*{\esssup}{ess\,sup}

\newcommand{\dd}{\;\mathrm{d}}

\newcommand{\N}{\mathbb{N}}
\newcommand{\R}{\mathbb{R}}

\newcommand{\Var}{\mathrm{Var}}
\newcommand{\Diss}{\mathrm{Diss}}
\newcommand{\NLip}{\Nrm^{\mathrm{Lip}}}
\newcommand{\ILip}{\Irm^{\mathrm{Lip}}}
\newcommand{\Lip}{\mathrm{Lip}}

\newcommand{\FI}{\mathbf{FI}}
\newcommand{\FN}{\mathbf{F}}

\newcommand{\Disl}{\mathrm{Disl}}
\newcommand{\Slip}{\mathrm{Slip}}
\newcommand{\BV}{\mathrm{BV}}

\newcommand{\sym}{\mathrm{sym}}
\renewcommand{\skew}{\mathrm{skew}}
\newcommand{\tr}{\mathrm{tr}\hspace{0.2em}}
\newcommand{\fr}{\mathrm{fr}}
\newcommand{\symfr}{\mathrm{sym}\hspace{0.2em}\mathrm{fr}}
\newcommand{\supp}{\mathrm{supp}\hspace{0.2em}}

\renewcommand{\div}{\mathrm{div}\hspace{0.2em}}
\newcommand{\curl}{\mathrm{curl}}

\newcommand{\tbf}{\mathbf{t}}
\newcommand{\pbf}{\mathbf{p}}

\newenvironment{subproof}[1][]
  {\par\vspace{0.5\baselineskip}
   \noindent{#1.}}
  {\par}

\newtheorem{definition}{Definition}[section]
\newtheorem{proposition}[definition]{Proposition}
\newtheorem{lemma}[definition]{Lemma}
\newtheorem{theorem}[definition]{Theorem}
\newtheorem{corollary}[definition]{Corollary}
\newtheorem*{claim}{Claim}
\newtheorem*{theorem*}{Theorem}

{\theoremstyle{remark}\newtheorem{remark}[definition]{Remark}}

\numberwithin{equation}{section}

\begin{document}
	
\title[Homogenisation of Elasto-Plastic Evolutions]{Homogenisation of Elasto-Plastic Evolutions driven by the Flow of Dislocations}
	
\begin{abstract}

This work establishes the existence of rate-independent elastoplastic evolutions driven by dislocation fields in single crystals via a homogenisation procedure. We use the theory of space-time currents to place the Small-Distortion Geometric Model of [arXiv:2609.23728] on a rigorous footing and conduct the homogenisation by using recently developed normal-to-integral strict approximation results for space-time currents which respect the inherent space-time structure. We prove both a homogenisation and recovery result and thus demonstrate that the field version of the Small-Distortion Geometric Model is well-justified from microscopic principles.
    	
\vspace{.5em}
        
\noindent\textsc{Date:} \today{}
\end{abstract}

\author{Paolo Bonicatto}
\address{Universit\`a di Trento, Dipartimento di Matematica, Via Sommarive 14, 38123 Trento, Italy}
\email{paolo.bonicatto@unitn.it}
	
\author{Filip Rindler}
\address{Mathematics Institute, University of Warwick, Coventry, CV4 7AL, United Kingdom}
\email{F.Rindler@warwick.ac.uk}
	
\author{Harry Turnbull}
\address{Mathematics Institute, University of Warwick, Coventry, CV4 7AL, United Kingdom}
\email{harry.turnbull@warwick.ac.uk}

\maketitle

\section{Introduction}

In crystalline materials, the principal mechanism behind macroscopic plastic deformation is the motion of dislocations, that is, lines of defects in the crystal lattice~\cite{AndersonHirthLothe17book,HullBacon11book,HanReddy13book}. The amount of dislocations in a typical material specimen is very large, with the total length of all dislocation lines in well-annealed metals being on the order of~$10^7$ to $10^8$ cm per cm$^3$~\cite{HullBacon11book}. Consequently, a \emph{field description} is indicated in order to abstract away from individual lines.

How to formulate such a field description that is rich enough to serve as the foundation of a convincing theory of macroscopic elasto-plasticity in crystalline materials is an old and important goal of theoretical solid mechanics. Indeed, such an approach is much closer to the underlying physics than the various proposals for phenomenological internal variables (such as backstresses affecting the critical resolved shear stress in von Mises or Tresca plasticity~\cite{HanReddy13book}). However, despite much work in the field over the last century (see~\cite{Rindler26?} for some historical background), identifying a suitable representation of dislocation fields that has all the required properties has proved to be a very challenging problem, and no universally accepted theory has emerged to date.

One approach to this question is to ``homogenise'' a model describing the motion and resulting plasticity for \emph{discrete} dislocation lines into a model based on dislocation density \emph{fields}. In fact, in~\cite{Rindler26?} it was postulated that this \emph{discrete-to-field dislocation homogenization problem} should be resolved by means of a mathematically rigorous limiting procedure.

For this purpose,~\cite{Rindler26?} developed the Small-Distortion Geometric Model for dislocation-driven elasto-plasticity, which is based on~\cite{HudsonRindler22}. One of the distinguishing features of this model is that it uses the theory of space-time currents~\cite{Rindler23,BonicattoRindlerTurnbull26} to represent dislocation motion. In particular, the evolution of discrete dislocation lines and dislocation density fields are described using space-time integral and normal $2$-currents respectively. The decisive point is that one can pass from the integral (i.e. discrete) setting to the normal (i.e. field) setting using the weak$^{*}$ topology on $2$-currents. On the other hand, to conduct a homogenisation procedure in classical formulations, one would need to pass to the limit in products of weakly converging sequences, and in general it is not possible to express such limits using the limits of the individual sequences.

The main contribution of the present article is to show that, in the setting of small-strain (geometrically linear) plasticity in single crystals with rate-independent dynamics, the aforementioned discrete-to-field homogenization procedure can indeed be carried out in the setting of the Small-Distortion Geometric Model of~\cite{Rindler26?}.

The notion of solution that we shall consider is the \textit{energetic solution} of \cite[Section 9.2]{Rindler26?}, which is in turn based on~\cite{MielkeTheil99,MielkeTheil04,MielkeTheilLevitas02}. After conducting a rigorous implementation of \cite{Rindler26?} in Sections \ref{s:DislSlip} and \ref{s:DispDistQ}, we define \textit{discrete} and \textit{field} variants of the aforementioned energetic solution (which describe the evolution of discrete dislocation lines and dislocation density fields respectively) in Section \ref{s:Solutions}. We then establish the following results:

\begin{itemize}
    \item \textit{Existence of Discrete Energetic Solutions (Theorem \ref{tm:DiscreteExist}):} We establish that, for an energetically stable (discrete) initial datum, there is a discrete energetic solution, starting from this initial datum, satisfying a posteriori estimates. In fact, we show that, if the initial datum is only \textit{approximately} stable, then we can find an \textit{approximate} discrete energetic solution starting from this initial datum with error parameters controlled by the error parameters in the initial datum.
    \item \textit{Homogenisation of Discrete Energetic Solutions (Theorem \ref{tm:Forwards}):} We show that a converging sequence of (appropriately weighted) approximate discrete energetic solutions with vanishing errors has a field energetic solution as its limit.
    \item \textit{Existence of Field Energetic Solutions (Theorem \ref{tm:FieldExist}):} We establish that, for an energetically stable (field) initial datum, there is a field energetic solution starting from this initial datum; the proof proceeds by approximating the initial data by a sequence of (appropriately weighted) discrete initial data, applying Theorem \ref{tm:DiscreteExist} to each element of the sequence, using compactness results to pass to a limit, and appealing to Theorem \ref{tm:Forwards}.
    \item \textit{Recovery of Discrete Energetic Solutions (Theorem \ref{tm:Backwards}):} We show that field energetic solutions can be recovered as limits (in a stronger sense than that of Theorem \ref{tm:Forwards}) of (appropriately weighted) approximate discrete energetic solutions with vanishing errors. 
\end{itemize}

Together, these results establish that the field version of the Small-Distortion Geometric Model is well-justified from microscopic principles.

\subsection*{Acknowledgments}

The authors would like to thank Amit Acharya, Gilles Francfort, Thomas Hudson, and Alexander Mielke  for discussions related to this work. This work was supported by the Italian Ministry of University and Research (MUR) through the FIS2 project BRIDGE: ‘‘Beyond Regularity In Differential Geometric Evolutions'', project code FIS-2023-02224 (CUP: E53C25001790001). This project has received funding from the UKRI Frontier Research Guarantee (ERC guarantee) grant EP/Z000297/1 (ERC CONCENTRATE). H.T.\ was supported by the Warwick Mathematics Institute Centre for Doctoral Training and the EPSRC grant EP/W524645/1. The authors have no relevant financial or non-financial interests to disclose. 

\subsection*{AI Declaration} AI has been used for proofreading and literature search. This manuscript contains no AI-generated text. 

\section{General Assumptions}\label{s:Assumptions} The following will be assumed throughout the paper:

\begin{enumerate}
    \item[\textbf{A1.}] \textbf{Basic Assumptions:}
    \begin{enumerate}[label=(A1-\roman*),ref=A1-\roman*]
        \item\label{A:1-i} The \textbf{reference configuration} of a crystal specimen is represented by a  bounded, contractible  domain $\Omega \subset \R^{3}$ with a $\Crm^{1}$ boundary. We model the deformation of $\Omega$ over a time interval $[0,T]$ for some fixed $T > 0$.
        \item\label{A:1-ii} The system of \textbf{Burgers vectors} is finite and given by $\Bcal := \{\pm b_{1},\dotsc, \pm b_{m}\}$.
        \item\label{A:1-iii} Our \textbf{boundary conditions} will be formulated using a non-empty, open holding set $H \subset \Omega$, a holding value $h_{0} \in \R^{3}$, and a smooth, compactly supported, non-zero $\chi: \Omega \to [0,1]$.
    \end{enumerate}
    \item[\textbf{A2.}] \textbf{Energetic Assumptions:}
    \begin{enumerate}[label=(A2-\roman*),ref=A2-\roman*]
        \item\label{A:2-i} The \textbf{elasticity tensor} $\Ebb: \R^{3 \times 3} \to \R^{3 \times 3}$ is a linear map which is positive definite on $\R^{3 \times 3}_{\sym}$ in the sense that there is $C > 0$ with the property that
        \begin{align*}
            A : \Ebb A \geq C \cdot \vert A \vert^{2} \hspace{1em} \text{for $A \in \R^{3 \times 3}_{\sym}$}.
        \end{align*}
        Furthermore, we assume that $\Ebb$ is symmetric in the sense that
        \begin{align*}
            A : \Ebb B &= \sym\hspace{0.2em}A:\Ebb(\sym\hspace{0.2em}B), & A : \Ebb B &= \Ebb A : B \hspace{1em} \text{for $A$, $B \in \R^{3 \times 3}$}.
        \end{align*}
        \item\label{A:2-ii} For $b \in \Bcal$, the \textbf{core line tension} $\psi^{b}$ is a norm on $\subext{1}\R^{3}$. We also assume that $\psi^{b} = \psi^{-b}$ for $b \in \Bcal$.
        \item\label{A:2-iii} The \textbf{external loading} is given by $f \in \Crm^{1}([0,T];\Lrm^{\infty}(\Omega;\R^{3}))$.
        \item\label{A:2-iv} A \textbf{hardening mechanism} is described by a continuous and increasing map $h: [0,\infty) \to [0,\infty)$. We also assume that there exist $C > 0$ and $q > 1$ with the property that $h(s) \geq C^{-1}s^{q} - C$ for $s \geq 0$.
        \item\label{A:2-v} For $b \in \Bcal$, the \textbf{dissipation potential} $\phi^{b}$ is a norm on $\subext{2}\R^{3}$. We also assume that $\phi^{b} = \phi^{-b}$ for $b \in \Bcal$. 
    \end{enumerate}
\end{enumerate}

In what follows, we shall often use $C > 0$ to denote a constant that depends only on the data given above (unless stated otherwise). As usual, this constant may change from line to line. 

In Section \ref{s:CompatibleLoad}, we shall introduce the further assumption \eqref{A:3} on the loading $f$ given above; we shall additionally assume this condition in the existence results (Theorems \ref{tm:DiscreteExist} and \ref{tm:FieldExist}) but not in the homogenisation and recovery results (Theorems \ref{tm:Forwards} and \ref{tm:Backwards}).

\section{Dislocation and Slip Trajectory Systems}\label{s:DislSlip}

In this section, we take the model for dislocations and slip trajectories presented in \cite{HudsonRindler22,Rindler26?} and we explore some consequences of results from \cite{Rindler23,BonicattoRindlerTurnbull26,BonicattoRindlerTurnbull26-I?} in this setting. We shall assume that the reader is familiar with the theory of currents (see \cite[Chapter 4]{Federer69book} for a thorough treatment), and we defer to \cite{BonicattoRindlerTurnbull26,BonicattoRindlerTurnbull26-I?} for any undefined notation.

\subsection{Dislocations}\label{s:Disl} Let $\boldsymbol{T} = (T^{b})_{b \in \Bcal} \subset \Nrm_{1}(\overline{\Omega})$. We say that $\boldsymbol{T}$ is a \textbf{dislocation system} on $\overline{\Omega}$ if 
\begin{align*}
    \partial T^{b} = 0 \hspace{1em} \text{and} \hspace{1em} T^{-b} = -T^{b} \hspace{1em} \text{for $b \in \Bcal$}.
\end{align*}
We shall use $\overline{\Disl}(\overline{\Omega})$ to denote the set of dislocation systems on $\overline{\Omega}$. For $n \in \N$, we say that $\boldsymbol{T}$ is an \textbf{$n$-discrete dislocation system} on $\overline{\Omega}$ if we additionally have
\begin{align*}
    n \cdot T^{b} \in \Irm_{1}(\overline{\Omega}) \hspace{1em} \text{for $b \in \Bcal$}.
\end{align*}
We shall use $\Disl_{1/n}(\overline{\Omega})$ to denote the set of $n$-discrete dislocation systems on $\overline{\Omega}$. Here and in what follows, we shall use the word `discrete' instead of `$1$-discrete', and omit any subscripts of the form $1/1$; we also set
\begin{align*}
    \boldsymbol{0} := (0)_{b \in \Bcal} \in \Disl(\overline{\Omega}).
\end{align*}

The \textbf{core energy} of $\boldsymbol{T}$ is given by
\begin{align*}
    \Wcal_{c}(\boldsymbol{T}) := \frac{1}{2} \sum_{b \in \Bcal} \Mbf_{\psi^{b}}(T^{b}),
\end{align*}
where, for $T \in \Nrm_{1}(\R^{3})$ and $b \in \Bcal$,
\begin{align*}
    \Mbf_{\psi^{b}}(T) := \int \psi^{b}(\vec{T}(x)) \dd \Vert T \Vert(x).
\end{align*}
One can check that
\begin{align}\label{CoreVsMass}
    C^{-1} \cdot \sum_{b \in \Bcal} \Mbf(T^{b}) \leq \Wcal_{c}(\boldsymbol{T}) \leq C \cdot \sum_{b \in \Bcal} \Mbf(T^{b}).
\end{align} 

Let $(\boldsymbol{T}_{j})_{j} \subset \overline{\Disl}(\overline{\Omega})$ and take $\boldsymbol{T} \in \overline{\Disl}(\overline{\Omega})$. We write $\boldsymbol{T}_{j} \wkto \boldsymbol{T}$ if $T_{j}^{b} \wkto T^{b}$ for $b \in \Bcal$. It follows from \cite[Lemma 2.2]{BonicattoRindlerTurnbull26-I?} that, if $\boldsymbol{T}_{j} \wkto \boldsymbol{T}$, then
\begin{align}\label{CoreLSC}
    \Wcal_{c}(\boldsymbol{T}) \leq \liminf_{j \to \infty} \Wcal_{c}(\boldsymbol{T}_{j}).
\end{align}
Now take $n \in \N$ and suppose that
\begin{align*}
    (\boldsymbol{T}_{j})_{j} &\subset \Disl_{1/n}(\overline{\Omega}), & \sup_{j \in \N} \Wcal_{c}(\boldsymbol{T}_{j}) < \infty.
\end{align*}
If $\boldsymbol{T}_{j} \wkto \boldsymbol{T}$, then we can use the Federer-Fleming closure theorem (see \cite[4.2.16]{Federer69book}) on the sequences $(n \cdot T^{b}_{j})_{j}$ for $b \in \Bcal$ to deduce that $n \cdot T^{b} \in \Irm_{1}(\overline{\Omega})$ for $b \in \Bcal$ and so $\boldsymbol{T} \in \Disl_{1/n}(\overline{\Omega})$.

\subsection{Slip Trajectories} 
Fix $\sigma,\tau \in \R$ with $\sigma < \tau$. Let $\boldsymbol{S} = (S^{b})_{b \in \Bcal} \subset \Nrm_{2}([\sigma,\tau] \times \overline{\Omega})$. We say that $\boldsymbol{S}$ is a \textbf{slip trajectory system} (on $[\sigma,\tau]$) if
\begin{align*}
    \partial S^{b} \mvert (\sigma,\tau) \times \R^{3} = 0 \hspace{1em} \text{and} \hspace{1em} S^{-b} = -S^{b} \hspace{1em} \text{for $b \in \Bcal$}.
\end{align*}
To recover dislocation systems from slip trajectories, we define the \textbf{slice} $\boldsymbol{S}(t)$ of $\boldsymbol{S}$ at $t$ for $t \in [\sigma,\tau]$ as follows:
\begin{align*}
    \boldsymbol{S}(t) := (S^{b}(t))_{b \in \Bcal}.
\end{align*}
Here, for $b \in \Bcal$, $S^{b}(t) := \pbf_{*}S^{b}\vert_{t}$ and $S^{b}\vert_{t}$ is the slice of $S^{b}$ with respect to the map $\tbf$ at $t$ (see \cite[Section 3.1]{BonicattoRindlerTurnbull26}). We emphasise that $\boldsymbol{S}(t)$ exists for \textit{all} $t \in [\sigma,\tau]$ by \cite[Proposition 3.4]{BonicattoRindlerTurnbull26}. We also define the \textbf{initial value} and \textbf{final value} of $\boldsymbol{S}$ as follows:
\begin{align*}
    \partial^{-}\boldsymbol{S} &:= (\partial^{-}S^{b})_{b \in \Bcal}, &
    \partial^{+}\boldsymbol{S} &:= (\partial^{+}S^{b})_{b \in \Bcal},
\end{align*}
where, for $b \in \Bcal$,
\begin{align}\label{EndpointsDefn}
    \partial^{-}S^{b} &:= -\pbf_{*}(\partial S^{b} \mvert \{\sigma\} \times \R^{3}), & 
    \partial^{+}S^{b} &:= \pbf_{*}(\partial S^{b} \mvert \{\tau\} \times \R^{3}).
\end{align}
One can check that
\begin{align*}
    \partial^{-}\boldsymbol{S}, \partial^{+}\boldsymbol{S} \in \overline{\mathrm{\Disl}}(\overline{\Omega}).
\end{align*}
Notice that we are abusing notation by omitting $\sigma,\tau$ from $\partial^{-}\boldsymbol{S}$ and $\partial^{+}\boldsymbol{S}$, but the choice of interval shall always be clear from the context. Take $b \in \Bcal$. One may expect that $\partial^{-}S^{b} = S^{b}(\sigma)$ and that $\partial^{+}S^{b} = S^{b}(\tau)$, but in fact one can use \cite[Proposition 3.4]{BonicattoRindlerTurnbull26} to show that
\begin{equation}\label{SlicesToEndpoints}
\begin{aligned}
    S^{b}(\sigma) &= \frac{1}{2}(\partial^{-}S^{b} + \pbf_{*}\partial(S^{b} \mvert \{\sigma\} \times \R^{3})), \\
    S^{b}(\tau) &= \frac{1}{2}(\partial^{+}S^{b} - \pbf_{*}\partial(S^{b} \mvert \{\tau\} \times \R^{3})),
\end{aligned}
\end{equation}
which may not agree with $\partial^{-}S^{b}$ and $\partial^{+}S^{b}$ respectively. Nevertheless, we can use \eqref{SlicesToEndpoints} to deduce that $\partial (S^{b}(\sigma)) = \partial(S^{b}(\tau)) = 0$. For $t \in (\sigma,\tau)$, one can use the boundary formula (see \cite[Proposition 3.4]{BonicattoRindlerTurnbull26}) to deduce that $\partial (S^{b}(t)) = 0$, and it is straightforward to check that $S^{-b}(t)=-S^{b}(t)$. It follows from \cite[Proposition 3.1]{BonicattoRindlerTurnbull26} that $S^{b}(t) \in \Nrm_{1}(\overline{\Omega})$ for $\mathscr{L}^{1}$-a.e.\ $t \in [\sigma,\tau]$, but it is not necessarily the case that $\Mbf(S^{b}(t)) < \infty$ for \textit{all} $t \in [\sigma,\tau]$. 
In particular, we have
\begin{align}\label{SliceDisl}
    \boldsymbol{S}(t) \in \overline{\Disl}(\overline{\Omega}) \hspace{2em} \text{for $\mathscr{L}^{1}$-a.e.\ $t \in [\sigma,\tau]$}.
\end{align}
In view of \eqref{SliceDisl} and \cite[Proposition 5.2]{BonicattoDelNinRindler25}, we shall use 
\begin{align*}
    \BV([\sigma,\tau];\overline{\Disl}(\overline{\Omega}))
\end{align*}
to denote the set of slip trajectory systems on $[\sigma,\tau]$. 

For $n \in \N$, we say that $\boldsymbol{S} \in \BV([\sigma,\tau];\overline{\Disl}(\overline{\Omega})) $ is an \textbf{$n$-discrete slip trajectory system} if 
\begin{align*}
    n \cdot S^{b} \in \Irm_{2}([\sigma,\tau] \times \overline{\Omega}) \hspace{1em} \text{for $b \in \Bcal$}.
\end{align*}
It follows from \cite[4.3.6]{Federer69book} that
\begin{align*}
    \boldsymbol{S}(t) \in \Disl_{1/n}(\overline{\Omega}) \hspace{2em} \text{for $\mathcal{L}^{1}$-a.e.\ $t \in [\sigma,\tau]$},
\end{align*}
and so we shall use 
\begin{align*}
    \BV([\sigma,\tau];\Disl_{1/n}(\overline{\Omega}))
\end{align*}
to denote the set of $n$-discrete slip trajectory systems on $[\sigma,\tau]$. 

We now define the two physically relevant energies associated to $\boldsymbol{S}$ that we shall need in the sequel. Take a Borel $I \subset \R$. The \textbf{dissipation} of $\boldsymbol{S}$ on $I$ is given by
\begin{align*}
    \Diss(\boldsymbol{S};I) := \frac{1}{2} \sum_{b \in \Bcal} \Var_{\phi^{b}}(S^{b};I)
\end{align*}
where, for $S \in \Nrm_{2}(\R^{1+3})$ and $b \in \Bcal$,
\begin{align*}
    \Var_{\phi^{b}}(S;I) := \int_{I \times \R^{3}} \phi^{b}(\pbf\vec{S}(t,x)) \dd \Vert S \Vert(t,x).
\end{align*}
On the other hand, the \textbf{maximum core energy} associated to $\boldsymbol{S}$ on $I$ is given by
\begin{align*}
    \Vert \boldsymbol{S} \Vert_{\Lrm^{\infty}_{c}(I)} := \esssup_{t \in I} \Wcal_{c}(\boldsymbol{S}(t)).
\end{align*}
Notice that 
\begin{align}\label{DissVsVar}
    C^{-1} \cdot \sum_{b \in \Bcal} \Var(S^{b};I) \leq \Diss(\boldsymbol{S};I) \leq C \cdot \sum_{b \in \Bcal} \Var(S^{b};I)
\end{align}
and
\begin{align}\label{MaxCoreVsLInfty}
    C^{-1} \cdot \sum_{b \in \Bcal} \Vert S^{b} \Vert_{\Lrm^{\infty}(I)} \leq \Vert \boldsymbol{S} \Vert_{\Lrm^{\infty}_{c}(I)} \leq C \cdot \sum_{b \in \Bcal} \Vert S^{b} \Vert_{\Lrm^{\infty}(I)}.
\end{align} 
Furthermore, we set
\begin{align*}
    \Diss(\boldsymbol{S}) &:= \Diss(\boldsymbol{S};I) \\
    \Vert \boldsymbol{S} \Vert_{\Lrm^{\infty}_{c}} &:= \Vert \boldsymbol{S} \Vert_{\Lrm^{\infty}_{c}(I)}
\end{align*}
whenever $\tbf(\supp S) \subset I$.

The \textbf{mass-measure} of $\boldsymbol{S}$ is the following positive finite Borel measure on $[\sigma,\tau] \times \overline{\Omega}$:
\begin{align*}
    \Vert \boldsymbol{S} \Vert := \frac{1}{2}\sum_{b \in \Bcal}\Vert S^{b} \Vert.
\end{align*}
Now take $\boldsymbol{T} \in \overline{\mathrm{\Disl}}(\overline{\Omega})$. We say that $\boldsymbol{S} \in \BV([\sigma,\tau];\overline{\Disl}(\overline{\Omega}))$ is a \textbf{Lipschitz slip trajectory system} on $[\sigma,\tau]$ starting from $\boldsymbol{T}$ if $\partial^{-}\boldsymbol{S} = \boldsymbol{T}$ and
\begin{align}\label{LipTrajConds}
    \tbf_{\#}\Vert \boldsymbol{S} \Vert(\{\sigma,\tau\}) &= 0, & \Vert \boldsymbol{S} \Vert_{\Lrm^{\infty}_{c}} &< \infty, & \Diss(\boldsymbol{S};(\sigma,\cdot)) &\in \Lip((\sigma,\tau)).
\end{align}
We define the \textbf{Lipschitz constant} $\Lip(\boldsymbol{S})$ of $\boldsymbol{S}$ as the Lipschitz constant of the map $t \mapsto \Diss(\boldsymbol{S};(\sigma,t))$ on $(\sigma,\tau)$. One can check that \eqref{LipTrajConds} is equivalent to requiring that
\begin{align*}
    S^{b} \in \NLip_{2}([\sigma,\tau] \times \overline{\Omega}) \hspace{1em} \text{for $b \in \Bcal$},
\end{align*}
where $\NLip_{2}([\sigma,\tau] \times \R^{3})$ is the space of Lip-normal $2$-currents on $[\sigma,\tau] \times \R^{3}$ defined in~\cite[Section 3.2]{BonicattoRindlerTurnbull26}. 
We shall use 
\begin{align*}
    \overline{\Slip}(\boldsymbol{T};[\sigma,\tau])
\end{align*}
to denote the set of Lipschitz slip trajectory systems starting from $\boldsymbol{T}$. Notice that $\overline{\Slip}(\boldsymbol{T};[\sigma,\tau])$ is \textit{not} a vector space. For $n \in \N$, we shall use $\Slip_{1/n}(\boldsymbol{T};[\sigma,\tau])$ to denote the set of Lipschitz slip trajectory systems on $[\sigma,\tau]$ starting from $\boldsymbol{T}$ which are also $n$-discrete, but notice that $\Slip_{1/n}(\boldsymbol{T};[\sigma,\tau]) = \varnothing$ if $\boldsymbol{T} \not\in \Disl_{1/n}(\overline{\Omega})$. For convenience, we set
\begin{align*}
    \overline{\Slip}(\boldsymbol{T}) &:= \overline{\Slip}(\boldsymbol{T};[0,1]), \\
    \Slip_{1/n}(\boldsymbol{T}) &:= \Slip_{1/n}(\boldsymbol{T};[0,1]).
\end{align*}

The following result will allow us to estimate the core energy of the endpoints of a Lipschitz slip trajectory in terms of its maximal core energy:

\begin{lemma}\label{lm:CoreEndpointsByMaxCore} Fix $\sigma,\tau \in \R$ with $\sigma < \tau$. Let $\boldsymbol{T} \in \overline{\Disl}(\overline{\Omega})$ and take $\boldsymbol{S} \in \overline{\Slip}(\boldsymbol{T};[\sigma,\tau])$. Then
\begin{align*}
    \Wcal_{c}(\partial^{-}\boldsymbol{S}) &\leq \Vert \boldsymbol{S} \Vert_{\Lrm^{\infty}_{c}}, & \Wcal_{c}(\partial^{+}\boldsymbol{S}) &\leq \Vert \boldsymbol{S} \Vert_{\Lrm^{\infty}_{c}}.
\end{align*}
\end{lemma}

\begin{proof} The proof follows that of \cite[Lemma 4.1]{BonicattoRindlerTurnbull26}, using \eqref{CoreLSC} in place of the lower semicontinuity of the mass with respect to the weak$^{*}$ convergence of currents.
\end{proof}

We shall now investigate how rescaling slip trajectory systems in time affects the quantities defined above. Again, fix $\sigma,\tau \in \R$ with $\sigma < \tau$ and recall that $\Lip([\sigma,\tau])$ denotes the space of Lipschitz functions $a: [\sigma,\tau] \to \R$. Take $\boldsymbol{S} \in \BV([\sigma,\tau];\overline{\Disl}(\overline{\Omega}))$ and let $a \in \Lip([\sigma,\tau])$ be injective. We define 
\begin{align*}
    a_{*}\boldsymbol{S} := (\overline{a}_{*}S^{b})_{b \in \Bcal} \in \BV(a([\sigma,\tau]);\overline{\Disl}(\overline{\Omega}))
\end{align*}
where $\overline{a}: [\sigma,\tau] \times \R^{3} \to \R^{1+3}$ is given by $\overline{a}(t,x) := (a(t),x)$ for $(t,x) \in [\sigma,\tau] \times \R^{3}$. First observe that 
\begin{align}\label{EndpointsRescale}
    \partial^{-}(a_{*}\boldsymbol{S}) = \left\{\begin{array}{cl} \partial^{-}\boldsymbol{S} & \text{if $a$ is increasing,} \\
    -\partial^{+}\boldsymbol{S} & \text{if $a$ is decreasing,}\end{array}\right. \hspace{1em}
    \partial^{+}(a_{*}\boldsymbol{S}) = \left\{\begin{array}{cl} \partial^{+}\boldsymbol{S} & \text{if $a$ is increasing,} \\
    -\partial^{-}\boldsymbol{S} & \text{if $a$ is decreasing.}\end{array}\right.
\end{align}

The following result demonstrates that the dissipation and the maximum core energy are `rate-independent' quantities:

\begin{proposition}\label{pp:DissMaxCoreRescale} Fix $\sigma,\tau \in \R$ with $\sigma < \tau$ and let $a: [\sigma,\tau] \to \R$ be bi-Lipschitz. Take $\boldsymbol{S} \in \BV([\sigma,\tau];\overline{\Disl}(\overline{\Omega}))$ and let $I \subset [\sigma,\tau]$ be Borel. Then
\begin{align}\label{DissMaxCoreRescale}
    \Diss(a_{*}\boldsymbol{S};a(I)) &= \Diss(\boldsymbol{S};I), & \Vert a_{*}\boldsymbol{S}\Vert_{\Lrm^{\infty}_{c}(a(I))} = \Vert \boldsymbol{S} \Vert_{\Lrm^{\infty}_{c}(I)}.
\end{align}
Furthermore, we have the following additional conclusions:
\begin{enumerate}
    \item Let $n \in \N$ and suppose that $\boldsymbol{S}\in \BV([\sigma,\tau];\Disl_{1/n}(\overline{\Omega}))$. Then
    \begin{align*}
        a_{*}\boldsymbol{S}\in \BV(a([\sigma,\tau]);\Disl_{1/n}(\overline{\Omega})).
    \end{align*}
    \item Suppose that $\boldsymbol{S} \in \overline{\Slip}(\partial^{-}\boldsymbol{S};[\sigma,\tau])$. Then 
    \begin{align*}
        a_{*}\boldsymbol{S} &\in \left\{\begin{array}{cl} 
    \overline{\Slip}(\partial^{-}\boldsymbol{S};[a(\sigma),a(\tau)]) & \text{if $a$ is increasing}, \\
    \overline{\Slip}(-\partial^{+}\boldsymbol{S};[a(\tau),a(\sigma)]) & \text{if $a$ is decreasing}.
    \end{array}\right.
\end{align*}
\end{enumerate}
\end{proposition}

\begin{proof} First notice that \eqref{DissMaxCoreRescale} follows from \cite[Proposition 2.7]{BonicattoRindlerTurnbull26-I?} and \cite[Proposition 3.6]{BonicattoRindlerTurnbull26}. The additional conclusion (1) follows from \cite[4.1.30]{Federer69book} and the additional conclusion (2) can be verified using \eqref{LipTrajConds}, \eqref{EndpointsRescale}, and \eqref{DissMaxCoreRescale}.
\end{proof}

It turns out that, after possibly rescaling, we can always ensure that the Lipschitz constant of a Lipschitz slip trajectory system is well-behaved:

\begin{lemma}\label{lm:Steadiness} Fix $\sigma,\tau \in \R$ with $\sigma < \tau$ and let $\boldsymbol{T} \in \overline{\Disl}(\overline{\Omega})$. Take $\boldsymbol{S} \in \overline{\Slip}(\boldsymbol{T};[\sigma,\tau])$. Then there is an increasing, bi-Lipschitz $a: [\sigma,\tau] \to [\sigma,\tau]$ such that
\begin{align*}
    \Lip(a_{*}\boldsymbol{S}) \leq C \cdot \Diss(\boldsymbol{S}).
\end{align*}
where $C > 0$ depends only on $\vert \tau - \sigma \vert$.
\end{lemma}

\begin{proof} The proof is similar to that of \cite[Lemma 4.2]{BonicattoRindlerTurnbull26}.
\end{proof}

We shall now describe various constructions of and operations on slip trajectory systems that will be useful in the sequel. For $i \in \{1,2\}$, fix $\sigma_{i},\tau_{i} \in \R$ with $\sigma_{i} < \tau_{i}$ and take $\boldsymbol{S}_{i} \in \BV([\sigma_{i},\tau_{i}];\overline{\Disl}(\overline{\Omega}))$. Suppose that $\partial^{+}\boldsymbol{S}_{1} = \partial^{-}\boldsymbol{S}_{2}$. We define the \textbf{concatenation} $\boldsymbol{S}_{2} \circ \boldsymbol{S}_{1}$ of $\boldsymbol{S}_{2}$ and $\boldsymbol{S}_{1}$ as follows:
\begin{align}\label{ConcatenationDefn}
    \boldsymbol{S}_{2} \circ \boldsymbol{S}_{1} := (a_{2})_{*}\boldsymbol{S}_{2} + (a_{1})_{*}\boldsymbol{S}_{1},
\end{align}
where, for $i \in \{1,2\}$, $a_{i}:[\sigma_{i},\tau_{i}] \to \R$ is given by
\begin{align*}
    a_{i}(t) := \frac{1}{2}\cdot\frac{t-\sigma_{i}}{\tau_{i}-\sigma_{i}} + \frac{i-1}{2} \hspace{1em} \text{for $t \in [\sigma_{i},\tau_{i}]$}.
\end{align*}
Notice that $\boldsymbol{S}_{2} \circ \boldsymbol{S}_{1} \in \BV([0,1];\overline{\Disl}(\overline{\Omega}))$ and that
\begin{align*}
    \partial^{-}(\boldsymbol{S}_{2} \circ \boldsymbol{S}_{1}) &= \partial^{-}\boldsymbol{S}_{1}, & 
    \partial^{+}(\boldsymbol{S}_{2} \circ \boldsymbol{S}_{1}) &= \partial^{+}\boldsymbol{S}_{2}.
\end{align*}
One can use Proposition \ref{pp:DissMaxCoreRescale} to establish the following properties of concatenations:
\begin{lemma}\label{lm:Concatenation} For $i \in \{1,2\}$, fix $\sigma_{i},\tau_{i} \in \R$ with $\sigma_{i} < \tau_{i}$ and take $\boldsymbol{S}_{i} \in \BV([\sigma_{i},\tau_{i}];\overline{\Disl}(\overline{\Omega}))$. Suppose that $\partial^{+}\boldsymbol{S}_{1} = \partial^{-}\boldsymbol{S}_{2}$. Then
\begin{align*}
    \Diss(\boldsymbol{S}_{2} \circ \boldsymbol{S}_{1}) &\leq \Diss(\boldsymbol{S}_{1}) + \Diss(\boldsymbol{S}_{2}), \\ \Vert \boldsymbol{S}_{2} \circ \boldsymbol{S}_{1}\Vert_{\Lrm^{\infty}_{c}} &= \max\{\Vert \boldsymbol{S}_{1}\Vert_{\Lrm^{\infty}_{c}},\Vert \boldsymbol{S}_{2}\Vert_{\Lrm^{\infty}_{c}}\}.
\end{align*}
Furthermore, if $\tbf_{\#}\Vert \boldsymbol{S}_{1} \Vert(\{\tau_{1}\}) = \tbf_{\#}\Vert \boldsymbol{S}_{2} \Vert(\{\sigma_{2}\}) = 0$, then
\begin{align*}
    \Diss(\boldsymbol{S}_{2} \circ \boldsymbol{S}_{1}) = \Diss(\boldsymbol{S}_{1}) + \Diss(\boldsymbol{S}_{2}).
\end{align*}
Moreover, we have the following additional conclusions:
\begin{enumerate}
    \item Let $n \in \N$ and suppose that $\boldsymbol{S}_{i} \in \BV([\sigma_{i},\tau_{i}];\Disl_{1/n}(\overline{\Omega}))$ for $i \in \{1,2\}$. Then 
    \begin{align*}
        \boldsymbol{S}_{2} \circ \boldsymbol{S}_{1} \in \BV([0,1];\Disl_{1/n}(\overline{\Omega})).
    \end{align*}
    \item For $i \in \{1,2\}$, take $\boldsymbol{T}_{i} \in \overline{\Disl}(\overline{\Omega})$ and suppose that $\boldsymbol{S}_{i} \in \overline{\Slip}(\boldsymbol{T}_{i};[\sigma_{i},\tau_{i}])$. Then 
    \begin{align*}
        \boldsymbol{S}_{2} \circ \boldsymbol{S}_{1} \in \overline{\Slip}(\boldsymbol{T}_{1}).
    \end{align*}
\end{enumerate}
\end{lemma}

Now fix $\sigma,\tau \in \R$ with $\sigma < \tau$ and let $\boldsymbol{S} \in \BV([\sigma,\tau];\overline{\Disl}(\overline{\Omega}))$. Let $a: [\sigma,\tau] \to \R$ where
\begin{align*}
    a(t) = \sigma + (\tau - t) \hspace{1em} \text{for $t \in [\sigma,\tau]$.}
\end{align*}
We define the \textbf{reversal} $\boldsymbol{S}^{-1}$ of $\boldsymbol{S}$ as follows:
\begin{align*}
    \boldsymbol{S}^{-1} := -a_{*}\boldsymbol{S} := (-\overline{a}_{*}S^{b})_{b \in \Bcal}.
\end{align*}
Here, as above, $\overline{a}: [\sigma,\tau] \times \R^{3} \to \R^{1+3}$ is given by $\overline{a}(t,x) := (a(t),x)$ for $(t,x) \in [\sigma,\tau] \times \R^{3}$. Notice that $\boldsymbol{S}^{-1} \in \BV([\sigma,\tau];\overline{\Disl}(\overline{\Omega}))$, and that
\begin{align*}
    \partial^{+}\boldsymbol{S}^{-1} &= \partial^{-}\boldsymbol{S}, & \partial^{-}\boldsymbol{S}^{-1} &= \partial^{+}\boldsymbol{S}.
\end{align*}
One can use Proposition \ref{pp:DissMaxCoreRescale} to establish the following properties of reversals:
\begin{lemma}\label{lm:Reversal} Fix $\sigma,\tau \in \R$ with $\sigma < \tau$ and take $\boldsymbol{S} \in \BV([\sigma,\tau];\overline{\Disl}(\overline{\Omega}))$. Then
\begin{align*}
    \Diss(\boldsymbol{S}^{-1}) &= \Diss(\boldsymbol{S}), & \Vert \boldsymbol{S}^{-1}\Vert_{\Lrm^{\infty}_{c}} &= \Vert \boldsymbol{S}\Vert_{\Lrm^{\infty}_{c}}.
\end{align*}
Furthermore, we have the following additional conclusions:
\begin{enumerate}
    \item Let $n \in \N$ and suppose that $\boldsymbol{S} \in \BV([\sigma,\tau];\Disl_{1/n}(\overline{\Omega}))$. Then 
    \begin{align*}
        \boldsymbol{S}^{-1} \in \BV([\sigma,\tau];\Disl_{1/n}(\overline{\Omega})).
    \end{align*}
    \item If $\boldsymbol{S} \in \overline{\Slip}(\partial^{-}\boldsymbol{S};[\sigma,\tau])$, then $\boldsymbol{S}^{-1} \in \overline{\Slip}(\partial^{+}\boldsymbol{S};[\sigma,\tau])$.
\end{enumerate}
\end{lemma}

Again, fix $\sigma,\tau \in \R$ with $\sigma < \tau$ and let $\boldsymbol{S} \in \BV([\sigma,\tau];\overline{\Disl}(\overline{\Omega}))$. Take $t_{1},t_{2} \in [\sigma,\tau]$ with $t_{1} < t_{2}$. We define the \textbf{restriction} $\boldsymbol{S} \mvert [t_{1},t_{2}]$ of $\boldsymbol{S}$ to $[t_{1},t_{2}]$ as follows:
\begin{align*}
    \boldsymbol{S} \mvert [t_{1},t_{2}] := (S^{b} \mvert [t_{1},t_{2}] \times \R^{3})_{b \in \Bcal}.
\end{align*}
Now suppose that
\begin{align*}
    t_{1},t_{2} \in \{\sigma,\tau\} \cup \{t \in [\sigma,\tau]: \tbf_{\#}\Vert \boldsymbol{S} \Vert(\{t\}) =0, \hspace{0.2em} \Wcal_{c}(\boldsymbol{S}(t)) < \infty\}.
\end{align*}
Then $\boldsymbol{S} \mvert [t_{1},t_{2}] \in \BV([t_{1},t_{2}];\overline{\Disl}(\overline{\Omega}))$ and one can use \cite[Proposition 3.4]{BonicattoRindlerTurnbull26} to deduce that
\begin{align}\label{RestrictionEndpoints}
\begin{aligned}
    \partial^{-}(\boldsymbol{S} \mvert [t_{1},t_{2}]) &= \left\{\begin{array}{cl} \partial^{-}\boldsymbol{S} & \text{if $t_{1}=\sigma$}, \\
    \boldsymbol{S}(t_{1}) & \text{otherwise}.\end{array}\right. \\
    \partial^{+}(\boldsymbol{S} \mvert [t_{1},t_{2}]) &= \left\{\begin{array}{cl} \partial^{+}\boldsymbol{S} & \text{if $t_{2}=\tau$}, \\
    \boldsymbol{S}(t_{2}) & \text{otherwise}.\end{array}\right. 
\end{aligned}
\end{align}
Furthermore, notice that
\begin{align*}
    \Diss(\boldsymbol{S} \mvert [t_{1},t_{2}]) &= \Diss(\boldsymbol{S};[t_{1},t_{2}]), & \Vert \boldsymbol{S} \mvert [t_{1},t_{2}] \Vert_{\Lrm^{\infty}_{c}} &= \Vert \boldsymbol{S} \Vert_{\Lrm^{\infty}_{c}([t_{1},t_{2}])}.
\end{align*}

Finally, take $\boldsymbol{T} \in \overline{\Disl}(\overline{\Omega})$. We define the \textbf{neutral slip trajectory system} $[\![\sigma,\tau]\!] \times \boldsymbol{T}$ starting from $\boldsymbol{T}$ on $[\sigma,\tau]$ as follows:
\begin{align*}
    [\![\sigma,\tau]\!] \times \boldsymbol{T} := ([\![\sigma,\tau]\!] \times T^{b})_{b \in \Bcal}.
\end{align*}
Notice that $[\![\sigma,\tau]\!] \times \boldsymbol{T} \in \overline{\Slip}(\boldsymbol{T};[\sigma,\tau])$ and that $\partial^{+}([\![\sigma,\tau]\!] \times \boldsymbol{T}) = \boldsymbol{T}$. Furthermore, it follows from the definition of the dissipation and the maximal core energy that
\begin{align}\label{Neutral}
    \Diss([\![\sigma,\tau]\!] \times \boldsymbol{T}) &=0, & \Vert [\![\sigma,\tau]\!] \times \boldsymbol{T} \Vert_{\Lrm^{\infty}_{c}} &= \Wcal_{c}(\boldsymbol{T}).
\end{align}

\subsection{Compactness and Continuity Results for Slip Trajectory Systems} 
Fix $\sigma,\tau \in \R$ with $\sigma < \tau$. Take $(\boldsymbol{S}_{j})_{j} \subset \BV([\sigma,\tau];\overline{\Disl}(\overline{\Omega}))$ and let $\boldsymbol{S} \in \BV([\sigma,\tau];\overline{\Disl}(\overline{\Omega}))$. We write
\begin{align*}
    \boldsymbol{S}_{j} \wkto \boldsymbol{S}
\end{align*}
if $S^{b}_{j} \wkto S^{b}$ for $b \in \Bcal$. 

Let $\varphi: \R \to [0,1]$ be a smooth cut-off function with $\varphi(t) = 1$ for $t \leq \sigma$ and $\varphi(t) = 0$ for $t \geq \tau$. It follows from \eqref{EndpointsDefn} that, if $\boldsymbol{S} \in \BV([\sigma,\tau];\overline{\Disl}(\overline{\Omega}))$, then
\begin{align*}
    \partial^{-}\boldsymbol{S} &= (-\pbf_{*}(\varphi\cdot \partial S^{b}))_{b \in \Bcal}, & \partial^{+}\boldsymbol{S} &= (\pbf_{*}((1-\varphi)\cdot \partial S^{b}))_{b \in \Bcal}.
\end{align*}
One can use this observation to prove the following continuity result:
\begin{lemma}\label{lm:EndpointsCty} Fix $\sigma,\tau \in \R$ with $\sigma < \tau$. Take  $(\boldsymbol{S}_{j})_{j} \subset \BV([\sigma,\tau];\overline{\Disl}(\overline{\Omega}))$, and suppose that $\boldsymbol{S}_{j} \wkto \boldsymbol{S}$ for some $\boldsymbol{S} \in \BV([\sigma,\tau];\overline{\Disl}(\overline{\Omega}))$. Then $\partial^{-}\boldsymbol{S}_{j} \wkto \partial^{-}\boldsymbol{S}$ and $\partial^{+}\boldsymbol{S}_{j} \wkto \partial^{+}\boldsymbol{S}$.
\end{lemma}

We shall now establish compactness results for slip trajectory systems and Lipschitz slip trajectory systems; these results are adaptations of \cite[Proposition 3.11]{BonicattoRindlerTurnbull26} and \cite[Theorem 3.7]{Rindler23}, but for the benefit of the reader we give the details here.

\begin{proposition}\label{pp:BVCompactness} Fix $\sigma,\tau \in \R$ with $\sigma < \tau$ and let $(\boldsymbol{S}_{j})_{j} \subset \BV([\sigma,\tau];\overline{\Disl}(\overline{\Omega}))$. Suppose that
\begin{align}\label{ass:BVCompactness}
    \sup_{j \in \N}\bigg(\Vert \boldsymbol{S}_{j} \Vert_{\Lrm^{\infty}_{c}} + \Diss(\boldsymbol{S}_{j}) + \Wcal_{c}(\partial^{-}\boldsymbol{S}_{j})+\Wcal_{c}(\partial^{+}\boldsymbol{S}_{j})\bigg) < \infty
\end{align}
Then there is $\boldsymbol{S} \in \BV([\sigma,\tau];\overline{\Disl}(\overline{\Omega}))$ and a (not relabelled) subsequence of $(\boldsymbol{S}_{j})_{j}$ such that
\begin{align}\label{cc:BVCompactness1}
    \boldsymbol{S}_{j} &\wkto \boldsymbol{S}, & \Diss(\boldsymbol{S}) &\leq \liminf_{j \to \infty} \Diss(\boldsymbol{S}_{j}),
\end{align}
and, for $\mathscr{L}^{1}$-a.e.\ $t \in [\sigma,\tau]$, we have
\begin{align}\label{cc:BVCompactness2}
    \boldsymbol{S}_{j} \mvert [\sigma,t] &\wkto \boldsymbol{S} \mvert [\sigma,t], & \Diss(\boldsymbol{S};[\sigma,t]) &\leq \liminf_{j \to \infty} \Diss(\boldsymbol{S}_{j};[\sigma,t]).
\end{align}
Furthermore, we have
\begin{align}\label{cc:BVCompactness3}
     \boldsymbol{S}_{j}(t) &\wkto \boldsymbol{S}(t) \hspace{1em} \text{for $\mathscr{L}^{1}$-a.e.\ $t \in [\sigma,\tau]$}, & \Vert\boldsymbol{S} \Vert_{\Lrm^{\infty}_{c}} &\leq \liminf_{j \to \infty} \Vert \boldsymbol{S}_{j} \Vert_{\Lrm^{\infty}_{c}}.
\end{align}
Moreover, if $(\boldsymbol{S}_{j})_{j} \subset \BV([\sigma,\tau];\Disl_{1/n}(\overline{\Omega}))$ for some $n \in \N$, then we can additionally conclude that $\boldsymbol{S} \in \BV([\sigma,\tau];\Disl_{1/n}(\overline{\Omega}))$.
\end{proposition}

\begin{proof} First notice that we can use \cite[Lemma 3.10]{BonicattoRindlerTurnbull26}, \eqref{MaxCoreVsLInfty}, and \eqref{DissVsVar} to yield
\begin{align}\label{BVCompactness1}
    \sum_{b \in \Bcal} \Mbf(S^{b}_{j}) &\leq C(\vert \tau - \sigma \vert \cdot \Vert \boldsymbol{S}_{j} \Vert_{\Lrm^{\infty}_{c}} + \Diss(\boldsymbol{S}_{j})).
\end{align}
Similarly, we can use \eqref{CoreVsMass} to deduce that
\begin{align*}
    \sum_{b \in \Bcal} \Mbf(\partial S^{b}_{j}) \leq C(\Wcal_{c}(\partial^{-}\boldsymbol{S}_{j}) + \Wcal_{c}(\partial^{+}\boldsymbol{S}_{j})).
\end{align*}
It then follows from \eqref{ass:BVCompactness} that there is $M > 0$ such that 
\begin{align}\label{BVCompactness2}
    \Mbf(S^{b}_{j}) + \Mbf(\partial S^{b}_{j}) \leq M \hspace{1em} \text{for $j \in \N$, $b \in \Bcal$}.
\end{align}
We can therefore find $\boldsymbol{S} \in \BV([\sigma,\tau];\overline{\Disl}(\overline{\Omega}))$ and a (not relabelled) subsequence of $(\boldsymbol{S}_{j})_{j}$ such that $\boldsymbol{S}_{j} \wkto \boldsymbol{S}$. Since \eqref{BVCompactness2} holds, we can use Reshetnyak's lower semicontinuity theorem (see, for example, \cite[Theorem 2.38]{AmbrosioFuscoPallara00book}) to deduce that \eqref{cc:BVCompactness1} holds.

After possibly passing to a further subsequence, we can assume without loss of generality that, for $b \in \Bcal$, there is $\mu^{b} \in \Mcal^{+}(\R^{4})$ with $\Vert S_{j}^{b} \Vert \wkto \mu^{b}$. Fix $t \in [\sigma,\tau]$ with 
\begin{align*}
    \tbf_{\#}\mu^{b}(\{t\}) &= 0 \hspace{1em} \text{for $b \in \Bcal$}.
\end{align*}
Then we can use \cite[Proposition 1.62(b)]{AmbrosioFuscoPallara00book} to deduce that $\tbf_{\#}\Vert \boldsymbol{S} \Vert(\{t\}) = 0$ and that
\begin{align*}
    S^{b}_{j} \mvert [\sigma,t] \times \R^{3} \wkto S^{b} \mvert [\sigma,t] \times \R^{3} \hspace{1em} \text{for $b \in \Bcal$}.
\end{align*}
We then appeal to Reshetnyak's lower semicontinuity theorem as above to deduce that \eqref{cc:BVCompactness2} holds. Furthermore, we have 
\begin{align*}
    \pbf_{*}\partial(S^{b}_{j} \mvert [\sigma,t] \times \R^{3}) \wkto \pbf_{*}\partial(S^{b} \mvert [\sigma,t] \times \R^{3}) \hspace{1em} \text{for $b \in \Bcal$}.
\end{align*}
and we can use Lemma \ref{lm:EndpointsCty} to deduce that $\partial^{-}\boldsymbol{S}_{j} \wkto \partial^{-}\boldsymbol{S}$. It then follows from \cite[Proposition 3.4]{BonicattoRindlerTurnbull26} that $\boldsymbol{S}_{j}(t) \wkto \boldsymbol{S}(t)$. We can then use \eqref{CoreLSC} to deduce that \eqref{cc:BVCompactness3} holds. To obtain the remaining conclusion, one can apply the Federer-Fleming closure theorem as in Section \ref{s:Disl}.
\end{proof}

\begin{proposition}\label{pp:SlipCompactness} Fix $\sigma,\tau \in \R$ with $\sigma < \tau$. Let $\boldsymbol{T} \in \overline{\Disl}(\overline{\Omega})$, take $(\boldsymbol{S}_{j})_{j} \subset \overline{\Slip}(\boldsymbol{T};[\sigma,\tau])$, and suppose that 
\begin{align}\label{ass:SlipCompactness}
    \sup_{j \in \N}\bigg(\Vert \boldsymbol{S}_{j} \Vert_{\Lrm^{\infty}_{c}} + \Diss(\boldsymbol{S}_{j})+\Lip(\boldsymbol{S}_{j})\bigg) < \infty
\end{align}
Then there is $\boldsymbol{S} \in \overline{\Slip}(\boldsymbol{T};[\sigma,\tau])$ and a (not relabelled) subsequence of $(\boldsymbol{S}_{j})_{j}$ such that \eqref{cc:BVCompactness1} holds, \eqref{cc:BVCompactness2}, \eqref{cc:BVCompactness3} hold for all $t \in [\sigma,\tau]$, and
\begin{align}\label{cc:SlipCompactness4}
    \Lip(\boldsymbol{S}) &\leq \liminf_{j \to \infty} \Lip(\boldsymbol{S}_{j}).
\end{align}
Moreover, if $(\boldsymbol{S}_{j})_{j} \subset \Slip_{1/n}(\boldsymbol{T};[\sigma,\tau])$ for some $n \in \N$, then we can additionally conclude that $\boldsymbol{S} \in \Slip_{1/n}(\boldsymbol{T};[\sigma,\tau])$.
\end{proposition}

\begin{proof} First notice that we can use Lemma \ref{lm:CoreEndpointsByMaxCore} to deduce that
\begin{align*}
    \Wcal_{c}(\partial^{-}\boldsymbol{S}_{j}) + \Wcal_{c}(\partial^{+}\boldsymbol{S}_{j}) &\leq 2\Vert\boldsymbol{S}_{j} \Vert_{\Lrm^{\infty}_{c}} \hspace{1em} \text{for $j \in \N$}.
\end{align*}
We can then work as in the proof of Proposition \ref{pp:BVCompactness}, using \eqref{ass:SlipCompactness} in place of \eqref{ass:BVCompactness}, to construct $\boldsymbol{S} \in \BV([\sigma,\tau];\overline{\Disl}(\overline{\Omega}))$ and a (not relabelled) subsequence of $(\boldsymbol{S}_{j})_{j}$ such that \eqref{cc:BVCompactness1} holds. Since $\partial^{-}\boldsymbol{S}_{j} = \boldsymbol{T}$ for $j \in \N$, we can use Lemma \ref{lm:EndpointsCty} to deduce that $\partial^{-}\boldsymbol{S} = \boldsymbol{T}$. For $b \in \Bcal$, take $M > 0$ and $\mu^{b} \in \Mcal^{+}(\R^{4})$ as in the proof of Proposition \ref{pp:BVCompactness}. Take $s,t \in [\sigma,\tau]$ with $s \leq t$. Then we can use \cite[Lemma 3.10]{BonicattoRindlerTurnbull26}, \eqref{MaxCoreVsLInfty}, and \eqref{DissVsVar} to obtain
\begin{align*}
    \tbf_{\#}\Vert \boldsymbol{S}_{j} \Vert([s,t]) \leq C(\vert t - s \vert \cdot \Vert \boldsymbol{S}_{j} \Vert_{\Lrm^{\infty}_{c}} + \Diss(\boldsymbol{S}_{j};[s,t]))
\end{align*}
It follows that
\begin{align*}
    \tbf_{\#}\Vert \boldsymbol{S}_{j} \Vert([s,t]) \leq C(\Vert \boldsymbol{S}_{j} \Vert_{\Lrm^{\infty}_{c}} + \Lip(\boldsymbol{S}_{j})) \cdot \vert t - s \vert
\end{align*}
and so we can use \eqref{ass:SlipCompactness} to deduce that $\tbf_{\#}\mu^{b}(\{t\}) = 0$ for all $b \in \Bcal$ and $t \in [\sigma,\tau]$. We can now work as in Proposition \ref{pp:BVCompactness} to deduce that \eqref{cc:BVCompactness2} and \eqref{cc:BVCompactness3} hold for all $t \in [\sigma,\tau]$. Again, take $s,t \in [\sigma,\tau]$ with $s \leq t$. Then we can use \cite[Proposition 1.62(b)]{AmbrosioFuscoPallara00book} to deduce that
\begin{align*}
    \boldsymbol{S}_{j} \mvert [s,t] \wkto \boldsymbol{S} \mvert [s,t].
\end{align*}
We can therefore appeal to Reshetnyak's lower semicontinuity theorem (as in the proof of Proposition \ref{pp:BVCompactness}) to deduce that
\begin{align}\label{SlipCompactness2}
    \Diss(\boldsymbol{S};[s,t]) \leq \liminf_{j \to \infty} \Diss(\boldsymbol{S}_{j};[s,t]) \leq \liminf_{j \to \infty} \Lip(\boldsymbol{S}_{j}) \cdot \vert t - s \vert.
\end{align}
Now \eqref{cc:SlipCompactness4} follows from \eqref{SlipCompactness2}. To obtain the remaining conclusion, one can apply the Federer-Fleming closure theorem as in Section \ref{s:Disl}.
\end{proof}

\subsection{Space-Time Connectivity Theorems}\label{s:STC} In the sequel, we shall often need to `connect' elements of a convergent sequence of dislocation systems to the limit using Lipschitz slip trajectory systems. In the case that the dislocation systems are discrete, this is facilitated by the following result, which is an adaptation of the \textit{Space-Time Connectivity Theorem} for integral currents (see \cite[Theorem 5.1]{Rindler23}, where it is called the `Equivalence Theorem') to this applied setting:

\begin{proposition}\label{pp:AppliedIntegralSTC} Fix $n \in \N$. Let $(\boldsymbol{T}_{j})_{j} \subset \Disl_{1/n}(\overline{\Omega})$ and suppose that there is $M > 0$ with $\Wcal_{c}(\boldsymbol{T}_{j}) \leq M$ for $j \in \N$. Take $\boldsymbol{T} \in \Disl_{1/n}(\overline{\Omega})$ and suppose that $\boldsymbol{T}_{j} \wkto \boldsymbol{T}$. Then, for all $j \in \N$ along a (not relabelled) subsequence, there is $\boldsymbol{S}_{j} \in \Slip_{1/n}(\boldsymbol{T}_{j})$ with $\partial^{+}\boldsymbol{S}_{j} = \boldsymbol{T}$ for $j \in \N$ and
\begin{align*}
    \Diss(\boldsymbol{S}_{j}) &\to 0, & \limsup_{j \to \infty}  \Vert \boldsymbol{S}_{j} \Vert_{\Lrm^{\infty}_{c}} &\leq C \cdot \limsup_{j \to \infty} \Wcal_{c}(\boldsymbol{T}_{j}).
\end{align*}
\end{proposition}

A similar result for general dislocation systems can be obtained using the Space-Time Connectivity Theorem for normal currents established in \cite[Theorem 5.4]{BonicattoRindlerTurnbull26} but, because we are conducting a homogenisation procedure rather than proving the existence of field energetic solutions directly, we do not need it. Instead, we shall need the following result:

\begin{proposition}\label{pp:AppliedNormalSTC} Let $(\boldsymbol{T}_{j})_{j} \subset \overline{\Disl}(\overline{\Omega})$ and suppose that there is $M > 0$ with $\Wcal_{c}(\boldsymbol{T}_{j}) \leq M$ for $j \in \N$. Take $\boldsymbol{T} \in \overline{\Disl}(\overline{\Omega})$ and suppose that $\boldsymbol{T}_{j} \wkto \boldsymbol{T}$. Then, for all $j \in \N$, there is $\boldsymbol{S}_{j} \in \overline{\Slip}(\boldsymbol{T}_{j})$ with $\partial^{+}\boldsymbol{S}_{j} = \boldsymbol{T}$ and $\Diss(\boldsymbol{S}_{j}) \to 0$.
\end{proposition}

Notice that, in Proposition \ref{pp:AppliedNormalSTC}, we obtain a slip trajectory system for \textit{every} element in the sequence $(\boldsymbol{T}_{j})_{j}$, but we do not have an $\Lrm^{\infty}$ estimate. 

We shall now work towards proving a more general version of Proposition \ref{pp:AppliedNormalSTC}. In what follows, we fix $k,d \in \N$ with $k \leq d$. To start, we show that, provided the boundary of a domain is well-behaved, contractibility is preserved under taking closures. To state the result, we need the notion of a $C^{0}$ domain (see \cite[Section 2.4]{BonicattoRindlerTurnbull26}).

\begin{lemma}\label{lm:ContractClosure} Let $A \subset \R^{d}$ be a $\Crm^{0}$ domain. If $A$ is contractible, then $\overline{A}$ is contractible. 
\end{lemma}

\begin{proof} Suppose that $A$ is contractible. Then there is a continuous map $F: [0,1] \times A \to A$ and $x_{0} \in A$ with $F(0,x)=x$ and $F(1,x)=x_{0}$ for $x \in A$. Since $A$ is a $C^{0}$ domain, we can apply \cite[Proposition 2.10]{BonicattoRindlerTurnbull26} to obtain a continuous $H_{A}: [0,1] \times \overline{A} \to \overline{A}$ with $H_{A}(0,x) = x$ and $H_{A}(1,x) \in A$ for $x \in \overline{A}$. Now define $H: [0,1] \times \overline{A} \to \overline{A}$ where
\begin{align*}
    H(t,x) &= \left\{\begin{array}{cc}
    H_{A}(2t,x) & \text{if $t \in [0,1/2]$} \\
    F(2t-1,H_{A}(1,x)) & \text{if $t \in [1/2,1]$}
\end{array}\right.
\end{align*}
for $(t,x) \in [0,1] \times \overline{A}$. One can check that $H$ is continuous; the result follows since $H(0,x)=x$ and $H(1,x)=x_{0}$ for $x \in \overline{A}$.
\end{proof}

Since $\Omega$ is a bounded Lipschitz domain, $\overline{\Omega}$ is a compact Lipschitz neighbourhood retract. Furthermore, since $\Omega$ is contractible and since Lipschitz domains are $\Crm^{0}$ domains, Lemma \ref{lm:ContractClosure} demonstrates that $\overline{\Omega}$ is contractible. To prove Proposition \ref{pp:AppliedNormalSTC}, it therefore suffices to establish the following result:

\begin{proposition}\label{pp:NormalContractibleSTC} Let $K \subset \R^{d}$ be a compact, contractible Lipschitz neighbourhood retract and take $M > 0$. Take $(T_{j})_{j} \subset \Nrm_{k}(K)$ with $\partial T_{j} = 0$ and $\Mbf(T_{j}) \leq M$ for $j \in \N$. Suppose that $T_{j} \wkto T$ for some $T \in \Nrm_{k}(K)$. Then, for all $j \in \N$, there is $S_{j} \in \NLip_{1+k}([0,1] \times K)$ with 
\begin{align*}
    \partial^{-}S_{j} &= T_{j}, & \partial^{+}S_{j} &= T, & \Var(S_{j}) &\to 0.
\end{align*}
\end{proposition}

To prove Proposition \ref{pp:NormalContractibleSTC}, we shall need the following lemma:

\begin{lemma}\label{lm:LipContract} Let $K \subset \R^{d}$ be a compact, contractible Lipschitz neighbourhood retract. Then there is a Lipschitz map $H: [0,1] \times K \to K$ and $x_{0} \in K$ with $H(0,x) = x$ and $H(1,x) = x_{0}$ for $x \in K$. 
\end{lemma}

\begin{proof} Since $K$ is a Lipschitz neighbourhood retract, there is an open subset $U \subset \R^{d}$ with $K \subset U$ and a Lipschitz retraction $r: U \to K$. Take $\varepsilon > 0$ with $K_{\varepsilon} \subset U$. Since $K$ is contractible, there is a continuous map $F: [0,1] \times K \to K$ and $x_{0} \in K$ with $F(0,x)=x$ and $F(1,x)=x_{0}$ for $x \in K$. Since $[0,1] \times K$ is compact, $F$ is uniformly continuous; let $\tilde{F}: [0,1] \times K \to \R^{d}$ be a Lipschitz function with $\vert\vert \tilde{F} - F \vert\vert_{\infty} < \varepsilon/2$ (one can construct such a function by applying the argument of \cite[Theorem 6.8]{Heinonen01} to $F$ component-wise). It follows that 
\begin{align}\label{LipContractEst1}
    \tilde{F}([0,1] \times K) &\subset U, & \tilde{F}(\{1\} \times K) &\subset B(x_{0},\varepsilon/2).
\end{align}
Now define $\tilde{H}_{1}, \tilde{H}_{2}: [0,1] \times K \to \R^{d}$ where
\begin{align*}
    \tilde{H}_{1}(t,x) &= (1-t)x+t\tilde{F}(0,x), & \tilde{H}_{2}(t,x) &= (1-t)\tilde{F}(1,x)+tx_{0}
\end{align*}
for $(t,x) \in [0,1] \times K$. Notice that
\begin{align*}
    \vert \tilde{H}_{1}(t,x)-x \vert = t \cdot \vert \tilde{F}(0,x)-x \vert < \varepsilon/2
\end{align*}
We therefore have 
\begin{align}\label{LipContractEst2}
    \tilde{H}_{1}([0,1] \times K) \subset K_{\varepsilon/2} \subset U.
\end{align}
On the other hand, since $B(x_{0},\varepsilon/2)$ is convex and \eqref{LipContractEst1} holds, we have 
\begin{align}\label{LipContractEst3}
\tilde{H}_{2}([0,1] \times K) \subset B(x_{0},\varepsilon/2) \subset K_{\varepsilon/2} \subset U.
\end{align}
Now define $\tilde{H}: [0,1] \times K \to \R^{d}$ where
\begin{align*}
    \tilde{H}(t,x) &= \left\{\begin{array}{cc}
    \tilde{H}_{1}(3t,x) & \text{if $t \in [0,1/3]$}, \\
    \tilde{F}(3t-1,x) & \text{if $t \in [1/3,2/3]$}, \\
    \tilde{H}_{2}(3t-2,x) & \text{if $t \in [2/3,1]$}
    \end{array}\right.
\end{align*}
for $(t,x) \in [0,1] \times K$. One can check that $\tilde{H}$ is Lipschitz and that $\tilde{H}(0,x) = x$ and $\tilde{H}(1,x) = x_{0}$ for $x \in K$. Furthermore, we can use \eqref{LipContractEst1}, \eqref{LipContractEst2} and \eqref{LipContractEst3} to deduce that $\tilde{H}([0,1] \times K) \subset U$. The result follows by setting $H := r \circ \tilde{H}$. 
\end{proof}

Let $K \subset \R^{d}$ be a compact, contractible Lipschitz neighbourhood retract. Before proceeding, we show how the Lipschitz null-homotopy $H$ constructed in Lemma \ref{lm:LipContract} can be used to connect any $T_{0},T_{1} \in \Nrm_{k}(K)$ with $\partial T_{0} = \partial T_{1} = 0$:

\begin{corollary}\label{cr:Coning} Let $K \subset \R^{d}$ be a compact, contractible Lipschitz neighbourhood retract. Take $T_{0},T_{1} \in \Nrm_{k}(K)$ with $\partial T_{1} = \partial T_{0} = 0$. Then there is $S \in \NLip_{1+k}([0,1] \times K)$ with 
\begin{align*}
    \partial^{-}S &= T_{0}, & \partial^{+}S &= T_{1}, & \Var(S) \leq C\cdot \Mbf(T_{1}-T_{0}),
\end{align*}
where $C > 0$ depends only on $K$. Furthermore, if $T_{0},T_{1} \in \Irm_{k}(K)$, then we can additionally conclude that $S \in \ILip_{1+k}([0,1] \times K)$. 
\end{corollary}

\begin{proof} Let $H: [0,1] \times K \to K$ be the Lipschitz map constructed in Lemma \ref{lm:LipContract}. Notice that
\begin{align*}
    \partial H_{*}([\![0,1]\!] \times (T_{0}-T_{1})) = T_{1}-T_{0}.
\end{align*}
and so
\begin{align*}
    \FN^{\circ}_{K}(T_{0} -T_{1}) \leq \Mbf(H_{*}([\![0,1]\!] \times (T_{0}-T_{1}))) \leq C \cdot \Mbf(T_{0}-T_{1})
\end{align*}
where $C > 0$ depends only on $K$. The result follows from \cite[Proposition 5.2]{BonicattoRindlerTurnbull26}; for the additional conclusion, notice that the integral analogue of \cite[Proposition 5.2]{BonicattoRindlerTurnbull26} holds (see \cite[Theorem 5.4]{Rindler23}).
\end{proof}

To prove Proposition \ref{pp:NormalContractibleSTC}, we must estimate more carefully:

\begin{lemma}\label{lm:HomFlatByFlat} Let $K \subset \R^{d}$ be a compact, contractible Lipschitz neighbourhood retract. Then, for all $T \in \Nrm_{k}(K)$ with $\partial T = 0$, we have
\begin{align*}
    \FN_{K}^{\circ}(T) \leq C \cdot \FN_{K}(T)
\end{align*}
where $C > 0$ depends only on $K$.
\end{lemma}

\begin{proof} Let $H: [0,1] \times K \to K$ be the Lipschitz map constructed in Lemma \ref{lm:LipContract}. We first use \cite[4.1.12]{Federer69book} to find $R \in \Nrm_{1+k}(K)$ with
\begin{align*}
    \FN_{K}(T) = \Mbf(T-\partial R) + \Mbf(R).
\end{align*}
Set $Y := T-\partial R$ and notice that
\begin{align*}
    \partial H_{*}([\![0,1]\!] \times Y) &= -Y, & \Mbf(H_{*}([\![0,1]\!] \times Y)) &\leq C \cdot \Mbf(Y),
\end{align*}
where $C > 0$ is a constant depending only on $K$. We can therefore estimate
\begin{align*}
    \FN_{K}^{\circ}(T) &\leq \FN_{K}^{\circ}(Y) + \FN_{K}^{\circ}(\partial R) \\
    &\leq \Mbf(H_{*}([\![0,1]\!] \times Y)) + \Mbf(R) \\
    &\leq C \cdot \Mbf(Y) + \Mbf(R) \\
    &\leq \max\{C,1\}\cdot \FN_{K}(T). \qedhere
\end{align*}
\end{proof}

\begin{proof}[Proof of Proposition \ref{pp:NormalContractibleSTC}] First notice that $\Fbf_{K}(Y) \leq \Mbf(Y)$ for $Y \in \Nrm_{k}(K)$. Since $T_{j} \wkto T$, we have $\FN_{K}(T_{j}-T) \to 0$ (see, for example, \cite[Theorem 7.1]{FedererFleming60}). We then appeal to Lemma \ref{lm:HomFlatByFlat} to deduce that $\FN_{K}^{\circ}(T_{j}-T) < \infty$ for $j \in \N$ and that $\FN_{K}^{\circ}(T_{j}-T) \to 0$. The result then follows from \cite[Proposition 5.2]{BonicattoRindlerTurnbull26}.
\end{proof}

\subsection{Approximation Results} We now establish several approximation results for dislocation and slip trajectory systems which are adapted from the results of \cite{BonicattoRindlerTurnbull26-I?}.

\begin{proposition}\label{pp:DislApprox} Let $\boldsymbol{T} \in \overline{\Disl}(\overline{\Omega})$. Then there is an $n: \N \to \N$ and a sequence $(\boldsymbol{T}_{j})_{j} \subset \overline{\Disl}(\overline{\Omega})$ such that $\boldsymbol{T}_{j} \in \Disl_{1/n(j)}(\overline{\Omega})$ for $j \in \N$ and
\begin{align*}
    \boldsymbol{T}_{j} &\wkto \boldsymbol{T}, & \Wcal_{c}(\boldsymbol{T}_{j}) &\to \Wcal_{c}(\boldsymbol{T}). 
\end{align*}
\end{proposition}

\begin{proof} After repeatedly applying \cite[Proposition 7.4]{BonicattoRindlerTurnbull26-I?} to the components of $\boldsymbol{T}$, we obtain $n: \N \to \N$ and $(\boldsymbol{T}_{j})_{j} \subset \overline{\Disl}(\overline{\Omega})$ with $\boldsymbol{T}_{j} \in \Disl_{1/n(j)}(\overline{\Omega})$ for $j \in \N$, $\boldsymbol{T}_{j} \wkto \boldsymbol{T}$, and  
\begin{align*}
    \limsup_{j \to \infty} \Mbf_{\psi^{b}}(T_{j}^{b}) &\leq \Mbf_{\psi^{b}}(T^{b}) \hspace{1em} \text{for $b \in \Bcal$}.
\end{align*}
We then appeal to \cite[Lemma 2.2]{BonicattoRindlerTurnbull26-I?} to deduce that
\begin{align*}
    \Mbf_{\psi^{b}}(T^{b}_{j}) \to \Mbf_{\psi^{b}}(T^{b}) \hspace{1em} \text{for $b \in \Bcal$}.
\end{align*}
The result follows.
\end{proof}

\begin{proposition}\label{pp:SlipRationalApprox} Fix $\sigma,\tau \in \R$ with $\sigma < \tau$ and let $\boldsymbol{S} \in \BV([\sigma,\tau];\overline{\Disl}(\overline{\Omega}))$. Then there is an $n: \N \to \N$ and a sequence $(\boldsymbol{S}_{j})_{j} \subset \BV([\sigma,\tau];\overline{\Disl}(\overline{\Omega}))$ such that
\begin{align}\label{cc:SlipRationalApprox1}
    \boldsymbol{S}_{j} \in \BV([\sigma,\tau];\Disl_{1/n(j)}(\overline{\Omega})) \hspace{1em} \text{for $j \in \N$}
\end{align}
and
\begin{align*}
    \boldsymbol{S}_{j} &\wkto \boldsymbol{S}, & \Diss(\boldsymbol{S}_{j}) &\to \Diss(\boldsymbol{S}), & \Wcal_{c}(\partial^{\pm}\boldsymbol{S}_{j}) &\to \Wcal_{c}(\partial^{\pm}\boldsymbol{S}).
\end{align*}
Furthermore, for $\mathscr{L}^{1}$-a.e.\ $t \in [\sigma,\tau]$, we have
\begin{align*}
    \Wcal_{c}(\boldsymbol{S}_{j}(t)) &\to \Wcal_{c}(\boldsymbol{S}(t)), & \Wcal_{c}(\boldsymbol{S}_{j}(t)) &\leq C \cdot \Vert \boldsymbol{S} \Vert_{\Lrm^{\infty}_{c}} + 1
\end{align*}
and
\begin{align*}
    \boldsymbol{S}_{j} \mvert [\sigma,t] &\wkto \boldsymbol{S} \mvert [\sigma,t], & \Diss(\boldsymbol{S}_{j};[\sigma,t]) &\to \Diss(\boldsymbol{S};[\sigma,t]).
\end{align*}
\end{proposition}

\begin{proof} After repeatedly applying \cite[Proposition 7.7]{BonicattoRindlerTurnbull26-I?} to the components of $\boldsymbol{S}$, we obtain $n: \N \to \N$ and $(\boldsymbol{S}_{j})_{j} \subset \BV([\sigma,\tau];\overline{\Disl}(\overline{\Omega}))$ satisfying \eqref{cc:SlipRationalApprox1} and the following conclusions:
\begin{enumerate}
    \item For $b \in \Bcal$, we have
    \begin{align*}
        S^{b}_{j} &\wkto S^{b}, & \limsup_{j \to \infty}\Var_{\phi^{b}}(S_{j}^{b}) &\leq \Var_{\phi^{b}}(S^{b}),
    \end{align*}
    and
    \begin{align*}
        \limsup_{j \to \infty} \int \Mbf_{\psi^{b}}(S_{j}^{b}(t)) \dd t &\leq \int \Mbf_{\psi^{b}}(S^{b}(t)) \dd t.
    \end{align*}
    \item We have
    \begin{align*}
        \limsup_{j \to \infty}\Mbf_{\psi^{b}}(\partial^{\pm}S_{j}^{b}) &\leq \Mbf_{\psi^{b}}(\partial^{\pm}S^{b}) \hspace{1em} \text{for $b \in \Bcal$}.
    \end{align*}
    \item We have
    \begin{align*}
        \limsup_{j \to \infty} \esssup_{t \in [\sigma,\tau]} \Mbf_{\psi^{b}}(S^{b}_{j}(t)) &\leq \esssup_{t \in [\sigma,\tau]}\Mbf_{\psi^{b}}(S^{b}(t))  \hspace{1em} \text{for $b \in \Bcal$}.
    \end{align*}
\end{enumerate}
Since $\boldsymbol{S}_{j} \wkto \boldsymbol{S}$, we can use Lemma \ref{lm:EndpointsCty} to deduce that $\partial^{\pm}\boldsymbol{S}_{j} \wkto \partial^{\pm}\boldsymbol{S}$. It then follows from \cite[Lemma 2.2]{BonicattoRindlerTurnbull26-I?} that
\begin{align*}
    \Mbf_{\psi^{b}}(\partial^{\pm}S_{j}^{b}) &\to \Mbf_{\psi^{b}}(\partial^{\pm}S^{b}) \hspace{1em} \text{for $b \in \Bcal$}.
\end{align*}
We now use \cite[Proposition 2.9]{BonicattoRindlerTurnbull26-I?} to deduce that, for $b \in \Bcal$,
\begin{align*}
    \Var_{\phi^{b}}(S_{j}^{b}) &\to \Var_{\phi^{b}}(S^{b}), & \int \vert\Mbf_{\psi^{b}}(S^{b}(t)) - \Mbf_{\psi^{b}}(S_{j}^{b}(t)) \vert \dd t &\to 0.
\end{align*}
After passing to a (not relabelled) subsequence, we can assume that, for $b \in \Bcal$,
\begin{align*}
    \Mbf_{\psi^{b}}(S^{b}_{j}(t)) &\to \Mbf_{\psi^{b}}(S^{b}(t)), & \hspace{1em} \text{for $\mathscr{L}^{1}$-a.e.\ $t \in [\sigma,\tau]$}
\end{align*} 
On the other hand, it follows from \cite[Proposition 2.13]{BonicattoRindlerTurnbull26-I?} that
\begin{align*}
    \Var_{\phi^{b}}(S^{b}_{j};\cdot) \wkto \Var_{\phi^{b}}(S^{b};\cdot) \hspace{1em} \text{for $b \in \Bcal$}
\end{align*}
and so, for $t \in [\sigma,\tau]$, we can use \cite[Proposition 1.62(b)]{AmbrosioFuscoPallara00book} to deduce that
\begin{align*}
    \Var_{\phi^{b}}(S^{b}_{j};[\sigma,t]) \to \Var_{\phi^{b}}(S^{b};[\sigma,t]) \hspace{1em} \text{whenever $\Var_{\phi^{b}}(S^{b};\{t\}) = 0$}.
\end{align*}
After possibly passing to a further (not relabelled) subsequence, we can assume that, for $b \in \Bcal$, there is $\mu^{b} \in \Mcal^{+}(\R^{4})$ with $\Vert S_{j}^{b} \Vert + \Vert \partial S_{j}^{b} \Vert \wkto \mu^{b}$. Fix $t \in [\sigma,\tau]$ with 
\begin{align*}
    \tbf_{\#}\mu^{b}(\{t\}) &= 0 \hspace{1em} \text{for $b \in \Bcal$}.
\end{align*}
It then follows from \cite[Proposition 1.62(b)]{AmbrosioFuscoPallara00book} that
\begin{align*}
    S^{b}_{j} \mvert [\sigma,t] \times \R^{3} \wkto S^{b} \mvert [\sigma,t] \times \R^{3} \hspace{1em} \text{for $b \in \Bcal$}.
\end{align*}
We now estimate
\begin{align*}
    \limsup_{j \to \infty} \Vert \boldsymbol{S}_{j} \Vert_{\Lrm^{\infty}_{c}} \leq \frac{1}{2} \sum_{b \in \Bcal} \esssup_{t \in [\sigma,\tau]} \Mbf_{\psi^{b}}(S^{b}(t)) \leq \vert \Bcal \vert \cdot \Vert \boldsymbol{S} \Vert_{\Lrm^{\infty}_{c}}
\end{align*}
After possibly passing to a further (not relabelled) subsequence, we can assume that
\begin{align*}
    \Vert \boldsymbol{S}_{j} \Vert_{\Lrm^{\infty}_{c}} \leq \vert \Bcal \vert \cdot \Vert \boldsymbol{S} \Vert_{\Lrm^{\infty}_{c}}+1.
\end{align*}
The result follows.
\end{proof}

\begin{proposition}\label{pp:SlipLipApprox} Fix $\sigma,\tau \in \R$ with $\sigma < \tau$ and let $\boldsymbol{S} \in \BV([\sigma,\tau];\overline{\Disl}(\overline{\Omega}))$. Then there is a sequence 
\begin{align*}
    (\boldsymbol{S}_{j})_{j} \subset \overline{\Slip}(\partial^{-}\boldsymbol{S};[\sigma,\tau])
\end{align*}
with $\partial^{+}\boldsymbol{S}_{j} = \partial^{+}\boldsymbol{S}$ for $j \in \N$ and
\begin{align*}
    \boldsymbol{S}_{j} &\wkto \boldsymbol{S}, & \Diss(\boldsymbol{S}_{j}) &\to \Diss(\boldsymbol{S}), & \int \vert \Wcal_{c}(\boldsymbol{S}_{j}(t)) - \Wcal_{c}(\boldsymbol{S}(t)) \vert \dd t \to 0.
\end{align*}
Furthermore, if $\boldsymbol{S} \in \BV([\sigma,\tau];\Disl_{1/n}(\overline{\Omega}))$ for some $n \in \N$, then we can additionally ensure that $(\boldsymbol{S}_{j})_{j} \subset \Slip_{1/n}(\partial^{-}\boldsymbol{S};[\sigma,\tau])$.
\end{proposition}

\begin{proof} After repeatedly applying \cite[Proposition 7.9]{BonicattoRindlerTurnbull26-I?} to the components of $\boldsymbol{S}$, we obtain $(\boldsymbol{S}_{j})_{j} \subset \overline{\Slip}(\partial^{-}\boldsymbol{S};[\sigma,\tau])$ with $\partial^{+}\boldsymbol{S}_{j} = \partial^{+}\boldsymbol{S}$ for $j \in \N$, $\boldsymbol{S}_{j} \wkto \boldsymbol{S}$, and
\begin{align*}
    \limsup_{j \to \infty} \Var_{\phi^{b}}(S^{b}_{j}) &\leq \Var_{\phi^{b}}(S^{b}), & \limsup_{j \to \infty} \int \Mbf_{\psi^{b}}(S^{b}_{j}(t)) \dd t &\leq \int \Mbf_{\psi^{b}}(S^{b}(t)) \dd t
\end{align*}
for $b \in \Bcal$. Furthermore, if $\boldsymbol{S} \in \BV([\sigma,\tau];\Disl_{1/n}(\overline{\Omega}))$, then we can appeal to the additional conclusion of \cite[Proposition 7.9]{BonicattoRindlerTurnbull26-I?} to ensure that $(\boldsymbol{S}_{j})_{j} \subset \Slip_{1/n}(\partial^{-}\boldsymbol{S};[\sigma,\tau])$. We can then appeal to \cite[Proposition 2.9]{BonicattoRindlerTurnbull26-I?} to deduce that
\begin{align*}
    \Var_{\phi^{b}}(S^{b}_{j}) &\to \Var_{\phi^{b}}(S^{b}), & \int \vert \Mbf_{\psi^{b}}(S^{b}_{j}(t)) - \Mbf_{\psi^{b}}(S^{b}(t)) \vert \dd t\to 0
\end{align*}
for $b \in \Bcal$. The result follows.
\end{proof}

Fix $\sigma, \tau \in \R$ with $\sigma < \tau$. In Proposition \ref{pp:SlipRationalApprox}, we approximated a slip trajectory system $\boldsymbol{S} \in \BV([\sigma,\tau];\overline{\Disl}(\overline{\Omega}))$ by a sequence $(\boldsymbol{S}_{j})_{j} \subset \BV([\sigma,\tau];\overline{\Disl}(\overline{\Omega}))$ of slip trajectory systems and constructed an $n: \N \to \N$ such that $\boldsymbol{S}_{j}$ was $n(j)$-discrete for $j \in \N$. It turns out that one can additionally prescribe the initial values of the approximants $(\boldsymbol{S}_{j})_{j}$ in the following sense: 

\begin{proposition}\label{pp:SlipPrescribedApprox} Fix $\sigma,\tau \in \R$ with $\sigma < \tau$ and let $n: \N \to \N$ be strictly increasing with the property that, for any $m, N \in \N$, there is $j \geq N$ with $m \vert n(j)$. Let $\boldsymbol{S} \in \BV([\sigma,\tau];\overline{\Disl}(\overline{\Omega}))$ and, for $j \in \N$, take $\boldsymbol{T}_{j} \in \Disl_{1/n(j)}(\overline{\Omega})$. Suppose that
\begin{align}\label{ass:SlipPrescribedApprox}
    \boldsymbol{T}_{j} &\wkto \partial^{-}\boldsymbol{S}, & \sup_{j \in \N} \Wcal_{c}(\boldsymbol{T}_{j}) &< \infty.
\end{align}
Then, for all $j \in \N$ along a (not relabelled) subsequence, there is 
\begin{align*}
    \boldsymbol{S}_{j} \in \Slip_{1/n(j)}(\boldsymbol{T}_{j};[\sigma,\tau])
\end{align*}
with
\begin{align*}
    \boldsymbol{S}_{j} &\wkto \boldsymbol{S}, & \Diss(\boldsymbol{S}_{j}) &\to \Diss(\boldsymbol{S}),
\end{align*}
and
\begin{align*}
    \int \vert \Wcal_{c}(\boldsymbol{S}_{j}(t)) - \Wcal_{c}(\boldsymbol{S}(t)) \vert \dd t &\to 0, & \Wcal_{c}(\partial^{+}\boldsymbol{S}_{j}) &\to \Wcal_{c}(\partial^{+}\boldsymbol{S}).
\end{align*}
\end{proposition}

Notice that the divisibility condition on $n: \N \to \N$ given in Proposition \ref{pp:SlipPrescribedApprox} is not restrictive; if $n: \N \to \N$ is strictly increasing, then replacing $n(j)$ with $j! \cdot n(j)$ for $j \in \N$ yields a function with the desired property.

To explain the role of this divisibility condition, take $\sigma,\tau \in \R$ with $\sigma < \tau$ and let $\boldsymbol{S} \in \BV([\sigma,\tau];\overline{\Disl}(\overline{\Omega}))$. Take $(\boldsymbol{T}_{j})_{j} \subset \overline{\Disl}(\overline{\Omega})$ and suppose that \eqref{ass:SlipPrescribedApprox} holds. For now, we shall only suppose that there is $n: \N \to \N$ such that $\boldsymbol{T}_{j} \in \Disl_{1/n(j)}(\overline{\Omega})$ for $j \in \N$. Let $\tilde{n}: \N \to \N$ and $(\tilde{\boldsymbol{S}}_{j})_{j} \subset \BV([\sigma,\tau];\overline{\Disl}(\overline{\Omega}))$ be the function and sequence constructed in Proposition \ref{pp:SlipRationalApprox} respectively; we can then use Lemma \ref{lm:EndpointsCty} to deduce that $\partial^{-}\tilde{\boldsymbol{S}}_{j} \wkto \partial^{-}\boldsymbol{S}$. Fix $b \in \Bcal$. It follows from \cite[Theorem 7.1]{FedererFleming60} and Lemma \ref{lm:HomFlatByFlat} that
\begin{align*}
    \Fbf^{\circ}_{\overline{\Omega}}(T_{j}^{b}-\partial^{-}\tilde{S}_{j}^{b}) &\to 0.
\end{align*}
Take $j \in \N$. If we could ensure that $n(j) \cdot \tilde{S}^{b}_{j} \in \Irm_{2}([\sigma,\tau] \times \overline{\Omega})$, then we could appeal to \cite[5.10]{Federer74} to find $R_{j}^{b} \in \Nrm_{2}(\overline{\Omega})$ with
\begin{align*}
    \Mbf(R^{b}_{j}) &= \FN^{\circ}_{\overline{\Omega}}(T^{b}_{j} -\partial^{-}\tilde{S}^{b}_{j}), & \partial R_{j}^{b} &= T^{b}_{j} -\partial^{-}\tilde{S}^{b}_{j}, & n(j) \cdot R_{j}^{b} &\in \Irm_{2}(\overline{\Omega}).
\end{align*}
It would then follow that
\begin{align*}
    n(j) \cdot (\tilde{S}^{b}_{j} - \delta_{\sigma} \times R^{b}_{j}) &\in  \Irm_{2}([\sigma,\tau] \times \overline{\Omega}), & \partial^{-}(\tilde{S}^{b}_{j} - \delta_{\sigma} \times R^{b}_{j}) &= T^{b}_{j}.
\end{align*}
One can then use results from \cite{BonicattoRindlerTurnbull26-I?} to obtain Proposition \ref{pp:SlipPrescribedApprox}. However, we have no control over $\tilde{n}$ and so we cannot ensure that $n(j) \cdot \tilde{S}^{b}_{j} \in \Irm_{2}([\sigma,\tau] \times \overline{\Omega})$. Nevertheless, if we assume that, for any $m,N \in \N$, there is $i \geq N$ with $m \vert n(i)$, then we can find $i \geq j$ with $\tilde{n}(j) \vert n(i)$. In particular, we have $n(i) \cdot \tilde{S}^{b}_{j} \in \Irm_{2}([\sigma,\tau] \times \overline{\Omega})$ and so we can proceed using the argument above but with $T^{b}_{i}$ in place of $T^{b}_{j}$. 

To prove Proposition \ref{pp:SlipPrescribedApprox}, we shall adopt a more general strategy than the approach presented above. Before proceeding, we recall that a compact set $K \subset \R^{3}$ has the pull-in property if it satisfies the conditions of \cite[Definition 7.1]{BonicattoRindlerTurnbull26-I?}.

\begin{proposition}\label{pp:STPrescribedApprox} Take positive convex integrands
\begin{align*}
    \psi &: \R^{3} \times \subext{1}\R^{3} \to [0,\infty), & \phi &: \R^{3} \times \subext{2}\R^{3} \to [0,\infty).
\end{align*}
Take a compact Lipschitz neighbourhood retract $K \subset \R^{3}$ with the pull-in property.  Let $n: \N \to \N$ be strictly increasing with the property that, for $m, N \in \N$, there is $j \geq N$ with $m \vert n(j)$. Take $S \in \Nrm_{2}([0,1] \times K)$ and suppose that $\partial S \mvert (0,1) \times \R^{3} = 0$. For $j \in \N$, take $T_{j} \in \Nrm_{1}(K)$ with $\partial T_{j} = 0$ and $n(j) \cdot T_{j} \in \Irm_{1}(K)$. Suppose that $\FN_{K}^{\circ}(T_{j} - \partial^{-}S) \to 0$. Take $\varepsilon > 0$ and $N \in \N$. Then there is $j \geq N$ and $V \in \NLip_{2}([0,1] \times K)$ with $\partial V \mvert (0,1) \times \R^{3} = 0$ and $n(j) \cdot V \in \ILip_{2}([0,1] \times K)$ such that $\partial^{-}V = T_{j}$,
\begin{align*}
    \FN_{[0,1] \times K}(V-S) &\leq \varepsilon, & \Var_{\phi}(V) &\leq \Var_{\phi}(S) + \varepsilon, 
\end{align*}
and
\begin{align*}    
    \int \Mbf_{\psi}(V(t)) \dd t &\leq \int \Mbf_{\psi}(S(t)) \dd t + \varepsilon, & \Mbf_{\psi}(\partial^{+}V) &\leq \Mbf_{\psi}(\partial^{+}S) + \varepsilon.
\end{align*}
\end{proposition}

\begin{proof} First notice that there is a constant $C_{\phi,K} > 0$ such that, for $x \in K$,
\begin{align*}
    C_{\phi,K}^{-1} \cdot \Vert \xi \Vert \leq \phi(x,\xi) \leq C_{\phi,K} \cdot \Vert \xi \Vert \hspace{1em} \text{for all $\xi \in \subext{2}\R^{3}$}.
\end{align*}
Furthermore, since $\subext{1}\R^{3}$ and $\subext{2}\R^{3}$ contain no non-simple vectors, $\phi$ and $\psi$ are flat integrands (see \cite[Section 3.1]{BonicattoRindlerTurnbull26-I?}). Fix $\alpha > 0$ (to be chosen later). Without loss of generality, we can ensure that 
\begin{align*}
    \FN^{\circ}_{K}(T_{i}-\partial^{-}S) \leq \alpha \hspace{1em} \text{for all $i \geq N$}.
\end{align*}
Fix $i \in \N$ with $i \geq N$ and take $R \in \Nrm_{2}(K)$ with $\partial R = T_{i} - \partial^{-}S$ and $\Mbf(R) \leq \alpha$. Set $\tilde{S} = S - \delta_{0} \times R$, so that $\partial^{-}\tilde{S} = T_{i}$ and $\partial^{+}\tilde{S} = \partial^{+}S$. We can then apply \cite[Proposition 7.8]{BonicattoRindlerTurnbull26-I?} to $\tilde{S}$ to obtain $\tilde{Q} \in \Nrm_{2}([0,1] \times K)$ and $m \in \N$ with $\partial \tilde{Q} \mvert (0,1) \times \R^{3} = 0$ and $m \cdot \tilde{Q} \in \Irm_{2}([0,1] \times K)$ such that $\partial^{-}\tilde{Q} = T_{i}$,
\begin{align*}
    \FN_{[0,1] \times K}(\tilde{Q}-\tilde{S}) &\leq \alpha, & \Var_{\phi}(\tilde{Q}) &\leq \Var_{\phi}(\tilde{S}) + \alpha
\end{align*}
and
\begin{align*}
    \int \Mbf_{\psi}(\tilde{Q}(t)) \dd t &\leq \int \Mbf_{\psi}(\tilde{S}(t)) \dd t + \alpha, & \Mbf_{\psi}(\partial^{+}\tilde{Q}) &\leq \Mbf_{\psi}(\partial^{+}\tilde{S}) + \alpha.
\end{align*}
Notice that, since $m \cdot \tilde{Q} \in \Irm_{2}([0,1] \times K)$ and $\partial^{-}\tilde{Q} = T_{i}$, we must have $m \cdot T_{i} \in \Irm_{1}(K)$. Take $j \geq N$ with $m \vert n(j)$. Then $n(j) \cdot (T_{j}-T_{i}) \in \Irm_{1}(K)$ and $\FN_{K}^{\circ}(T_{j}-T_{i}) \leq 2\alpha$. Since $K$ is a compact Lipschitz neighbourhood retract, we can use \cite[5.10]{Federer74} to deduce that
\begin{align*}
    \FI_{K}^{\circ}(n(j) \cdot (T_{j}-T_{i})) = \FN^{\circ}_{K}(n(j) \cdot (T_{j}-T_{i})) \leq 2n(j) \cdot \alpha. 
\end{align*}
We can therefore take $X \in \Nrm_{2}(K)$ with $\partial X = T_{j}-T_{i}$, $n(j) \cdot X \in \Irm_{2}(K)$, and $\Mbf(X) \leq 2 \alpha$. Now set $Q = \tilde{Q} - \delta_{0} \times X$, so that $n(j) \cdot Q \in \Irm_{2}([0,1] \times K)$, $\partial^{-}Q = T_{j}$, and $\partial^{+}Q = \partial^{+}\tilde{Q}$. We now apply \cite[Proposition 7.9]{BonicattoRindlerTurnbull26-I?} to $Q$ to obtain $V \in \NLip_{2}([0,1] \times K)$ with $\partial V \mvert (0,1) \times \R^{3} = 0$ and $n(j) \cdot V \in \ILip_{2}([0,1] \times K)$ such that $\partial V = \partial Q$,
\begin{align*}
    \Fbf_{[0,1] \times K}(V-Q) &\leq \alpha, & \Var_{\phi}(V) &\leq \Var_{\phi}(Q) + \alpha,
\end{align*}
and
\begin{align*}
    \int \Mbf_{\psi}(V(t)) \dd t \leq \int \Mbf_{\psi}(Q(t)) \dd t + \alpha.
\end{align*}
Combining the estimates above yields
\begin{align*}
    \FN_{[0,1] \times K}(V-S) &\leq 5\alpha, & \Var_{\phi}(V) &\leq \Var_{\phi}(S) + (3C_{\phi,K}+2)\alpha, 
\end{align*}
and
\begin{align*}
    \int \Mbf_{\psi}(V(t)) \dd t &\leq \int \Mbf_{\psi}(S(t)) \dd t + 2\alpha, & \Mbf_{\psi}(\partial^{+}V) &\leq \Mbf_{\psi}(\partial^{+}S) + \alpha.
\end{align*}
The result follows after choosing $\alpha$ so that 
\begin{align*}
    \alpha &\leq \varepsilon \cdot \min\bigg\{\frac{1}{5},\frac{1}{3C_{\phi,K}+2}\bigg\}. \qedhere
\end{align*}
\end{proof}

\begin{proof}[Proof of Proposition \ref{pp:SlipPrescribedApprox}] First notice that, since $\Omega \subset \R^{3}$ is a bounded Lipschitz domain, $\overline{\Omega}$ is a compact Lipschitz neighbourhood retract with the pull-in property. Take $b \in \Bcal$. It follows from \eqref{ass:SlipPrescribedApprox} and \eqref{CoreVsMass} that 
\begin{align*}
    \sup_{j \in \N} \Mbf(T^{b}_{j}) &< \infty, & T^{b}_{j} \wkto \partial^{-}S^{b}.
\end{align*}
We can therefore use \cite[Theorem 7.1]{FedererFleming60} to deduce that $\FN_{\overline{\Omega}}(T_{j}^{b}-\partial^{-}S^{b}) \to 0$ for $b \in \Bcal$. It then follows from Lemma \ref{lm:HomFlatByFlat} that 
\begin{align*}
    \FN_{\overline{\Omega}}^{\circ}(T_{j}^{b} - \partial^{-}S^{b}) \to 0.
\end{align*}
We now pass to a (not relabelled) subsequence with the property that, for $m \in \N$, there is $N \in \N$ such that $m \vert n(j)$ for all $j \geq N$. After repeatedly applying Proposition \ref{pp:STPrescribedApprox} to the components of $\boldsymbol{S}$, we obtain, for all $j \in \N$ along a further (not relabelled) subsequence, $\boldsymbol{S}_{j} \in \Slip_{1/n(j)}(\boldsymbol{T}_{j};[\sigma,\tau])$ such that $\boldsymbol{S}_{j} \wkto \boldsymbol{S}$ and, for $b \in \Bcal$,
\begin{align*}
    \limsup_{j \to \infty} \Var_{\phi^{b}}(S^{b}_{j}) &\leq \Var_{\phi^{b}}(S^{b}), & \limsup_{j \to \infty} \int \Mbf_{\psi^{b}}(S^{b}_{j}(t)) \dd t &\leq \int \Mbf_{\psi^{b}}(S^{b}(t)) \dd t
\end{align*}
and
\begin{align*}
    \limsup_{j \to \infty} \Mbf_{\psi^{b}}(\partial^{+}S^{b}_{j}) &\leq \Mbf_{\psi^{b}}(\partial^{+}S^{b}).
\end{align*}
It follows from \cite[Proposition 2.9]{BonicattoRindlerTurnbull26-I?} that 
\begin{align*}
    \Var_{\phi^{b}}(S^{b}_{j}) &\to \Var_{\phi^{b}}(S^{b}), & \int \vert \Mbf_{\psi^{b}}(S^{b}_{j}(t)) - \Mbf_{\psi^{b}}(S^{b}(t)) \vert \dd t\to 0
\end{align*}
for $b \in \Bcal$. Since $\boldsymbol{S}_{j} \wkto \boldsymbol{S}$, we can use Lemma \ref{lm:EndpointsCty} to deduce that $\partial^{+}\boldsymbol{S}_{j} \wkto \partial^{+}\boldsymbol{S}$. It then follows from \cite[Lemma 2.2]{BonicattoRindlerTurnbull26-I?} that
\begin{align*}
    \Mbf_{\psi^{b}}(\partial^{+}S^{b}_{j}) \to \Mbf_{\psi^{b}}(\partial^{+}S^{b}) \hspace{1em} \text{for $b \in \Bcal$}.
\end{align*}
The result follows.
\end{proof}

\section{Displacement, Distortion, and State Spaces}\label{s:DispDistQ}

\subsection{The Plastic Distortion}\label{s:PlasticDistortion} We say that $p \in \Mcal(\overline{\Omega};\R^{3 \times 3})$ is a \textbf{plastic distortion} if $\curl \hspace{0.2em} p \in \Mcal(\overline{\Omega};\R^{3 \times 3})$, and we set
\begin{align*}
    \Mcal_{\mathrm{curl}}(\overline{\Omega};\R^{3 \times 3}) := \{p \in \Mcal(\overline{\Omega};\R^{3 \times 3}): \curl_{\R^{3}} \hspace{0.2em} p \in \Mcal(\overline{\Omega};\R^{3 \times 3})\}.
\end{align*}
Here and in what follows, we use $\curl_{\R^{3}}$ and $\curl_{\Omega}$ to denote the distributional curl on $\R^{3}$ and $\Omega$ respectively; we also make the abbreviation $\curl := \curl_{\R^{3}}$. Take $p \in \Mcal_{\mathrm{curl}}(\overline{\Omega};\R^{3 \times 3})$ and let $\boldsymbol{T} \in \overline{\Disl}(\overline{\Omega})$. In this paper, our guiding assumption is that the plastic component $p$ of the distortion is determined entirely by the presence of dislocations in the following sense: We say that the pair $(p,\boldsymbol{T})$ is \textbf{consistent} if 
\begin{align}\label{Consistency}
    \curl \hspace{0.2em} p = \frac{1}{2} \sum_{b \in \Bcal} b \otimes T^{b}.
\end{align}
In this case, notice that
\begin{align}\label{CurlDistByCore}
    \Vert \curl \hspace{0.2em} p \Vert(\overline{\Omega}) \leq C\cdot\Wcal_{c}(\boldsymbol{T}).
\end{align}

Fix $\sigma,\tau \in \R$ with $\sigma < \tau$ and let $\boldsymbol{S} \in \BV([\sigma,\tau];\overline{\Disl}(\overline{\Omega}))$. Given $p \in \Mcal(\overline{\Omega};\R^{3 \times 3})$, we define the \textbf{plastic distortion path} (corresponding to $\boldsymbol{S}$) as $p_{\boldsymbol{S}}: [\sigma,\tau] \to \Mcal(\overline{\Omega};\R^{3 \times 3})$ where
\begin{align*}
    p_{\boldsymbol{S}}(t) := p + \frac{1}{2}\sum_{b \in \Bcal}b \otimes \star\pbf_{*}(S^{b} \mvert [\sigma,t] \times \R^{3}) \hspace{1em} \text{for $t \in [\sigma,\tau]$.}
\end{align*}
Here, and in what follows, we identify $\subext{1}\R^{3}$ with $\R^{3}$ so that $p_{\boldsymbol{S}}$ is well-defined. We then define the \textbf{plastic forward operator} by setting
\begin{align}\label{PlasticForward}
    \boldsymbol{S}_{\gg}p := p_{\boldsymbol{S}}(\tau).
\end{align}

We shall now establish some regularity properties for plastic distortion paths.
\begin{lemma}\label{lm:PlasticPathBV} Fix $\sigma,\tau \in \R$ with $\sigma < \tau$. Let $\boldsymbol{S} \in \BV([\sigma,\tau];\overline{\Disl}(\overline{\Omega}))$ and take $p \in \Mcal(\overline{\Omega};\R^{3 \times 3})$. Then, for $s,t \in [\sigma,\tau]$ with $s < t$, we have
\begin{align}\label{cc:PlasticPathBV1}
    \Vert p_{\boldsymbol{S}}(t)-p_{\boldsymbol{S}}(s) \Vert(\overline{\Omega}) \leq C \cdot \Diss(\boldsymbol{S};(s,t])
\end{align}
Moreover,
\begin{align}\label{cc:PlasticPathBV2}
    \Var(p_{\boldsymbol{S}};[\sigma,\tau]) \leq C \cdot \Diss(\boldsymbol{S};[\sigma,\tau]).
\end{align}
\end{lemma}

\begin{proof} Fix $s,t \in [\sigma,\tau]$ with $s < t$. Then we can use \cite[(3.24)]{BonicattoRindlerTurnbull26} to obtain
\begin{align*}
    \Vert p_{\boldsymbol{S}}(t)-p_{\boldsymbol{S}}(s)\Vert(\overline{\Omega}) &\leq \frac{1}{2} \sum_{b \in \Bcal} \Vert b \otimes \star\pbf_{*}(S^{b} \mvert (s,t] \times \R^{3})\Vert(\overline{\Omega}) \\
    &\leq C \sum_{b \in \Bcal} \Mbf(\pbf_{*}(S^{b} \mvert (s,t] \times \R^{3})) \\
    &\leq C \sum_{b \in \Bcal} \Var(S^{b};(s,t]).
\end{align*}
The first estimate \eqref{cc:PlasticPathBV1} now follows from \eqref{DissVsVar}. To deduce \eqref{cc:PlasticPathBV2}, we sum the estimates obtained by applying \eqref{cc:PlasticPathBV1} to each interval of an arbitrary partition of $[\sigma,\tau]$.
\end{proof}

\begin{lemma}\label{lm:FlowConsistent} Fix $\sigma,\tau \in \R$ with $\sigma < \tau$. Let $\boldsymbol{S} \in \BV([\sigma,\tau];\overline{\Disl}(\overline{\Omega}))$ and take $p \in \Mcal_{\mathrm{curl}}(\overline{\Omega};\R^{3 \times 3})$. Suppose that $(p,\partial^{-}\boldsymbol{S})$ is consistent. Then $(\boldsymbol{S}_{\gg}p,\partial^{+}\boldsymbol{S})$ is consistent.
\end{lemma}

\begin{proof} We take the curl of \eqref{PlasticForward} and use the relationships $\curl = \partial\star$ (see \cite[Section 7]{HudsonRindler22}) and $\star^{-1} = \star$ (see \cite[1.7.8]{Federer69book}) to deduce that
\begin{align*}
    \curl(\boldsymbol{S}_{\gg}p) &= \curl \hspace{0.2em} p + \frac{1}{2} \sum_{b \in \Bcal} b \otimes (\partial\star\star (\pbf_{*}S^{b})) \\
    &= \curl \hspace{0.2em} p + \frac{1}{2} \sum_{b \in \Bcal} b \otimes \pbf_{*}\partial S^{b} \\
    &= \curl \hspace{0.2em} p + \frac{1}{2} \sum_{b \in \Bcal} b \otimes (\partial^{+}S^{b} - \partial^{-}S^{b}).
\end{align*}
The result follows.
\end{proof}

After identifying $\Mcal_{\mathrm{curl}}(\overline{\Omega};\R^{3 \times 3})$ with the graph 
\begin{align*}
    \{(p,\curl \hspace{0.2em} p): p \in \Mcal_{\curl}(\overline{\Omega};\R^{3 \times 3})\} \subset \Mcal(\overline{\Omega};\R^{3 \times 3}) \times \Mcal(\overline{\Omega};\R^{3 \times 3}),
\end{align*}
we see that $\Mcal_{\mathrm{curl}}(\overline{\Omega};\R^{3 \times 3})$ naturally inherits a weak$^{*}$ topology which we shall denote by $\Tcal^{*}_{\mathrm{curl}}$ to avoid confusion with the weak$^{*}$ topology on $\Mcal(\overline{\Omega};\R^{3 \times 3})$. Notice that, if $(p_{j})_{j} \subset \Mcal_{\mathrm{curl}}(\overline{\Omega};\R^{3 \times 3})$ and $p \in \Mcal_{\mathrm{curl}}(\overline{\Omega};\R^{3 \times 3})$, then
\begin{align*}
    p_{j} \to p \hspace{1em} \text{in $\Tcal^{*}_{\mathrm{curl}}$} \hspace{1em} &\Leftrightarrow \hspace{1em} \text{$p_{j} \wkto p$ and $\curl \hspace{0.2em} p_{j} \wkto \curl \hspace{0.2em} p$ in $\Mcal(\overline{\Omega};\R^{3 \times 3})$}.
\end{align*}
Our final regularity result concerns the measurability of the plastic distortion path with respect to this topology:

\begin{lemma}\label{lm:PlasticPathMeasurable} Fix $\sigma,\tau \in \R$ with $\sigma < \tau$. Take $\boldsymbol{S} \in \BV([\sigma,\tau];\overline{\Disl}(\overline{\Omega}))$, let $M > 0$ and set
\begin{align*}
    I := \{t \in (\sigma,\tau): \Mbf(S^{b}(t)) \leq M \hspace{0.2em} \text{for $b \in \Bcal$}, \hspace{0.2em} \tbf_{\#}\Vert \boldsymbol{S} \Vert(\{t\}) = 0\}.
\end{align*}
Take $p \in \Mcal_{\mathrm{curl}}(\overline{\Omega};\R^{3 \times 3})$ and suppose that $(p,\partial^{-}\boldsymbol{S})$ is consistent. Then 
\begin{align*}
    p_{\boldsymbol{S}}(I) \subset \Mcal_{\mathrm{curl}}(\overline{\Omega};\R^{3 \times 3})
\end{align*}
and $p_{\boldsymbol{S}}\vert_{I}$ is $\mathscr{L}^{1}$-measurable with respect to the topology $\Tcal^{*}_{\mathrm{curl}}$ on $\Mcal_{\mathrm{curl}}(\overline{\Omega};\R^{3 \times 3})$.
\end{lemma}

\begin{proof} It follows immediately from Lemma \ref{lm:PlasticPathBV} that $p_{\boldsymbol{S}}: [\sigma,\tau] \to \Mcal(\overline{\Omega};\R^{3 \times 3})$ is weakly$^{*}$ $\mathscr{L}^{1}$-measurable. Take $t \in I$. Then $\boldsymbol{S} \mvert [\sigma,t] \in \BV([\sigma,t];\overline{\Disl}(\overline{\Omega}))$ and it follows from \eqref{PlasticForward} that $p_{\boldsymbol{S}}(t) = (\boldsymbol{S} \mvert [\sigma,t])_{\gg}p$. We then use \eqref{RestrictionEndpoints} to deduce that $\partial^{+}(\boldsymbol{S} \mvert [\sigma,t]) = \boldsymbol{S}(t)$. It then follows from Lemma \ref{lm:FlowConsistent} that
\begin{align}\label{PlasticPathMeasurable1}
    \curl \hspace{0.2em} p_{\boldsymbol{S}}(t) = \frac{1}{2} \sum_{b \in \Bcal} b \otimes S^{b}(t).
\end{align}
Fix $b \in \Bcal$. It follows from \cite[Proposition 5.2]{BonicattoDelNinRindler25} that the map $S^{b}(\cdot)\vert_{I}$ is $\mathscr{L}^{1}$-measurable with respect to $\Nrm_{1}(\overline{\Omega})$ equipped with the weak$^{*}$ topology. Since $\Mbf(S^{b}(t)) \leq M$ for $t \in I$, the map $S^{b}(\cdot)\vert_{I}$ is also weakly$^{*}$ $\mathscr{L}^{1}$-measurable (on $\Mcal(\overline{\Omega};\R^{3})$). The result follows from \eqref{PlasticPathMeasurable1}.
\end{proof}

Our next result establishes a `rate-independence' property for the plastic distortion path.

\begin{lemma}\label{lm:PlasticRescale} Fix $\sigma,\tau \in \R$ with $\sigma < \tau$ and let $a \in \Lip([\sigma,\tau])$ be strictly increasing. Let $\boldsymbol{S} \in \BV([\sigma,\tau];\overline{\Disl}(\overline{\Omega}))$ and take $p \in \Mcal(\overline{\Omega};\R^{3 \times 3})$. Then, for $t \in [\sigma,\tau]$,
\begin{align*}
    p_{a_{*}\boldsymbol{S}}(a(t)) = p_{\boldsymbol{S}}(t).
\end{align*}
In particular, $(a_{*}\boldsymbol{S})_{\gg}p = \boldsymbol{S}_{\gg}p.$
\end{lemma}

\begin{proof} Fix $t \in [\sigma,\tau]$. We compute as follows:
\begin{align*}
    p_{a_{*}\boldsymbol{S}}(a(t)) &= p + \frac{1}{2}\sum_{b \in \Bcal} b \otimes \star\pbf_{*}((a_{*}S^{b})\mvert [a(\sigma),a(t)] \times \R^{3}) \\
    &= p + \frac{1}{2}\sum_{b \in \Bcal} b \otimes \star\pbf_{*}a_{*}(S^{b}\mvert [\sigma,t] \times \R^{3}) \\
    &= p + \frac{1}{2}\sum_{b \in \Bcal} b \otimes \star\pbf_{*}(S^{b}\mvert [\sigma,t] \times \R^{3}).
\end{align*}
The result follows.
\end{proof}

The following result is an immediate consequence of Lemma \ref{lm:PlasticRescale}:

\begin{lemma}\label{lm:PlasticConcatenation} For $i \in \{1,2\}$, fix $\sigma_{i},\tau_{i} \in \R$ with $\sigma_{i} < \tau_{i}$ and take $\boldsymbol{S}_{i} \in \BV([\sigma_{i},\tau_{i}];\overline{\Disl}(\overline{\Omega}))$. Suppose that $\partial^{+}\boldsymbol{S}_{1} = \partial^{-}\boldsymbol{S}_{2}$. Let $p \in \Mcal(\overline{\Omega};\R^{3 \times 3})$. Then
\begin{align*}
    (\boldsymbol{S}_{2} \circ \boldsymbol{S}_{1})_{\gg}p = (\boldsymbol{S}_{2})_{\gg}((\boldsymbol{S}_{1})_{\gg}p).
\end{align*}
\end{lemma}

Similarly, we have
\begin{lemma}\label{lm:PlasticReversal} Fix $\sigma,\tau \in \R$ with $\sigma < \tau$. Let $\boldsymbol{S} \in \BV([\sigma,\tau];\overline{\Disl}(\overline{\Omega}))$ and take $p \in \Mcal(\overline{\Omega};\R^{3 \times 3})$. Then
\begin{align*}
    (\boldsymbol{S}^{-1}\circ \boldsymbol{S})_{\gg}p = p.
\end{align*}
\end{lemma}

The next two results establish continuity properties for the plastic forward operator:

\begin{lemma}\label{lm:pCtyVanishingVar} Fix $\sigma,\tau \in \R$ with $\sigma < \tau$. Take $(p_{j})_{j} \subset \Mcal(\overline{\Omega};\R^{3 \times 3})$ and suppose that $p_{j} \wkto p$ for some $p \in \Mcal(\overline{\Omega};\R^{3 \times 3})$. Take $(\boldsymbol{S}_{j})_{j} \subset \BV([\sigma,\tau];\overline{\Disl}(\overline{\Omega}))$ and suppose that $\Diss(\boldsymbol{S}_{j}) \to 0$. Then
\begin{align*}
    (\boldsymbol{S}_{j})_{\gg}p_{j} \wkto p.
\end{align*}
\end{lemma}

\begin{lemma}\label{lm:pCty} Fix $\sigma,\tau \in \R$ with $\sigma < \tau$. Take $(p_{j})_{j} \subset \Mcal(\overline{\Omega};\R^{3 \times 3})$ and suppose that $p_{j} \wkto p$ for some $p \in \Mcal(\overline{\Omega};\R^{3 \times 3})$. Take a sequence $(\boldsymbol{S}_{j})_{j} \subset \BV([\sigma,\tau];\overline{\Disl}(\overline{\Omega}))$ and suppose that $\boldsymbol{S}_{j} \wkto \boldsymbol{S}$ for some $\boldsymbol{S} \in \BV([\sigma,\tau];\overline{\Disl}(\overline{\Omega}))$. Suppose that 
\begin{align*}
    \sup_{j \in \N}\bigg(\int \Wcal_{c}(\boldsymbol{S}_{j}(t)) \dd t + \Diss(\boldsymbol{S}_{j})\bigg) < \infty.
\end{align*}
Then
\begin{align*}
    (\boldsymbol{S}_{j})_{\gg}p_{j} \wkto \boldsymbol{S}_{\gg}p.
\end{align*}
\end{lemma}

\subsection{The Free Distortion} For a plastic distortion $p \in \Mcal_{\mathrm{curl}}(\overline{\Omega};\R^{3 \times 3})$, recall from \cite[Section 9.1]{Rindler26?} (see also \cite[Section 6.2]{Rindler26?}) that the \textbf{geometrically-necessary distortion} $\beta: \Omega \to \R^{3 \times 3}$ associated with $p$ is given by the solution to the following linear PDE system
\begin{align}\label{betaPDE}
        \left\{\begin{aligned}
        -\div\Ebb\beta &= 0 &&\text{in $\Omega$}, \\
        \curl_{\Omega} \hspace{0.2em} \beta &= (\curl_{\R^{3}}\hspace{0.2em}p) \mvert \Omega &&\text{in $\Omega$}, \\
        (\Ebb\beta)n &= 0 &&\text{on $\partial\Omega$},
        \end{aligned}\right. && \skew\hspace{0.2em}[\beta]_{\Omega} &= 0.
\end{align} 

\begin{proposition}\label{pp:betaExist} Take $p \in \Mcal_{\mathrm{curl}}(\overline{\Omega};\R^{3 \times 3})$. Then there is a unique $\beta \in \Lrm^{3/2}(\Omega;\R^{3 \times 3})$ solving \eqref{betaPDE}. Furthermore, we have 
\begin{align}\label{betaEst}
    \Vert \beta \Vert_{\Lrm^{3/2}} \leq C \Vert \curl \hspace{0.2em} p \Vert(\overline{\Omega})
\end{align}
where $C > 0$ depends only on $\Omega$ and $\Ebb$.  
\end{proposition}

\begin{proof} Since $\curl\hspace{0.2em}p \in \Mcal(\overline{\Omega};\R^{3 \times 3})$ and $\div \curl \hspace{0.2em}p=0$, the result follows from~\cite[Theorem 4.1]{GarroniMalusa21}.
\end{proof}

For $p \in \Mcal_{\mathrm{curl}}(\overline{\Omega};\R^{3 \times 3})$, we shall use $\beta[\curl\hspace{0.2em}p] \in \Lrm^{3/2}(\Omega;\R^{3 \times 3})$ to denote the unique $\beta \in \Lrm^{3/2}(\Omega;\R^{3 \times 3})$ constructed in Proposition \ref{pp:betaExist}. We remark in passing that Proposition \ref{pp:betaExist} is the only result in Sections \ref{s:DislSlip} and \ref{s:DispDistQ} that requires $\Omega$ to have a $\Crm^{1}$ boundary; all other preliminary results remain valid if $\Omega$ has only a Lipschitz boundary.

Let $(p_{j})_{j} \subset \Mcal_{\mathrm{curl}}(\overline{\Omega};\R^{3 \times 3})$ and take $p \in \Mcal_{\mathrm{curl}}(\overline{\Omega};\R^{3 \times 3})$. One can use the uniqueness of solutions to \eqref{betaPDE} to deduce that
\begin{align}\label{betaCty}
    \text{$\curl \hspace{0.2em} p_{j} \wkto \curl \hspace{0.2em} p$ in $\Mcal(\overline{\Omega};\R^{3 \times 3})$} \hspace{0.5em} \Rightarrow \hspace{0.5em} \text{$\beta[\curl \hspace{0.2em} p_{j}] \rightharpoonup \beta[\curl \hspace{0.2em} p]$ in $\Lrm^{3/2}(\overline{\Omega};\R^{3 \times 3})$}.
\end{align}

A \textbf{displacement} of the crystal specimen $\Omega$ is described by a vector field $u \in \BV(\Omega;\R^{3})$. Given a displacement $u \in \BV(\Omega;\R^{3})$ and a plastic distortion $p \in \Mcal_{\curl}(\overline{\Omega};\R^{3 \times 3})$, we then define the \textbf{free (non-dislocation) distortion} corresponding to $Du$ and $p$ as 
\begin{align*}
    \fr[Du - p] := Du - p \mvert \Omega + \beta[\curl \hspace{0.2em} p] \in \Mcal(\Omega;\R^{3 \times 3}).
\end{align*}
It follows from \eqref{betaPDE} and the relation $(\curl_{\R^{3}}\hspace{0.2em}p)\mvert \Omega = \curl_{\Omega}\hspace{0.2em}p$ that
\begin{align*}
    \curl_{\Omega}\hspace{0.2em}\fr[Du-p] = 0 \hspace{1em} \text{on $\Omega$}.
\end{align*}
Since $\Omega$ is contractible, we can find $w \in \BV(\Omega;\R^{3})$ with 
\begin{align}\label{FreeDistBVDerivative}
    Dw = \fr[Du-p].
\end{align}

Let $u \in \BV(\Omega;\R^{3})$, take $p \in \Mcal_{\mathrm{curl}}(\overline{\Omega};\R^{3 \times 3})$, and suppose that 
\begin{align*}
    \symfr[Du-p] \in \Lrm^{2}(\Omega;\R^{3 \times 3}).
\end{align*}
We define the \textbf{elastic deformation energy} associated to $u$ and $p$ as
\begin{align}\label{ElasticEnergyDefn}
    \Wcal_{e}(u,p) := \frac{1}{2}\int_{\Omega} \vert \fr[Du - p] \vert^{2}_{\Ebb} \dd x.
\end{align}
Here, $\Ebb$ is the elasticity tensor of \eqref{A:2-i} and $\vert A \vert_{\Ebb}^{2} := A : \Ebb A$ for $A \in \R^{3 \times 3}$. Notice that the properties of $\Ebb$ outlined in \eqref{A:2-i} ensure that, for $u \in \BV(\Omega;\R^{3})$ and a plastic distortion $p \in \Mcal_{\curl}(\overline{\Omega};\R^{3 \times 3})$,
\begin{align}\label{ElasticCoercivity}
    \Wcal_{e}(u,p) \geq C^{-1}\Vert \symfr[Du-p]\Vert^{2}_{\Lrm^{2}},
\end{align}
where $C > 0$ depends only on the constant in \eqref{A:2-i}.

\subsection{Admissible Displacements} For $h \in \R^{3}$, we define the following weak$^{*}$ closed affine subspace of $\BV(\Omega;\R^{3})$:
\begin{align*}
    \BV_{h,\chi}(\Omega;\R^{3}) := \bigg\{u \in \BV(\Omega;\R^{3}): \int_{H} u \dd x = h, \hspace{0.5em} \skew\int_{\Omega} \chi \dd Du = 0\bigg\}.
\end{align*}
The \textbf{admissible displacements} are given by $\BV_{h_{0},\chi}(\Omega;\R^{3})$, where $h_{0}$ is given in \eqref{A:1-iii}. 

We shall often need the following Poincar\'{e} inequality: For $u \in \BV(\Omega;\R^{3})$, we have
\begin{align}\label{Poincare}
     \Vert u - [u]_{H} \Vert_{\Lrm^{1}} \leq C \cdot \Vert Du \Vert(\Omega)
\end{align}
We shall also need the following Korn inequality:
\begin{lemma}\label{lm:Korn} Take $u \in \BV(\Omega;\R^{3})$. Then
\begin{align*}
    \Vert Du \Vert_{\Lrm^{2}} \leq C \bigg(\Vert \sym\hspace{0.2em}Du \Vert_{\Lrm^{2}} + \bigg\vert \skew \int_{\Omega} \chi \dd Du \bigg\vert\bigg).
\end{align*}
\end{lemma}

\begin{proof} It is not restrictive to assume that $\sym\hspace{0.2em}Du \in \Lrm^{2}(\Omega;\R^{3 \times 3})$. In this case, we can use \cite[Lemma 6]{LewintanNeff21} to deduce that $Du \in \Lrm^{2}(\Omega;\R^{3 \times 3})$. It then follows from \cite[Theorem 7]{LewintanNeff21} that there is $R \in \R^{3 \times 3}_{\skew}$ with 
\begin{align}\label{ClassicalKorn}
    \Vert Du - R \Vert_{\Lrm^{2}} \leq C \cdot \Vert \sym\hspace{0.2em}Du \Vert_{\Lrm^{2}}.
\end{align}
Now notice that
\begin{align*}
    \int_{\Omega} \chi R \dd x = \skew \int_{\Omega} \chi \dd Du + \skew \int_{\Omega} \chi(R-Du) \dd x.
\end{align*}
We can then use \eqref{ClassicalKorn} to estimate
\begin{align}\label{Korn1}
    \Vert \chi \Vert_{\Lrm^{1}(\Omega)} \cdot \vert R \vert \leq \bigg\vert \skew \int_{\Omega} \chi \dd Du \bigg\vert+ C \cdot \Vert \chi \Vert_{\Lrm^{2}}\cdot \Vert \sym\hspace{0.2em}Du \Vert_{\Lrm^{2}}.
\end{align}
The result follows after combining \eqref{ClassicalKorn} and \eqref{Korn1} and possibly adjusting $C > 0$.
\end{proof}

We can use this to estimate the free distortion in the following way:

\begin{corollary}\label{cr:frKorn} Let $p \in \Mcal_{\mathrm{curl}}(\overline{\Omega};\R^{3 \times 3})$ and take $u \in \BV(\Omega;\R^{3})$ with
\begin{align}\label{ZeroSkew}
    \skew \int_{\Omega} \chi \dd Du = 0.
\end{align}
Then 
\begin{align*}
    \Vert \fr[Du-p] \Vert_{\Lrm^{2}} \leq C (\Vert \symfr[Du-p]\Vert_{\Lrm^{2}} + \Vert p \Vert(\overline{\Omega}) + \Vert \curl \hspace{0.2em} p \Vert(\overline{\Omega})).
\end{align*}
\end{corollary}

\begin{proof} Recall from \eqref{FreeDistBVDerivative} that there is $w \in \BV(\Omega;\R^{3})$ with $Dw = \fr[Du-p]$. It follows from \eqref{ZeroSkew} that
\begin{align*}
    \bigg\vert \skew \int_{\Omega} \chi \dd Dw \bigg\vert \leq \Vert p \Vert(\overline{\Omega}) + \Vert \beta[\curl \hspace{0.2em} p] \Vert_{\Lrm^{1}} .
\end{align*}
We can then use \eqref{betaEst} to estimate
\begin{align*}
    \bigg\vert \skew \int_{\Omega} \chi \dd Dw \bigg\vert \leq \Vert p \Vert(\overline{\Omega}) + C\cdot\Vert \curl \hspace{0.2em} p \Vert(\overline{\Omega}).
\end{align*}
The result now follows from Lemma \ref{lm:Korn}.
\end{proof}

The following crucial \textit{corrector lemma} lets us adjust an admissible displacement in response to a change in plastic distortion so that the symmetric part of the corresponding free distortion remains unchanged:

\begin{lemma}\label{lm:Corrector} Let $p_{1},p_{2} \in \Mcal_{\mathrm{curl}}(\overline{\Omega};\R^{3 \times 3})$. For $u_{1} \in \BV_{h_{0},\chi}(\Omega;\R^{3})$, there is a unique $u_{2} \in \BV_{h_{0},\chi}(\Omega;\R^{3})$ with
\begin{align}\label{Corrector}
    \symfr[Du_{2}-p_{2}] &= \symfr[Du_{1}-p_{1}].
\end{align}
Furthermore, we have
\begin{align}\label{CorrectorEst}
    \Vert u_{2} - u_{1} \Vert_{\BV} \leq C(\Vert p_{2}-p_{1} \Vert(\overline{\Omega}) + \Vert\curl(p_{2}-p_{1})\Vert(\overline{\Omega})).
\end{align} 
\end{lemma}

\begin{proof} It follows from \eqref{betaPDE} that 
\begin{align*}
    \curl_{\Omega}((p_{2}-p_{1})\mvert \Omega -\beta[\curl(p_{2}-p_{1})]) = 0.
\end{align*}
Since $\Omega$ is contractible, there is a $w \in \BV(\Omega;\R^{3})$ with
\begin{align*}
    Dw = (p_{2}-p_{1}) \mvert \Omega -\beta[\curl(p_{2}-p_{1})].
\end{align*}
We can use the linearity of \eqref{betaPDE} to deduce that
\begin{align}\label{Corrector1}
    Dw -p_{2} \mvert \Omega +\beta[\curl \hspace{0.2em} p_{2}] = -p_{1} \mvert \Omega +\beta[\curl \hspace{0.2em} p_{1}].
\end{align}
Define $u_{2} \in \BV(\Omega;\R^{3})$ by setting
\begin{align*}
    u_{2}(x) := u_{1}(x) + w(x) + Rx + c \hspace{1em} \text{for $\mathscr{L}^{3}$-a.e.\ $x \in \Omega$},
\end{align*}
where $c \in \R^{3}$ and $R \in \R^{3 \times 3}_{\skew}$ are chosen so that $u_{2} \in \BV_{h_{0},\chi}(\Omega;\R^{3})$. We now use \eqref{Corrector1} to compute:
\begin{align*}
    \fr[Du_{2}-p_{2}] &= Du_{2}-p_{2} \mvert \Omega +\beta[\curl \hspace{0.2em} p_{2}] \\
    &= Du_{1}+Dw+R\cdot(\mathscr{L}^{3} \mvert \Omega)-p_{2} \mvert \Omega +\beta[\curl \hspace{0.2em} p_{2}] \\
    &= \fr[Du_{1}-p_{1}]+R\cdot(\mathscr{L}^{3} \mvert \Omega).
\end{align*}
It follows that $\symfr[Du_{2}-p_{2}]=\symfr[Du_{1}-p_{1}]$ and so we have \eqref{Corrector}.

To establish \eqref{CorrectorEst}, first notice that
\begin{align*}
    D(u_{2}-u_{1}) = \fr[D(u_{2}-u_{1})-(p_{2}-p_{1})] + (p_{2}-p_{1}) \mvert \Omega - \beta[\curl(p_{2}-p_{1})].
\end{align*}
It follows from \eqref{betaEst} and Corollary \ref{cr:frKorn} that
\begin{align*}
    \Vert D(u_{2}-u_{1}) \Vert(\Omega) &\leq C(\Vert \symfr[D(u_{2}-u_{1})-(p_{2}-p_{1})]\Vert_{\Lrm^{2}} \\
    &\qquad + \Vert p_{2}-p_{1} \Vert(\overline{\Omega}) + \Vert \curl(p_{2}-p_{1})\Vert(\overline{\Omega})).
\end{align*}
It follows from \eqref{Corrector} that $\symfr[D(u_{2}-u_{1})-(p_{2}-p_{1})] = 0$ and so
\begin{align*}
    \Vert D(u_{2}-u_{1}) \Vert(\Omega) &\leq C(\Vert p_{2}-p_{1} \Vert(\overline{\Omega}) + \Vert \curl(p_{2}-p_{1})\Vert(\overline{\Omega})).
\end{align*}
Since $[u_{2}-u_{1}]_{H} = 0$, the estimate \eqref{CorrectorEst} now follows from the Poincar\'{e} inequality \eqref{Poincare}. The uniqueness of $u_{2}$ follows from \eqref{CorrectorEst}.
\end{proof}

Take $p_{1},p_{2} \in \Mcal_{\curl}(\overline{\Omega};\R^{3 \times 3})$ and let $u_{1} \in \BV_{h_{0},\chi}(\Omega;\R^{3})$. In the following result, we shall use
\begin{align*}
    \Gamma[[p_{1},p_{2}],u_{1}]
\end{align*}
to denote the unique element $u_{2} \in \BV_{h_{0},\chi}(\Omega;\R^{3})$, constructed in Lemma \ref{lm:Corrector}, satisfying \eqref{Corrector}.

\begin{corollary}\label{cr:CorrectorCty} Let $(p_{1}^{j})_{j}, (p_{2}^{j})_{j} \subset \Mcal_{\mathrm{curl}}(\overline{\Omega};\R^{3 \times 3})$ and take $(u_{1}^{j})_{j} \subset \BV_{h_{0},\chi}(\Omega;\R^{3})$. Suppose that
\begin{align*}
    p_{1}^{j} &\to p_{1} \hspace{1em} \text{in $\Tcal^{*}_{\mathrm{curl}}$}, & p_{2}^{j} &\to p_{2} \hspace{1em} \text{in $\Tcal^{*}_{\mathrm{curl}}$}, & u_{1}^{j} &\wkto u_{1} \hspace{1em} \text{in $\BV(\Omega;\R^{3})$}
\end{align*}
for some $p_{1},p_{2} \in \Mcal_{\mathrm{curl}}(\overline{\Omega};\R^{3 \times 3})$ and $u_{1} \in \BV(\Omega;\R^{3})$. Then
\begin{align*}
    \Gamma[[p_{1}^{j},p_{2}^{j}],u^{j}_{1}] \wkto \Gamma[[p_{1},p_{2}],u_{1}].
\end{align*}
\end{corollary}

\subsection{State Spaces} The \textbf{state space} $\mathcal{Q}$ is given by
\begin{align*}
    \Qcal := \bigg\{(u,p,\boldsymbol{T}) &\in \BV_{h_{0},\chi}(\Omega;\R^{3}) \times \Mcal_{\mathrm{curl}}(\overline{\Omega};\R^{3\times 3}) \times \overline{\Disl}(\overline{\Omega}): \\
    &\symfr[Du-p] \in \Lrm^{2}(\Omega;\R^{3 \times 3}), \hspace{0.2em} \curl \hspace{0.2em} p = \frac{1}{2}\sum_{b \in \Bcal}b \otimes T^{b}\bigg\}.
\end{align*}
Notice that we do not include the incompressibility constraint $\tr p = 0$ in the definition of $\Qcal$, because we cannot guarantee that the recovery sequences for $p$ that we will construct in the following satisfy this property. However, the dissipation can of course be made to penalise dislocation climb. 

The following estimate is an immediate consequence of the Poincar\'{e} inequality \eqref{Poincare}, Corollary \ref{cr:frKorn}, \eqref{betaEst}, and \eqref{CurlDistByCore}:

\begin{lemma}\label{lm:uQ} Take $(u,p,\boldsymbol{T}) \in \Qcal$. Then
\begin{align*}
    \Vert u \Vert_{\BV} \leq C(1+ \Vert \symfr[Du-p] \Vert_{\Lrm^{2}} + \Vert p \Vert(\overline{\Omega}) + \Wcal_{c}(\boldsymbol{T})).
\end{align*}
\end{lemma}

Fix $n \in \N$. The \textbf{$n$-discrete state space} is given by
\begin{align*}
    \Qcal_{1/n} := \{(u,p,\boldsymbol{T}) \in \Qcal: \boldsymbol{T} \in \Disl_{1/n}(\overline{\Omega})\}.
\end{align*}
We record the following closure property for $\Qcal$ and $\Qcal_{1/n}$:

\begin{lemma}\label{lm:QClosure} Let $(u_{j},p_{j},\boldsymbol{T}_{j}) \subset \Qcal$ and suppose that there is $M > 0$ with
\begin{align}\label{QClosureBound}
    \Vert\symfr[Du_{j}-p_{j}]\Vert_{\Lrm^{2}} + \Wcal_{c}(\boldsymbol{T}_{j}) \leq M \hspace{1em} \text{for $j \in \N$}.
\end{align}
Take $u \in \BV(\Omega;\R^{3})$, $p \in \Mcal(\overline{\Omega};\R^{3 \times 3})$, and $\boldsymbol{T} \in \overline{\Disl}(\overline{\Omega})$ and suppose that
\begin{align*}
    u_{j} \wkto u, \hspace{1em} p_{j} \wkto p, \hspace{1em} \text{and} \hspace{0.5em}\boldsymbol{T}_{j} &\wkto \boldsymbol{T}.
\end{align*}
Then $(u,p,\boldsymbol{T}) \in \Qcal$. If additionally $(u_{j},p_{j},\boldsymbol{T}_{j}) \subset \Qcal_{1/n}$ for some $n \in \N$, then also $(u,p,\boldsymbol{T}) \in \Qcal_{1/n}$. 
\end{lemma}

\begin{proof} First notice that $u \in \BV_{h_{0},\chi}(\Omega;\R^{3})$. We can use \eqref{QClosureBound} together with \eqref{CoreVsMass} to deduce that $p \in \Mcal_{\mathrm{curl}}(\overline{\Omega};\R^{3 \times 3})$ and
\begin{align}\label{QClosure1}
    \curl \hspace{0.2em} p_{j} &\wkto \curl \hspace{0.2em} p \hspace{1em} \text{in $\Mcal(\overline{\Omega};\R^{3 \times 3})$}, & \curl \hspace{0.2em} p = \frac{1}{2}\sum_{b \in \Bcal}b \otimes T^{b}.
\end{align}
It then follows from \eqref{QClosure1} and \eqref{betaCty} that $\beta[\curl\hspace{0.2em}p_{j}] \wkto \beta[\curl \hspace{0.2em} p]$ in $\Lrm^{3/2}(\Omega;\R^{3 \times 3})$ and so $\fr[Du_{j} - p_{j}] \wkto \fr[Du-p]$ in $\Mcal(\Omega;\R^{3 \times 3})$. We can then use \eqref{QClosureBound} to deduce that 
\begin{align}\label{QClosureL2Conv}
    \symfr[Du_{j} - p_{j}] \rightharpoonup \symfr[Du-p] \hspace{1em} \text{in $\Lrm^{2}(\Omega;\R^{3 \times 3})$}.
\end{align}
Finally, in the case that $((u_{j},p_{j},\boldsymbol{T}_{j}))_{j} \subset \Qcal_{1/n}$ for some $n \in \N$, we additionally appeal to the Federer-Fleming closure theorem to deduce that $\boldsymbol{T} \in \Disl_{1/n}(\overline{\Omega})$.
\end{proof}

Let $(u,p,\boldsymbol{T}) \in \Qcal$ and take $t \in [0,T]$. We define the \textbf{total energy} associated to $(u,p,\boldsymbol{T})$ at time $t$ as
\begin{align}\label{EDefn}
    \Ecal(t,u,p,\boldsymbol{T}) := \Wcal_{e}(u,p) - \langle f(t),u \rangle + h(\Wcal_{c}(\boldsymbol{T})),
\end{align}
where $f$ is defined in \eqref{A:2-iii} and $h$ is defined in \eqref{A:2-iv}. The final two results of this section establish the coercivity and lower semicontinuity of $\Ecal$.

\begin{lemma}\label{lm:ECoercive} Fix $t \in [0,T]$ and take $(u,p,\boldsymbol{T}) \in \Qcal$. Then
\begin{equation}\label{ECoercive}
\begin{aligned}
    \Ecal(t,u,p,\boldsymbol{T}) + C(1 +  \Vert p \Vert(\overline{\Omega})) &\geq C^{-1}(\Vert \symfr[Du-p]\Vert^{2}_{\Lrm^{2}}+ \Wcal_{c}(\boldsymbol{T})^{q})
\end{aligned}
\end{equation}
where $q > 1$ is given in \eqref{A:2-iv}.
\end{lemma}

\begin{proof} We start by using \eqref{EDefn}, \eqref{ElasticCoercivity}, \eqref{A:2-iii}, and \eqref{A:2-iv} to estimate
\begin{align}\label{ECoercive1}
    \Ecal(t,u,p,\boldsymbol{T}) \geq C^{-1}\Vert \symfr[Du-p] \Vert_{\Lrm^{2}}^{2} - C\Vert u \Vert_{\BV} + C^{-1}\Wcal_{c}(\boldsymbol{T})^{q} - C.
\end{align}
On the other hand, we can use Lemma \ref{lm:uQ} to estimate
\begin{align}\label{ECoercive2}
    \Vert u \Vert_{\BV} &\leq C(1+ \Vert \symfr[Du-p]\Vert_{\Lrm^{2}} + \Vert p \Vert(\overline{\Omega}) + \Wcal_{c}(\boldsymbol{T})).
\end{align}
Combining \eqref{ECoercive1} and \eqref{ECoercive2} yields
\begin{align*}
    \Ecal(t&,u,p,\boldsymbol{T}) + C(1+\Vert p \Vert(\overline{\Omega})) \\
    &\geq C^{-1}(\Vert \symfr[Du-p]\Vert_{\Lrm^{2}}^{2} + \Wcal_{c}(\boldsymbol{T})^{q})-C(\Vert \symfr[Du-p]\Vert_{\Lrm^{2}}+\Wcal_{c}(\boldsymbol{T})).
\end{align*}
The result follows after appealing to Young's inequality.
\end{proof}

\begin{lemma}\label{lm:ELSC} Let $(t_{j})_{j} \subset [0,T]$. Take $(u_{j},p_{j},\boldsymbol{T}_{j}) \subset \Qcal$. Suppose that there is $M > 0$ with 
\begin{align}
    \Vert \symfr[Du_{j}-p_{j}] \Vert_{\Lrm^{2}} \leq M \hspace{1em} \text{for $j \in \N$}.
\end{align}
If there is $t \in [0,T]$ and $(u,p,\boldsymbol{T}) \in \Qcal$ with
\begin{align*}
    t_{j} &\to t, & u_{j} &\wkto u, & p_{j} &\wkto p, & \boldsymbol{T}_{j} &\wkto \boldsymbol{T},
\end{align*}
then
\begin{align*}
    \Ecal(t,u,p,\boldsymbol{T}) &\leq \liminf_{j \to \infty} \Ecal(t_{j},u_{j},p_{j},\boldsymbol{T}_{j}).
\end{align*}

\end{lemma}

\begin{proof} First notice that we can work as in Lemma \ref{lm:QClosure} to deduce that \eqref{QClosureL2Conv} holds.
The first term $\Wcal_{e}$ in the definition of $\Ecal$ (see \eqref{EDefn}) is weakly lower semicontinuous by a standard convexity argument (via strong lower semicontinuity and Mazur's lemma --- see \cite[Theorem 2.6]{Rindler26book} for a similar argument). The second term is in fact continuous since $t \mapsto f(t)$ is continuous with values in the dual space to $\BV(\Omega;\R^{3})$ by assumption \eqref{A:2-iii}. The third and final term is weakly$^{*}$ lower semicontinuous because $h$ is continuous and increasing and \eqref{CoreLSC} holds. 
\end{proof}

\section{Energetic Solutions and A Uniform Compatible Load Condition}

\subsection{Energetic Solutions}\label{s:Solutions} In what follows, we introduce notions of energetic solution that are adapted from \cite[Section 9.2]{Rindler26?}:

\begin{definition}\label{df:DiscreteEnergetic} Fix $n \in \N$. Take $J \subset [0,T]$ with $\mathscr{L}^{1}(J) = 0$, let $(u_{0},p_{0},\boldsymbol{T}_{0}) \in \Qcal_{1/n}$, and take
\begin{align*}
    u &\in \Lrm^{\infty}([0,T]; \BV(\Omega;\R^{3})), \\
    p &\in \BV([0,T];\mathcal{M}(\overline{\Omega};\R^{3 \times 3})), \\
    \boldsymbol{S} &\in \BV([0,T];\Disl_{1/n}(\overline{\Omega})).
\end{align*}
The triple $(u,p,\boldsymbol{S})$ is an \textbf{$n$-discrete energetic solution} with \textbf{jump set} $J$ and initial condition $(u_{0},p_{0},\boldsymbol{T}_{0})$ if 
\begin{align*}
    p(0) &= p_{0}, & \partial^{-}\boldsymbol{S} &= \boldsymbol{T}_{0},
\end{align*}
and, for $t \in (0,T)\backslash J$, the following conditions hold:
\begin{enumerate}
    \item[\textbf{(C)}] \textbf{Consistency:} 
    \begin{align*}
        (u(t),p(t),\boldsymbol{S}(t)) \in \Qcal_{1/n}
    \end{align*}
    \item[\textbf{(S)}] \textbf{Stability:} For $\hat{u} \in \BV(\Omega;\R^{3})$, $\widehat{\boldsymbol{S}} \in \Slip_{1/n}(\boldsymbol{S}(t))$ with $(\hat{u},\widehat{\boldsymbol{S}}_{\gg}p(t),\partial^{+}\widehat{\boldsymbol{S}}) \in \Qcal_{1/n}$:
    \begin{align*}
        \Ecal(t,u(t),p(t),\boldsymbol{S}(t)) \leq \Ecal(t,\hat{u},\widehat{\boldsymbol{S}}_{\gg}p(t),\partial^{+}\widehat{\boldsymbol{S}}) + \Diss(\widehat{\boldsymbol{S}})
    \end{align*}
    \item[\textbf{(E)}] \textbf{Energy Balance:}
    \begin{align*}
        \Ecal(t,u(t),p(t),\boldsymbol{S}(t)) - \Ecal(0,u_{0},p_{0},\boldsymbol{T}_{0}) = -\Diss(\boldsymbol{S};[0,t]) - \int_{0}^{t}\langle \dot{f}(s),u(s)\rangle \dd s
    \end{align*}
    \item[\textbf{(P)}] \textbf{Plastic Flow:}
    \begin{align*}
        p(t) = p_{0} + \frac{1}{2}\sum_{b \in \Bcal}b \otimes \star\pbf_{*}(S^{b} \mvert [0,t] \times \R^{3})
    \end{align*}
\end{enumerate}
\end{definition}

\begin{remark}\label{rm:EnergyIntegralWellDefined} To see that the integral in the energy balance of Definition \ref{df:DiscreteEnergetic} is well-defined, suppose $u \in \Lrm^{\infty}([0,T];\BV(\Omega;\R^{3}))$. Then we can take a representative $u:[0,T] \to \BV(\Omega;\R^{3})$ and $M > 0$ such that $u$ is weakly$^{*}$ $\mathscr{L}^{1}$-measurable and $\Vert u(t) \Vert_{\BV} \leq M$ for all $t \in [0,T]$ (see Section \ref{s:Measurability}). The inclusion $\BV(\Omega;\R^{3}) \hookrightarrow \Lrm^{1}(\Omega;\R^{3})$ is weak$^{*}$-to-strong sequentially continuous and so the restriction of this inclusion to $\{v \in \BV(\Omega;\R^{3}): \Vert v \Vert_{\BV} \leq M\}$ is weak$^{*}$-to-strong continuous. It follows that $u$ is strongly $\mathscr{L}^{1}$-measurable with values in $\Lrm^{1}(\Omega;\R^{3})$ and so 
\begin{align*}
    s \mapsto \langle \dot{f}(s),u(s) \rangle \hspace{1em} \text{is $\mathscr{L}^{1}$-measurable}.
\end{align*}
Furthermore, the integral is finite since $u$ and $\dot{f}$ are uniformly bounded on $[0,T]$.
\end{remark}

By replacing discrete quantities with their general counterparts, we obtain:

\begin{definition}\label{df:FieldEnergetic} Take $J \subset [0,T]$ with $\mathscr{L}^{1}(J) = 0$, let $(u_{0},p_{0},\boldsymbol{T}_{0}) \in \Qcal$, and take
\begin{align*}
    u &\in \Lrm^{\infty}([0,T]; \BV(\Omega;\R^{3})), \\
    p &\in \BV([0,T];\mathcal{M}(\overline{\Omega};\R^{3 \times 3})), \\
    \boldsymbol{S} &\in \BV([0,T];\overline{\Disl}(\overline{\Omega})).
\end{align*}
We say that the triple $(u,p,\boldsymbol{S})$ is a \textbf{field energetic solution} with \textbf{jump set} $J$ and initial condition $(u_{0},p_{0},\boldsymbol{T}_{0})$ if 
\begin{align*}
    p(0) &= p_{0}, & \partial^{-}\boldsymbol{S} &= \boldsymbol{T}_{0},
\end{align*}
and, for $t \in (0,T)\backslash J$, the following conditions hold:
\begin{enumerate}
    \item[\textbf{(C)}] \textbf{Consistency:}
    \begin{align*}
        (u(t),p(t),\boldsymbol{S}(t)) \in \Qcal
    \end{align*}
    \item[\textbf{(S)}] \textbf{Stability:} For $\hat{u} \in \BV(\Omega;\R^{3})$, $\widehat{\boldsymbol{S}} \in \overline{\Slip}(\boldsymbol{S}(t))$ with $(\hat{u},\widehat{\boldsymbol{S}}_{\gg}p(t),\partial^{+}\widehat{\boldsymbol{S}}) \in \Qcal$:
    \begin{align*}
        \Ecal(t,u(t),p(t),\boldsymbol{S}(t)) \leq \Ecal(t,\hat{u},\widehat{\boldsymbol{S}}_{\gg}p(t),\partial^{+}\widehat{\boldsymbol{S}}) + \Diss(\widehat{\boldsymbol{S}})
    \end{align*}
    \item[\textbf{(E)}] \textbf{Energy Balance:}
    \begin{align*}
        \Ecal(t,u(t),p(t),\boldsymbol{S}(t)) - \Ecal(0,u_{0},p_{0},\boldsymbol{T}_{0}) = -\Diss(\boldsymbol{S};[0,t]) - \int_{0}^{t}\langle \dot{f}(s),u(s)\rangle \dd s
    \end{align*}
    \item[\textbf{(P)}] \textbf{Plastic Flow:}
    \begin{align*}
        p(t) = p_{0} + \frac{1}{2}\sum_{b \in \Bcal}b \otimes \star\pbf_{*}(S^{b} \mvert [0,t] \times \R^{3})
    \end{align*}
\end{enumerate}
\end{definition}

We shall also need the following approximate version of Definition \ref{df:DiscreteEnergetic}:

\begin{definition}\label{df:ApproxDiscreteEnergetic} Fix $n \in \N$. Take $J \subset [0,T]$ with $\mathscr{L}^{1}(J) = 0$, let $(u_{0},p_{0},\boldsymbol{T}_{0}) \in \Qcal_{1/n}$, and take
\begin{align*}
    u &\in \Lrm^{\infty}([0,T]; \BV(\Omega;\R^{3})) \\
    p &\in \BV([0,T];\mathcal{M}(\overline{\Omega};\R^{3 \times 3})), \\
    \boldsymbol{S} &\in \BV([0,T];\Disl_{1/n}(\overline{\Omega})).
\end{align*}
Let $\delta,\eta: [0,T] \to \R$. We say that the triple $(u,p,\boldsymbol{S})$ is an \textbf{approximate $n$-discrete energetic solution} with \textbf{jump set} $J$ with error parameters $(\delta,\eta)$ and initial condition $(u_{0},p_{0},\boldsymbol{T}_{0})$ if 
\begin{align*}
    p(0) &= p_{0}, & \partial^{-}\boldsymbol{S} &= \boldsymbol{T}_{0},
\end{align*}
and, for $t \in (0,T)\backslash J$, the following conditions hold:
\begin{enumerate}
    \item[\textbf{(C)}] \textbf{Consistency:}  
    \begin{align*}
        (u(t),p(t),\boldsymbol{S}(t)) \in \Qcal_{1/n}
    \end{align*}
    \item[\textbf{(S)}] \textbf{Stability:} For $\hat{u} \in \BV(\Omega;\R^{3})$, $\widehat{\boldsymbol{S}} \in \Slip_{1/n}(\boldsymbol{S}(t))$ with $(\hat{u},\widehat{\boldsymbol{S}}_{\gg}p(t),\partial^{+}\widehat{\boldsymbol{S}}) \in \Qcal_{1/n}$:
    \begin{align*}
        \Ecal(t,u(t),p(t),\boldsymbol{S}(t)) \leq \Ecal(t,\hat{u},\widehat{\boldsymbol{S}}_{\gg}p(t),\partial^{+}\widehat{\boldsymbol{S}}) + \Diss(\widehat{\boldsymbol{S}}) + \delta(t)
    \end{align*}
    \item[\textbf{(E)}] \textbf{Energy Balance:}
    \begin{align*}
        \bigg\vert \Ecal(t,&u(t),p(t),\boldsymbol{S}(t)) - \Ecal(0,u_{0},p_{0},\boldsymbol{T}_{0}) 
        + \Diss(\boldsymbol{S};[0,t]) + \int_{0}^{t}\langle \dot{f}(s),u(s)\rangle \dd s \bigg\vert \leq \eta(t)
    \end{align*}
    \item[\textbf{(P)}] \textbf{Plastic Flow:}
    \begin{align*}
        p(t) = p_{0} + \frac{1}{2}\sum_{b \in \Bcal}b \otimes \star\pbf_{*}(S^{b} \mvert [0,t] \times \R^{3})
    \end{align*}
\end{enumerate}
\end{definition}

\subsection{A Uniform Compatible-Load Condition}\label{s:CompatibleLoad} Let $t \in [0,T]$ and take $(u,p,\boldsymbol{T}) \in \Qcal$. Suppose that $(u,p,\boldsymbol{T})$ is stable at $t$ in the sense that 
\begin{align}\label{CompatibleLoadStability}
    \Ecal(t,u,p,\boldsymbol{T}) \leq \Ecal(t,\hat{u},\widehat{\boldsymbol{S}}_{\gg}p,\partial^{+}\widehat{\boldsymbol{S}}) + \Diss(\widehat{\boldsymbol{S}})
\end{align}
for all $\hat{u} \in \BV(\Omega;\R^{3})$ and $\widehat{\boldsymbol{S}} \in \overline{\Slip}(\boldsymbol{T})$ with $(\hat{u},\widehat{\boldsymbol{S}}_{\gg}p,\partial^{+}\widehat{\boldsymbol{S}}) \in \Qcal$. Take $\boldsymbol{Z} \in \overline{\Slip}(\boldsymbol{0})$ with $\partial^{+}\boldsymbol{Z} = \boldsymbol{0}$ and set $\boldsymbol{S} := [\![0,1]\!] \times \boldsymbol{T} + \boldsymbol{Z}$. Notice that
\begin{align*}
    \boldsymbol{S}_{\gg}p-p = \frac{1}{2}\sum_{b \in \Bcal} b \otimes \star\pbf_{*}Z^{b} = \boldsymbol{Z}_{\gg}0.
\end{align*}
We can use Lemma \ref{lm:FlowConsistent} to deduce that
\begin{align}\label{CompatibleLoadCurlZero}
    \curl_{\Omega}\hspace{0.2em}\boldsymbol{Z}_{\gg}0 = 0.
\end{align}
Since $\Omega$ is contractible, there is $w(\boldsymbol{Z}) \in \BV_{0,\chi}(\Omega;\R^{3})$ with
\begin{align}\label{wDefn}
    \sym\hspace{0.2em}Dw(\boldsymbol{Z}) = \sym (\boldsymbol{Z}_{\gg}0)\mvert \Omega .
\end{align}
It follows from Lemma \ref{lm:Korn} that $w(\boldsymbol{Z})$ is unique and so
\begin{align*}
    w:\bigl\{\boldsymbol{Z}\in \overline{\Slip}(\boldsymbol{0}): \partial^{+}\boldsymbol{Z} = \boldsymbol{0} \bigr\} \to \BV_{0,\chi}(\Omega;\R^{3}).
\end{align*}
is well-defined. Also notice that $\sym\hspace{0.2em}Dw(\boldsymbol{Z}) = \sym(\boldsymbol{S}_{\gg}p-p)$; one can use this to check that
\begin{align*}
    \symfr[D(u+w(\boldsymbol{Z}))-\boldsymbol{S}_{\gg}p] &= \symfr[Du-p]
\end{align*}
and so $(u+w(\boldsymbol{Z}),\boldsymbol{S}_{\gg}p,\boldsymbol{T}) \in \Qcal$. We can therefore use \eqref{CompatibleLoadStability} to deduce that
\begin{align}\label{CompatibleLoad1}
    \Ecal(t,u,p,\boldsymbol{T}) \leq \Ecal(t,u+w(\boldsymbol{Z}),\boldsymbol{S}_{\gg}p,\boldsymbol{T}) + \Diss(\boldsymbol{S}).
\end{align}
On the other hand, it follows from \eqref{Neutral} that
\begin{align}\label{CompatibleLoad2}
    \Diss(\boldsymbol{S}) = \Diss(\boldsymbol{Z}).
\end{align}
Combining \eqref{CompatibleLoad1} and \eqref{CompatibleLoad2} with \eqref{EDefn} yields
\begin{align}\label{ProtoCompatibleLoad}
    \langle f(t),w(\boldsymbol{Z}) \rangle \leq \Diss(\boldsymbol{Z}).
\end{align}
Notice that the condition \eqref{ProtoCompatibleLoad} is independent of the choice of $(u,p,\boldsymbol{T})$: If there is $\boldsymbol{Z} \in \overline{\Slip}(\boldsymbol{0})$ with $\partial^{+}\boldsymbol{Z} = 0$ and $\langle f(t), w(\boldsymbol{Z}) \rangle > \Diss(\boldsymbol{Z})$, then the arguments above show that no state $(u,p,\boldsymbol{T}) \in \Qcal$ is stable at $t$. This demonstrates that a condition of the form \eqref{ProtoCompatibleLoad} is necessary for establishing existence results for energetic solutions; in the sequel, we shall impose the following additional assumption alongside those given in Section \ref{s:Assumptions}: 
\begin{enumerate}[label=\textbf{A3}.,ref=A3]
    \item\label{A:3} \textbf{Uniform Compatible-Load Condition:} There is $\lambda \in (0,1)$ such that, for $t \in [0,T]$ and $\boldsymbol{Z} \in \overline{\Slip}(\boldsymbol{0})$ with $\partial^{+}\boldsymbol{Z}=\boldsymbol{0}$, 
    \begin{align*}
        \langle f(t),w(\boldsymbol{Z})\rangle \leq (1-\lambda) \Diss(\boldsymbol{Z}),
    \end{align*}
    where $w(\boldsymbol{Z})$ is defined via \eqref{wDefn}.
\end{enumerate}
This condition is needed because the structure of the external loading allows for an infinite amount of energy to be injected into the system; the above condition provides a control on this energy via Lemma \ref{lm:ImprovedCoercivity} below.

Although we shall not need it in the sequel, we record the following boundedness result for the map $w$ defined via \eqref{wDefn}:

\begin{lemma}\label{lm:wByDiss} Take $\boldsymbol{Z} \in \overline{\Slip}(\boldsymbol{0})$ with $\partial^{+}\boldsymbol{Z} = \boldsymbol{0}$. Then
\begin{align}\label{wByDiss}
    \Vert w(\boldsymbol{Z}) \Vert_{\BV} \leq C \cdot \Diss(\boldsymbol{Z})
\end{align}
where $w$ is defined via \eqref{wDefn}.
\end{lemma}

\begin{proof} Since \eqref{CompatibleLoadCurlZero} holds, we can use Proposition \eqref{pp:betaExist} to deduce that $\beta[\curl\hspace{0.2em}\boldsymbol{Z}_{\gg}0] = 0$. We therefore have
\begin{align*}
    \fr[Dw(\boldsymbol{Z})-\boldsymbol{Z}_{\gg}0] = Dw(\boldsymbol{Z})-(\boldsymbol{Z}_{\gg}0) \mvert \Omega.
\end{align*}
and it follows from \eqref{wDefn} that $\symfr[Dw(\boldsymbol{Z})-\boldsymbol{Z}_{\gg}0] = 0$. We can therefore use Corollary \ref{cr:frKorn} to deduce that
\begin{align*}
    \Vert \fr[Dw(\boldsymbol{Z})-\boldsymbol{Z}_{\gg}0] \Vert_{\Lrm^{2}} \leq C \cdot \Vert \boldsymbol{Z}_{\gg}0 \Vert(\overline{\Omega}).
\end{align*}
We then use the Poincar\'{e} inequality \eqref{Poincare} to estimate
\begin{align*}
    \Vert w(\boldsymbol{Z}) \Vert_{\BV} \leq C \cdot \Vert \boldsymbol{Z}_{\gg}0 \Vert(\overline{\Omega})
\end{align*}
and the result follows from Lemma \ref{lm:PlasticPathBV}.
\end{proof}

\begin{remark} Notice that one can enforce \eqref{A:3} by assuming that 
\begin{align}\label{SmallLoading}
    C \cdot \sup_{t \in [0,T]} \Vert f(t) \Vert_{\Lrm^{\infty}} \leq 1-\lambda
\end{align}
where $C > 0$ is the constant from \eqref{wByDiss}. Under this \textit{small load condition}, one could also relax the condition imposed in \eqref{A:2-iv} and take $h(s)=s$ after possibly increasing $C > 0$; we retain the weaker assumption \eqref{A:3} to allow for a wider variety of forces.
\end{remark}

The utility of the uniform compatible-load condition is illustrated by the following refinement of Lemma \ref{lm:ECoercive}:

\begin{lemma}\label{lm:ImprovedCoercivity} Suppose that the uniform compatible-load condition \eqref{A:3} holds. Fix $t \in [0,T]$, take $(u,p,\boldsymbol{T}) \in \Qcal$, and let $\boldsymbol{S} \in \overline{\Slip}(\boldsymbol{T})$. Then there is a constant $C > 0$ depending only on the data in Section \ref{s:Assumptions} such that
\begin{align*}
    \Ecal(t,u,p,\boldsymbol{T}) + (1-\lambda)\Diss(\boldsymbol{S}) &+ C(1+\Vert \boldsymbol{S}_{\gg}p\Vert(\overline{\Omega})+\Wcal_{c}(\partial^{+}\boldsymbol{S})) \\
    &\geq C^{-1}(\Vert \symfr[Du-p] \Vert_{\Lrm^{2}}^{2}+\Wcal_{c}(\boldsymbol{T})^{q}).
\end{align*} 
\end{lemma}

Notice that Lemma \ref{lm:ImprovedCoercivity} reduces to Lemma \ref{lm:ECoercive} after setting $\boldsymbol{S} := [\![0,1]\!] \times \boldsymbol{T}$ and using Young's inequality.

\begin{proof} We start by applying Corollary \ref{cr:Coning} to the components of $\partial^{+}\boldsymbol{S}$ and $\boldsymbol{T}$ to obtain
$\boldsymbol{R} \in \overline{\Slip}(\partial^{+}\boldsymbol{S})$ with $\partial^{+}\boldsymbol{R} = \boldsymbol{T}$ and
\begin{align}\label{ImprovedCoercivity1}
    \Diss(\boldsymbol{R}) \leq C \cdot (\Wcal_{c}(\partial^{+}\boldsymbol{S})+\Wcal_{c}(\boldsymbol{T})).
\end{align}
Set $\boldsymbol{Z} := [\![0,1]\!] \times \boldsymbol{T} - \boldsymbol{R} \circ \boldsymbol{S}$. Then we use \eqref{A:3} to estimate
\begin{align}\label{ImprovedCoercivity2}
    \langle f(t),w(\boldsymbol{Z}) \rangle \leq (1-\lambda)\Diss(\boldsymbol{Z}) \leq (1-\lambda)\Diss(\boldsymbol{S}) + \Diss(\boldsymbol{R}).
\end{align}
Notice that
\begin{align*}
    \sym\hspace{0.2em}Dw(\boldsymbol{Z}) &= \sym(p-\boldsymbol{R}_{\gg}\boldsymbol{S}_{\gg}p), & \curl\hspace{0.2em}p &= \curl\hspace{0.2em}\boldsymbol{R}_{\gg}\boldsymbol{S}_{\gg}p
\end{align*}
and so
\begin{align*}
    \symfr[D(u-w(\boldsymbol{Z}))-\boldsymbol{R}_{\gg}\boldsymbol{S}_{\gg}p]=\symfr[Du-p].
\end{align*}
It follows that $(u-w(\boldsymbol{Z}),\boldsymbol{R}_{\gg}\boldsymbol{S}_{\gg}p,\boldsymbol{T}) \in \Qcal$ and so we can use Lemma \ref{lm:uQ} to estimate
\begin{align*}
    \Vert u - w(\boldsymbol{Z}) \Vert_{\BV} \leq C(1+\Vert \symfr[Du-p]\Vert_{\Lrm^{2}}+\Vert \boldsymbol{R}_{\gg}\boldsymbol{S}_{\gg}p \Vert(\overline{\Omega}) + \Wcal_{c}(\boldsymbol{T})).
\end{align*}
We then use Lemma \ref{lm:PlasticPathBV} to deduce that
\begin{align}\label{ImprovedCoercivity3}
    \Vert u-w(\boldsymbol{Z})\Vert_{\BV} \leq C(1+\Vert \symfr[Du-p] \Vert_{\Lrm^{2}} + \Wcal_{c}(\boldsymbol{T})+\Vert \boldsymbol{S}_{\gg}p\Vert(\overline{\Omega}) + \Diss(\boldsymbol{R})).
\end{align}
Combining \eqref{ImprovedCoercivity1}, \eqref{ImprovedCoercivity2}, and  \eqref{ImprovedCoercivity3} yields
\begin{equation}\label{ImprovedCoercivity4}
\begin{aligned}
    \langle f(t),u\rangle &\leq (1-\lambda)\Diss(\boldsymbol{S}) + C(1+\Vert \symfr[Du-p] \Vert_{\Lrm^{2}} \\ &\qquad + \Wcal_{c}(\boldsymbol{T}) + \Wcal_{c}(\partial^{+}\boldsymbol{S})+ \Vert \boldsymbol{S}_{\gg}p\Vert(\overline{\Omega})).
\end{aligned}
\end{equation}
On the other hand, it follows from \eqref{ElasticCoercivity} and \eqref{A:2-iv} that
\begin{align}\label{ImprovedCoercivity5}
    \Ecal(t,u,p,\boldsymbol{T}) \geq C^{-1}\Vert \symfr[Du-p] \Vert_{\Lrm^{2}}^{2} - \langle f(t),u \rangle + C^{-1}\Wcal_{c}(\boldsymbol{T})^{q}-C.
\end{align}
Combining \eqref{ImprovedCoercivity4} and \eqref{ImprovedCoercivity5} yields
\begin{align*}
    \Ecal(t&,u,p,\boldsymbol{T}) + (1-\lambda)\Diss(\boldsymbol{S}) + C(1+\Vert \boldsymbol{S}_{\gg}p\Vert(\overline{\Omega})+\Wcal_{c}(\partial^{+}\boldsymbol{S})) \\
    &\geq C^{-1}(\Vert \symfr[Du-p]\Vert_{\Lrm^{2}}^{2} + \Wcal_{c}(\boldsymbol{T})^{q})-C(\Vert \symfr[Du-p]\Vert_{\Lrm^{2}}+\Wcal_{c}(\boldsymbol{T})).
\end{align*}
The result follows by appealing to Young's inequality.
\end{proof}

We now use Lemma \ref{lm:ImprovedCoercivity} to establish a joint coercivity result that will allow us to establish the existence of discrete energetic solutions in Section \ref{s:DiscreteExist}:

\begin{corollary}\label{cr:MainCoercivity} Suppose that the uniform compatible-load condition \eqref{A:3} holds. Fix $t \in [0,T]$, take $(u,p,\boldsymbol{T}) \in \Qcal$, and let $\boldsymbol{S} \in \overline{\Slip}(\boldsymbol{T})$. Then there is a constant $C_{\lambda} > 0$, depending only on the data in Section \ref{s:Assumptions} and the constant $\lambda > 0$ from \eqref{A:3}, such that
\begin{align*}
    \Ecal(t,u,p,\boldsymbol{T}) &+ \Diss(\boldsymbol{S}) + C_{\lambda}(1+\Vert \boldsymbol{S}_{\gg}p\Vert(\overline{\Omega})+\Wcal_{c}(\partial^{+}\boldsymbol{S})) \\
    &\geq C_{\lambda}^{-1}(\Vert u \Vert_{\BV} + \Vert \symfr[Du-p] \Vert_{\Lrm^{2}} + \Diss(\boldsymbol{S}) \\
    &\qquad + \Wcal_{c}(\boldsymbol{T}) + \Vert \boldsymbol{S}_{\gg}p \Vert(\overline{\Omega}) + \Wcal_{c}(\partial^{+}\boldsymbol{S})).
\end{align*}
\end{corollary}

\begin{proof} We first estimate as in the proof of Lemma \ref{lm:ImprovedCoercivity}, using Lemma \ref{lm:uQ} and Lemma \ref{lm:PlasticPathBV} to deduce that
\begin{align*}
    \Vert u \Vert_{\BV} \leq C(1+\Vert \symfr[Du-p]\Vert_{\Lrm^{2}} + \Vert \boldsymbol{S}_{\gg}p \Vert(\overline{\Omega})+\Diss(\boldsymbol{S})+\Wcal_{c}(\boldsymbol{T}))
\end{align*}
We then use Young's inequality to obtain
\begin{equation}\label{MainCoercivity1}
\begin{aligned}
    \Vert &u \Vert_{\BV} + \Vert \symfr[Du-p]\Vert_{\Lrm^{2}} + \Diss(\boldsymbol{S}) + \Wcal_{c}(\boldsymbol{T}) + \Vert \boldsymbol{S}_{\gg}p \Vert(\overline{\Omega}) + \Wcal_{c}(\partial^{+}\boldsymbol{S}) \\
    &\leq C(1 + \Vert \symfr[Du-p]\Vert_{\Lrm^{2}}^{2} + \Diss(\boldsymbol{S}) + \Wcal_{c}(\boldsymbol{T})^{q} + \Vert \boldsymbol{S}_{\gg}p \Vert(\overline{\Omega}) + \Wcal_{c}(\partial^{+}\boldsymbol{S})).
\end{aligned}
\end{equation}
Now define
\begin{align}\label{MainCoercivity1.5}
    \alpha := \Ecal(t,u,p,\boldsymbol{T}) + \Diss(\boldsymbol{S}) + C(1+\Vert \boldsymbol{S}_{\gg}p\Vert(\overline{\Omega})+\Wcal_{c}(\partial^{+}\boldsymbol{S})),
\end{align}
Here, we choose $C > 0$ to be the constant from Lemma \ref{lm:ImprovedCoercivity}. It follows from Lemma \ref{lm:ImprovedCoercivity} that
\begin{align}\label{MainCoercivity2}
    \Vert \symfr[Du-p]\Vert_{\Lrm^{2}}^{2} + \Wcal_{c}(\boldsymbol{T})^{q} &\leq C\cdot\alpha, & \Diss(\boldsymbol{S}) \leq \frac{C}{\lambda} \cdot \alpha.
\end{align}
The result follows after combining \eqref{MainCoercivity1} and \eqref{MainCoercivity2} and choosing $C_{\lambda}$ appropriately.
\end{proof}

\section{Statements of Results}\label{s:Results} Our first result concerns the solvability of the $n$-discrete system:

\begin{theorem}\label{tm:DiscreteExist} Fix $n \in \N$ and take $\delta \geq 0$. Take $(u_{0},p_{0},\boldsymbol{T}_{0}) \in \Qcal_{1/n}$ and suppose that, for $\hat{u} \in \BV(\Omega;\R^{3})$ and $\widehat{\boldsymbol{S}} \in \Slip_{1/n}(\boldsymbol{T}_{0})$ with $(\hat{u},\widehat{\boldsymbol{S}}_{\gg}p_{0},\partial^{+}\widehat{\boldsymbol{S}}) \in \Qcal_{1/n}$, we have
\begin{align}\label{DiscreteInitialStability}
        \Ecal(0,u_{0},p_{0},\boldsymbol{T}_{0}) \leq \Ecal(0,\hat{u},\widehat{\boldsymbol{S}}_{\gg}p_{0},\partial^{+}\widehat{\boldsymbol{S}}) + \Diss(\widehat{\boldsymbol{S}}) + \delta.
\end{align}
In addition to the standard assumptions of Section \ref{s:Assumptions}, assume that the uniform compatible-load condition \eqref{A:3} holds. Then there is an approximate $n$-discrete energetic solution $(u,p,\boldsymbol{S})$ with error parameters $(0,\delta)$, and initial condition $(u_{0},p_{0},\boldsymbol{T}_{0})$. Furthermore, there is a constant $\overline{C} > 0$, depending only on the data in Section \ref{s:Assumptions} and the constant $\lambda > 0$ from \eqref{A:3}, such that,
\begin{align}\label{DiscreteExistBounds1}
    \Var(p) + \Diss(\boldsymbol{S}) + \Wcal_{c}(\partial^{-}\boldsymbol{S}) + \Wcal_{c}(\partial^{+}\boldsymbol{S}) \leq \overline{C}\alpha_{0}
\end{align}
and, for $t \in [0,T] \backslash J$,
\begin{align}\label{DiscreteExistBounds2}
    \Vert u(t) \Vert_{\BV} + \Vert \symfr[Du(t)-p(t)]\Vert_{\Lrm^{2}} + \Vert p(t) \Vert(\overline{\Omega}) + \Wcal_{c}(\boldsymbol{S}(t)) \leq \overline{C}\alpha_{0}
\end{align}
where
\begin{align}\label{Alpha0Defn}
    \alpha_{0} := \Ecal(0,u_{0},p_{0},\boldsymbol{T}_{0})+C_{\lambda}(1+\Vert p_{0} \Vert(\overline{\Omega})+\Wcal_{c}(\boldsymbol{T}_{0}))
\end{align}
and $C_{\lambda} > 0$ is given by Corollary \ref{cr:MainCoercivity}. In particular, if $\delta = 0$, then $(u,p,\boldsymbol{S})$ is an $n$-discrete energetic solution on $[0,T]$ with initial condition $(u_{0},p_{0},\boldsymbol{T}_{0})$.
\end{theorem}

Motivated by the fact that $n$-discrete energetic solutions exist for all $n \in \N$, we conduct a homogenisation procedure to obtain field energetic solutions from sequences of approximate $n$-discrete energetic solutions with vanishing error parameters. We then use this to obtain an existence result for field energetic solutions:

\begin{theorem}\label{tm:Forwards} Let $n: \N \to \N$ be strictly increasing with the property that, for $m, N \in \N$, there is $j \geq N$ with $m \vert n(j)$. For $j \in \N$, take
\begin{align*}
    (u^{j}_{0},p^{j}_{0},\boldsymbol{T}^{j}_{0}) &\in \Qcal_{1/n(j)}, &
    \delta_{j}, \eta_{j} &: [0,T] \to \R,
\end{align*}
and suppose that $(u_{j},p_{j},\boldsymbol{S}_{j})$ is an approximate $n(j)$-discrete energetic solution with error parameters $(\delta_{j},\eta_{j})$ and initial condition $(u^{j}_{0},p^{j}_{0},\boldsymbol{T}^{j}_{0})$. Take $(u_{0},p_{0},\boldsymbol{T}_{0}) \in \Qcal$ and suppose that
\begin{align*}
    \Ecal(0,u_{0},p_{0},\boldsymbol{T}_{0}) \leq \Ecal(0,\hat{u},\widehat{\boldsymbol{S}}_{\gg}p_{0},\partial^{+}\widehat{\boldsymbol{S}}) + \Diss(\widehat{\boldsymbol{S}})
\end{align*}
for all $\hat{u} \in \BV(\Omega;\R^{3})$ and $\widehat{\boldsymbol{S}} \in \overline{\Slip}(\boldsymbol{T}_{0})$ with $(\hat{u},\widehat{\boldsymbol{S}}_{\gg}p_{0},\partial^{+}\widehat{\boldsymbol{S}}) \in \Qcal$. Take $J \subset [0,T]$ with $\mathscr{L}^{1}(J) = 0$ and let
\begin{align*}
    u &\in \Lrm^{\infty}([0,T];\BV(\Omega;\R^{3})), \\
    p &\in \BV([0,T];\mathcal{M}(\overline{\Omega};\R^{3 \times 3})), \\
    \boldsymbol{S} &\in \BV([0,T];\overline{\Disl}(\overline{\Omega})).
\end{align*}
Suppose that:
\begin{enumerate}
    \item[(i)] The error parameters satisfy
    \begin{align*}
        \delta_{j}(t) &\to 0, \hspace{1em} \eta_{j}(t) \to 0 & & \text{for $t \in [0,T]\backslash J$}.
    \end{align*}
    \item[(ii)] The initial data satisfies
    \begin{align}\label{ForwardsInitialConv}
        p_{0}^{j} &\wkto p_{0}, & p_{j}(0) &\wkto p(0), & \boldsymbol{T}_{0}^{j} &\wkto \boldsymbol{T}_{0} ,&\Ecal(0,u^{j}_{0},p^{j}_{0},\boldsymbol{T}^{j}_{0}) &\to \Ecal(0,u_{0},p_{0},\boldsymbol{T}_{0}).
    \end{align}
    \item[(iii)] The sequence $(\boldsymbol{S}_{j})_{j}$ satisfies 
    \begin{align*}
        \boldsymbol{S}_{j} &\wkto \boldsymbol{S}, & \Diss(\boldsymbol{S}) &\leq \liminf_{j \to \infty} \Diss(\boldsymbol{S}_{j}).
    \end{align*}
    Furthermore, for $t \in [0,T]\backslash J$, assume that
    \begin{align*}
        \tbf_{\#}\Vert \boldsymbol{S} \Vert(\{t\}) &= 0, & \boldsymbol{S}_{j}(t) &\wkto \boldsymbol{S}(t)
    \end{align*}
    and
    \begin{align}\label{ForwardsSRestrictConv}
        \boldsymbol{S}_{j} \mvert [0,t] &\wkto \boldsymbol{S} \mvert [0,t], & \Diss(\boldsymbol{S};[0,t]) &\leq \liminf_{j \to \infty}\Diss(\boldsymbol{S}_{j};[0,t]). 
    \end{align}
    \item[(iv)] The sequences $(u_{j})_{j}$ and $(p_{j})_{j}$ satisfy
    \begin{align*}
        u_{j}(t) &\wkto u(t), \hspace{1em} p_{j}(t) \wkto p(t) & & \text{for $t \in [0,T]\backslash J$}.
    \end{align*}
    \item[(v)] There is $\overline{C} > 0$ such that, for $j \in \N$,
    \begin{align*}
        \Var(p_{j}) + \Diss(\boldsymbol{S}_{j}) \leq \overline{C}
    \end{align*}
    and, for all $t \in [0,T] \backslash J$,
    \begin{align}\label{ForwardsBounds}
        \Vert u_{j}(t)\Vert_{\BV} + \Vert \symfr[Du_{j}(t)-p_{j}(t)] \Vert_{\Lrm^{2}} + \Wcal_{c}(\boldsymbol{S}_{j}(t)) \leq \overline{C}.
    \end{align}
\end{enumerate}
Then $(u,p,\boldsymbol{S})$ is a field energetic solution with initial condition $(u_{0},p_{0},\boldsymbol{T}_{0})$. 
\end{theorem}

The divisibility condition imposed on $n$ in Theorem \ref{tm:Forwards} will allow us to apply Proposition \ref{pp:SlipPrescribedApprox}; to see the role of this condition, we refer to the discussion following the statement of Proposition \ref{pp:SlipPrescribedApprox}. We then use this homogenisation result to obtain an existence result for field energetic solutions:

\begin{theorem}\label{tm:FieldExist} Let $(u_{0},p_{0},\boldsymbol{T}_{0}) \in \Qcal$ and suppose that, for $\hat{u} \in \BV(\Omega;\R^{3})$ and $\widehat{\boldsymbol{S}} \in \overline{\Slip}(\boldsymbol{T}_{0})$ with $(\hat{u},\widehat{\boldsymbol{S}}_{\gg}p_{0},\partial^{+}\widehat{\boldsymbol{S}}) \in \Qcal$, we have
\begin{align}\label{FieldInitialStability}
        \Ecal(0,u_{0},p_{0},\boldsymbol{T}_{0}) \leq \Ecal(0,\hat{u},\widehat{\boldsymbol{S}}_{\gg}p_{0},\partial^{+}\widehat{\boldsymbol{S}}) + \Diss(\widehat{\boldsymbol{S}}).
\end{align}
In addition to the standard assumptions of Section \ref{s:Assumptions}, assume that the uniform compatible-load condition \eqref{A:3} holds. Then there is a field energetic solution $(u,p,\boldsymbol{S})$ with initial condition $(u_{0},p_{0},\boldsymbol{T}_{0})$. 
\end{theorem}

Finally, we demonstrate that every field energetic solution can be recovered from a sequence of approximate $n$-discrete energetic solutions with vanishing error parameters:

\begin{theorem}\label{tm:Backwards} Take $(u_{0},p_{0},\boldsymbol{T}_{0}) \in \Qcal$, let $(u,p,\boldsymbol{S})$ be a field energetic solution with initial condition $(u_{0},p_{0},\boldsymbol{T}_{0})$, and suppose that
\begin{align*}
    \Vert \boldsymbol{S} \Vert_{\Lrm^{\infty}_{c}} < \infty.
\end{align*}
Then there is $J \subset [0,T]$ with $\mathscr{L}^{1}(J) = 0$, an $n: \N \to \N$ and, for $j \in \N$, there are
\begin{align*}
    (u_{0}^{j},p_{0}^{j},\boldsymbol{T}_{0}^{j}) &\in \Qcal_{1/n(j)}, & \delta_{j} &: [0,T] \to \R, & \eta_{j} &: [0,T] \to \R, 
\end{align*}
and an approximate $n(j)$-discrete energetic solution $(u_{j},p_{j},\boldsymbol{S}_{j})$ on $[0,T]$ with error parameters $(\delta_{j},\eta_{j})$ and initial condition $(u_{0}^{j},p_{0}^{j},\boldsymbol{T}_{0}^{j})$ with the following properties:
\begin{enumerate}
    \item[(i)] The error parameters satisfy
    \begin{align*}
        \delta_{j}(t) &\to 0, \hspace{1em} \eta_{j}(t) \to 0 & & \text{for $t \in [0,T]\backslash J$}.
    \end{align*}
    \item[(ii)] The initial data satisfies 
    \begin{align*}
        u_{0}^{j} &\wkto u_{0}, & \Vert p_{0}^{j} - p_{0} \Vert(\overline{\Omega}) &\to 0, & \boldsymbol{T}^{j}_{0} &\wkto \boldsymbol{T}_{0}, & \Wcal_{c}(\boldsymbol{T}^{j}_{0}) &\to \Wcal_{c}(\boldsymbol{T}_{0}).
    \end{align*}
    Furthermore, we have
    \begin{align*}
        \symfr[Du_{0}^{j}-p^{j}_{0}] &= \symfr[Du_{0}-p_{0}] \hspace{1em} \text{for $j \in \N$}.
    \end{align*}
    In particular, we have 
    \begin{align*}
        \Ecal(0,u_{0}^{j},p_{0}^{j},\boldsymbol{T}_{0}^{j}) \to \Ecal(0,u_{0},p_{0},\boldsymbol{T}_{0}).
    \end{align*}
    \item[(iii)] The sequence $(\boldsymbol{S}_{j})_{j}$ satisfies
    \begin{align*}
        \boldsymbol{S}_{j} &\wkto \boldsymbol{S}, & \Diss(\boldsymbol{S}_{j}) \to \Diss(\boldsymbol{S}).
    \end{align*}
    Furthermore, for $t \in [0,T]\backslash J$, we have
    \begin{align*}
        \boldsymbol{S}_{j}(t) &\wkto \boldsymbol{S}(t), & 
        \Wcal_{c}(\boldsymbol{S}_{j}(t)) &\to \Wcal_{c}(\boldsymbol{S}(t)),
    \end{align*}
    and
    \begin{align*}
    \boldsymbol{S}_{j} \mvert [0,t] &\wkto \boldsymbol{S} \mvert [0,t], & \Diss(\boldsymbol{S}_{j};[0,t]) &\to \Diss(\boldsymbol{S};[0,t]).
    \end{align*}
    \item[(iv)] For $t \in [0,T]\backslash J$, the sequences $(u_{j}(t))_{j}$ and $(p_{j}(t))_{j}$ satisfy
    \begin{align*}
        u_{j}(t) &\wkto u(t), & p_{j}(t) &\wkto p(t)
    \end{align*}
    and 
    \begin{align*}
        \symfr[Du_{j}(t)-p_{j}(t)] = \symfr[Du(t)-p(t)] \hspace{1em} \text{for $j \in \N$}.
    \end{align*}
    \item[(v)] There is $\overline{C} > 0$ such that, for $j \in \N$,
    \begin{align}\label{BackwardsBound1}
        \Var(p_{j}) + \Diss(\boldsymbol{S}_{j}) \leq \overline{C}
    \end{align}
    and, for all $t \in [0,T] \backslash J$,
    \begin{align}\label{BackwardsBound2}
        \Vert p_{j}(t) \Vert(\overline{\Omega}) + \Vert u_{j}(t) \Vert_{\BV} + \Wcal_{c}(\boldsymbol{S}_{j}(t)) &\leq \overline{C}.
    \end{align}
\end{enumerate}
\end{theorem}

\section{Existence of Discrete Energetic Solutions}\label{s:DiscreteExist}

The broad strategy for the proof of Theorem \ref{tm:DiscreteExist} follows the one for \cite[Theorem 4.11]{Rindler25}, which is itself structurally similar to the classical minimising movements proof for the existence of energetic solutions~\cite{MielkeRoubicek15book}, but with various differences originating from the space-time formulation. Moreover, in contrast to~\cite{Rindler25}, we have a linearised energy, different internal variables, and various technical differences. 

Throughout this section, we shall assume that the uniform compatible-load condition \eqref{A:3} holds; we also fix $n \in \N$.

\subsection{Existence of a Time-Incremental Solution} Fix $(u_{0},p_{0},\boldsymbol{T}_{0}) \in \Qcal_{1/n}$ and take $M > 0$ with $M \geq \Wcal_{c}(\boldsymbol{T}_{0})$. For $N \in \N$, set
\begin{align*}
    u^{N}_{0} &:= u_{0}, & p^{N}_{0} &:= p_{0}, & \boldsymbol{T}^{N}_{0} &:= \boldsymbol{T}_{0}, & \boldsymbol{S}^{N}_{0} &:= [\![0,1]\!] \times \boldsymbol{T}_{0},
\end{align*}
and, for $i \in \{0,\dotsc,N\}$, set 
\begin{align*}
    t^{N}_{i} := i \Delta T^{N} \hspace{1em} \text{where} \hspace{1em} \Delta T^{N} := \frac{T}{N}.
\end{align*}
Now fix $N \in \N$ and take $j \in \{0,\dotsc,N\}$. We say that a collection
\begin{align*}
    \{(u_{i}^{N},p_{i}^{N},\boldsymbol{T}_{i}^{N},\boldsymbol{S}_{i}^{N})\}_{i=0}^{j} \subset \Qcal_{1/n} \times \BV([0,1];\Disl_{1/n}(\overline{\Omega}))
\end{align*}
is a \textbf{$\Delta T^{N}$-time-incremental solution} up to step $j$ (starting from $(u_{0},p_{0},\boldsymbol{T}_{0})$) if the following conditions hold:
\begin{enumerate}[
    label=(IP-\roman*),
    ref=IP-\roman*]
    \item\label{itm:IP-i} For $i \in \{1,\dotsc,j\}$, $\boldsymbol{S}^{N}_{i} \in \Slip_{1/n}(\boldsymbol{T}^{N}_{i-1})$ and $(p_{i}^{N},\boldsymbol{T}_{i}^{N}) = ((\boldsymbol{S}_{i}^{N})_{\gg}p_{i-1},\partial^{+}\boldsymbol{S}_{i}^{N})$.
    \item\label{itm:IP-ii} For $i \in \{1,\dotsc,j\}$, $(u_{i}^{N},\boldsymbol{S}_{i}^{N})$ minimises 
    \begin{align*}
        (\hat{u},\widehat{\boldsymbol{S}}) \mapsto \Ecal(t^{N}_{i},\hat{u},\widehat{\boldsymbol{S}}_{\gg}p^{N}_{i-1},\partial^{+}\widehat{\boldsymbol{S}}) + \Diss(\widehat{\boldsymbol{S}})
    \end{align*}
    over all $\hat{u} \in \BV(\Omega;\R^{3})$ and $\widehat{\boldsymbol{S}} \in \Slip_{1/n}(\boldsymbol{T}^{N}_{i-1})$ with 
    \begin{align}\label{TestCondIP-ii}
    (\hat{u},\widehat{\boldsymbol{S}}_{\gg}p^{N}_{i-1},\partial^{+}\widehat{\boldsymbol{S}}) \in \Qcal_{1/n} \hspace{1em} \text{and} \hspace{1em}\Vert \widehat{\boldsymbol{S}} \Vert_{\Lrm^{\infty}_{c}} \leq M.
    \end{align}
\end{enumerate}
If $j=N$, then we shall simply say that $\{(u^{N}_{i},p^{N}_{i},\boldsymbol{T}^{N}_{i},\boldsymbol{S}^{N}_{i})\}_{i=0}^{N}$ is a $\Delta T^{N}$-time-incremental solution (starting from $(u_{0},p_{0},\boldsymbol{T}_{0})$).

In what follows we shall construct $\Delta T^{N}$-time-incremental solutions for $N \in \N$. We begin with the following a priori estimate:

\begin{lemma}\label{lm:AlphaAP} Take $N \in \N$ and let $j \in \{1,\dotsc,N\}$. Suppose that $\{(u^{N}_{i},p^{N}_{i},\boldsymbol{T}^{N}_{i},\boldsymbol{S}^{N}_{i})\}_{i=0}^{j}$ is a $\Delta T^{N}$-time-incremental solution up to step $j$. Then there is a constant $C^{*} > 0$, depending only on the data in Section \ref{s:Assumptions} and the constant $\lambda > 0$ from \eqref{A:3}, such that
\begin{align}\label{AlphaDiffIneq}
    \frac{\alpha^{N}_{j}-\alpha^{N}_{j-1}}{\Delta T^{N}} &\leq C^{*} \cdot \alpha^{N}_{j-1}
\end{align}
where, for $i \in \{0,\dotsc,j\}$,
\begin{align}\label{AlphaDefn}
    \alpha^{N}_{i} := \Ecal(t^{N}_{i},u^{N}_{i},p^{N}_{i},\boldsymbol{T}^{N}_{i}) + \sum_{l=0}^{i} \Diss(\boldsymbol{S}^{N}_{l}) + C_{\lambda}(1+\Vert p_{0}^{N} \Vert(\overline{\Omega}) + \Wcal_{c}(\boldsymbol{T}_{0}^{N})).
\end{align}
and $C_{\lambda} > 0$ is given by Corollary \ref{cr:MainCoercivity}.
\end{lemma}
\begin{proof} For $i \in \{0,\dotsc,j\}$, we shall use the following abbreviations:
\begin{align*}
    e^{N}_{i} &:= \Ecal(t^{N}_{i},u^{N}_{i},p^{N}_{i},\boldsymbol{T}^{N}_{i}), &d^{N}_{i} &:= \Diss(\boldsymbol{S}^{N}_{i}).
\end{align*}
We begin by testing the minimality at $i=j$ with $(u^{N}_{j-1},[\![0,1]\!] \times \boldsymbol{T}^{N}_{j-1})$ to obtain the following estimate:
\begin{align*}
    e^{N}_{j}+d^{N}_{j} \leq \Ecal(t^{N}_{j},u^{N}_{j-1},p^{N}_{j-1},\boldsymbol{T}^{N}_{j-1}) = e^{N}_{j-1}-\int^{t^{N}_{j}}_{t^{N}_{j-1}}\langle \dot{f}(s),u^{N}_{j-1}\rangle \dd s.
\end{align*}
It follows that
\begin{align}\label{APEst1}
    \alpha^{N}_{j}-\alpha^{N}_{j-1} \leq \int_{t^{N}_{j-1}}^{t^{N}_{j}} \Vert \dot{f}(s)\Vert_{\Lrm^{\infty}} \cdot \Vert u^{N}_{j-1}\Vert_{\Lrm^{1}} \dd s.
\end{align}
It therefore remains to estimate $\Vert u^{N}_{j-1} \Vert_{\Lrm^{1}}$. To this end, we apply Corollary \ref{cr:MainCoercivity} with 
\begin{align*}
    t&:= t^{N}_{j-1}, & (u,p,\boldsymbol{T}) &:= (u^{N}_{j-1},p^{N}_{j-1},\boldsymbol{T}^{N}_{j-1}), & \boldsymbol{S} &:= (\boldsymbol{S}^{N}_{j-1} \circ \cdots \circ \boldsymbol{S}^{N}_{0})^{-1},
\end{align*}
to obtain
\begin{align}\label{APEst2}
    \Vert u^{N}_{j-1}\Vert_{\BV} \leq C_{\lambda}\bigg(e^{N}_{j-1}+\sum_{i=0}^{j-1}d^{N}_{i}+C_{\lambda}(1+\Vert p_{0}^{N} \Vert(\overline{\Omega}) + \Wcal_{c}(\boldsymbol{T}_{0}^{N}))\bigg) = C_{\lambda}\alpha^{N}_{j-1}.
\end{align}
The result follows after combining \eqref{APEst1} and \eqref{APEst2}.
\end{proof}

Notice that the $C^{*}$ constructed in Lemma \ref{lm:AlphaAP} does not depend on $M$ or $n$. We now proceed to construct $\Delta T^{N}$-time-incremental solutions:

\begin{proposition}\label{pp:TIExist} Take $N \in \N$. Then there is a $\Delta T^{N}$-time-incremental solution starting from $(u_{0},p_{0},\boldsymbol{T}_{0})$.
\end{proposition}

\begin{proof} We proceed by induction over $j \in \{1,\dotsc,N\}$. Fix $j \in \{1,\dotsc,N\}$ and suppose that $\{(u^{N}_{i},p^{N}_{i},\boldsymbol{T}^{N}_{i},\boldsymbol{S}^{N}_{i})\}_{i=0}^{j-1}$ is a $\Delta T^{N}$-time-incremental solution up to step $j-1$ (notice that we can always do this if $j=1$ since $\{(u_{0}^{N},p_{0}^{N},\boldsymbol{T}^{N}_{0},\boldsymbol{S}^{N}_{0})\}$ is the unique $\Delta T^{N}$-time-incremental solution up to step $0$ starting from $(u_{0},p_{0},\boldsymbol{T}_{0})$). We will construct a minimiser 
\begin{align*}
(u_{*},\boldsymbol{S}_{*}) \in \BV(\Omega;\R^{3}) \times \Slip_{1/n}(\boldsymbol{T}^{N}_{j-1})
\end{align*}
of the problem \eqref{itm:IP-ii} at $i=j$. Notice that the set of candidates is not empty since it contains $\hat{u}:=u^{N}_{j-1}$ and $\widehat{\boldsymbol{S}} := [\![0,1]\!] \times \boldsymbol{T}^{N}_{j-1}$. For $i \in \{0,\dotsc,j-1\}$, we define $\alpha^{N}_{i}$ as in \eqref{AlphaDefn}.
\begin{claim} The minimisation problem \eqref{itm:IP-ii} at $i=j$ may equivalently be taken over $\hat{u} \in \BV(\Omega;\R^{3})$ and $\widehat{\boldsymbol{S}} \in \Slip_{1/n}(\boldsymbol{T}^{N}_{j-1})$ satisfying \eqref{TestCondIP-ii} and
\begin{align}\label{TIBound}
    \Vert \hat{u} \Vert_{\BV} + \Vert \symfr[D\hat{u}-\widehat{\boldsymbol{S}}_{\gg}p^{N}_{j-1}]\Vert_{\Lrm^{2}} + \Diss(\widehat{\boldsymbol{S}}) &\leq C(\alpha^{N}_{j-1}),
\end{align}
where $C(\alpha^{N}_{j-1}) > 0$ depends only on the data in Section \ref{s:Assumptions}, the constant $\lambda > 0$ from \eqref{A:3}, and $\alpha^{N}_{j-1}$.
\end{claim}
\begin{proof} For $i \in \{0,\dotsc,j-1\}$, we abbreviate $d^{N}_{i} := \Diss(\boldsymbol{S}^{N}_{i})$ as in the proof of Lemma \ref{lm:AlphaAP}; for $\hat{u} \in \BV(\Omega;\R^{3})$ and $\widehat{\boldsymbol{S}} \in \Slip_{1/n}(\boldsymbol{T}^{N}_{j-1})$, we abbreviate
\begin{align*}
    e^{N}_{j}(\hat{u},\widehat{\boldsymbol{S}}) &:= \Ecal(t^{N}_{j},\hat{u},\widehat{\boldsymbol{S}}_{\gg}p^{N}_{j-1},\partial^{+}\widehat{\boldsymbol{S}}), & d^{N}_{j}(\widehat{\boldsymbol{S}}) &:= \Diss(\widehat{\boldsymbol{S}}),
\end{align*}
and define
\begin{align*}
    \alpha^{N}_{j}(\hat{u},\widehat{\boldsymbol{S}}) &:= e^{N}_{j}(\hat{u},\widehat{\boldsymbol{S}})+\sum_{i=0}^{j-1}d^{N}_{i} + d_{j}^{N}(\widehat{\boldsymbol{S}}) + C_{\lambda}(1+\Vert p_{0}^{N} \Vert(\overline{\Omega}) + \Wcal_{c}(\boldsymbol{T}_{0}^{N})). \vphantom{\sum_{i=0}^{j-1}}
\end{align*}
where $C_{\lambda} > 0$ is the constant from Corollary \ref{cr:MainCoercivity}. We can use Lemma \ref{lm:AlphaAP} to deduce that we can restrict the minimisation problem \eqref{itm:IP-ii} at $i=j$ to $\hat{u} \in \BV(\Omega;\R^{3})$ and $\widehat{\boldsymbol{S}} \in \Slip_{1/n}(\boldsymbol{T}^{N}_{j-1})$ satisfying
\begin{align}\label{TIExistEst1}
    \alpha^{N}_{j}(\hat{u},\widehat{\boldsymbol{S}}) \leq (1 + C^{*}\Delta T^{N}) \alpha^{N}_{j-1}.
\end{align}
Now suppose that $\hat{u} \in \BV(\Omega;\R^{3})$ and $\widehat{\boldsymbol{S}} \in \Slip_{1/n}(\boldsymbol{T}^{N}_{j-1})$ additionally satisfy \eqref{TIExistEst1}. We apply Corollary \ref{cr:MainCoercivity} with 
\begin{align*}
    t&:= t^{N}_{j}, & (u,p,\boldsymbol{T}) &:= (\hat{u},\widehat{\boldsymbol{S}}_{\gg}p^{N}_{j-1},\partial^{+}\widehat{\boldsymbol{S}}), & \boldsymbol{S} &:= (\widehat{\boldsymbol{S}} \circ\boldsymbol{S}^{N}_{j-1} \circ \cdots \circ \boldsymbol{S}^{N}_{0})^{-1},
\end{align*}
to obtain
\begin{align}\label{TIExistEst2}
    \Vert \hat{u} \Vert_{\BV} + \Vert \symfr[D\hat{u}-\widehat{\boldsymbol{S}}_{\gg}p^{N}_{j-1}] \Vert_{\Lrm^{2}} + \Diss(\widehat{\boldsymbol{S}}) \leq C_{\lambda} \cdot \alpha^{N}_{j}(\hat{u},\widehat{\boldsymbol{S}}).
\end{align}
Combining \eqref{TIExistEst1} and \eqref{TIExistEst2} yields \eqref{TIBound} with $C(\alpha^{N}_{j-1}) = C_{\lambda}(1 + C^{*}T)\alpha^{N}_{j-1}$.
\end{proof}

Take a minimising sequence
\begin{align*}
    ((\hat{u}_{l},\widehat{\boldsymbol{S}}_{l}))_{l} \subset \BV(\Omega;\R^{3}) \times \Slip_{1/n}(\boldsymbol{T}^{N}_{j-1})
\end{align*}
for \eqref{itm:IP-ii} at $i=j$ satisfying \eqref{TestCondIP-ii} and \eqref{TIBound}. Take $l \in \N$. Notice that, after possibly rescaling $\widehat{\boldsymbol{S}}_{l}$ using Lemma \ref{lm:Steadiness}, we can use \eqref{TIBound} to deduce that $(\Lip(\widehat{\boldsymbol{S}}_{l}))_{l}$ is uniformly bounded without loss of generality. It follows from Lemma \ref{lm:PlasticRescale}, Proposition \ref{pp:DissMaxCoreRescale}, and \eqref{EndpointsRescale} that, in the course of this rescaling procedure, $\Vert \widehat{\boldsymbol{S}}_{l} \Vert_{\Lrm^{\infty}_{c}}$ and
\begin{align*}
    \Ecal(t^{N}_{j},\hat{u}_{l},(\widehat{\boldsymbol{S}}_{l})_{\gg}p^{N}_{j-1},\partial^{+}\widehat{\boldsymbol{S}}_{l})+\Diss(\widehat{\boldsymbol{S}}_{l})
\end{align*}
do not change. We can now appeal to \cite[Theorem 3.23]{AmbrosioFuscoPallara00book} and Proposition \ref{pp:SlipCompactness} to deduce that there are 
\begin{align*}
    u_{*} &\in \BV(\Omega;\R^{3}), & \boldsymbol{S}_{*} &\in \Slip_{1/n}(\boldsymbol{T}^{N}_{j-1}),
\end{align*}
and a (not relabelled) subsequence of $(\hat{u}_{l},\widehat{\boldsymbol{S}}_{l})$ such that
\begin{align*}
    \hat{u}_{l} &\wkto u_{*} \hspace{1em} \text{in $\BV(\Omega;\R^{3})$}, &
    \widehat{\boldsymbol{S}}_{l} &\wkto \boldsymbol{S}_{*} \hspace{1em} \text{in $\Slip_{1/n}(\boldsymbol{T}^{N}_{j-1})$},
\end{align*}
and
\begin{align}\label{TIExist1}
    \Diss(\boldsymbol{S}_{*}) &\leq \liminf_{l \to \infty} \Diss(\widehat{\boldsymbol{S}}_{l}).
\end{align}
We can then use Lemmas \ref{lm:EndpointsCty}, \ref{lm:pCty} to deduce that
\begin{align*}
    \partial^{+}\widehat{\boldsymbol{S}}_{l} &\wkto \partial^{+}\boldsymbol{S}_{*}, &
    (\widehat{\boldsymbol{S}}_{l})_{\gg}p^{N}_{j-1} &\wkto (\boldsymbol{S}_{*})_{\gg}p^{N}_{j-1}.
\end{align*}
It follows from Lemma \ref{lm:CoreEndpointsByMaxCore} and \eqref{TestCondIP-ii} that
\begin{align*}
    \Wcal_{c}(\partial^{+}\widehat{\boldsymbol{S}}_{l}) \leq \Vert \widehat{\boldsymbol{S}}_{l} \Vert_{\Lrm^{\infty}_{c}} \leq M.
\end{align*}
Since \eqref{TIBound} also holds, we can use Lemma \ref{lm:QClosure} to deduce that
\begin{align*}
    (u_{*},(\boldsymbol{S}_{*})_{\gg}p^{N}_{j-1},\partial^{+}\boldsymbol{S}_{*}) \in \Qcal_{1/n}.
\end{align*}
We then appeal to Lemma \ref{lm:ELSC} and \eqref{TIExist1} to pass to a lower limit and deduce that $(u_{*},\boldsymbol{S}_{*})$ is a minimiser of the problem \eqref{itm:IP-ii} at $i=j$. The result follows after setting $u^{N}_{j} := u_{*}$, $p^{N}_{j} := (\boldsymbol{S}_{*})_{\gg}p^{N}_{j-1}$, $\boldsymbol{T}^{N}_{j} := \partial^{+}\boldsymbol{S}_{*}$, and $\boldsymbol{S}^{N}_{j} := \boldsymbol{S}_{*}$.
\end{proof}

\begin{proposition}\label{pp:TIDiscProperties} Take $N \in \N$ and let $\{(u^{N}_{j},p^{N}_{j},\boldsymbol{T}^{N}_{j},\boldsymbol{S}^{N}_{j})\}_{j=0}^{N}$ be a $\Delta T^{N}$-time-incremental solution starting from $(u_{0},p_{0},\boldsymbol{T}_{0})$. For $j \in \{1,\dotsc,N\}$, the following conditions hold:
\begin{enumerate}
    \item[(i)] For all $\hat{u} \in \BV(\Omega;\R^{3})$ and all $\widehat{\boldsymbol{S}} \in \Slip_{1/n}(\boldsymbol{T}^{N}_{j})$ with
    \begin{align}\label{TIDiscStabilityConds}
        (\hat{u},\widehat{\boldsymbol{S}}_{\gg}p^{N}_{j},\partial^{+}\widehat{\boldsymbol{S}}) &\in \Qcal_{1/n}, & \Vert \widehat{\boldsymbol{S}} \Vert_{\Lrm^{\infty}_{c}} &\leq M,
    \end{align}
    we have
    \begin{align}\label{TIDiscStability}
        \Ecal(t^{N}_{j},u^{N}_{j},p^{N}_{j},\boldsymbol{T}^{N}_{j}) \leq \Ecal(t^{N}_{j},\hat{u},\widehat{\boldsymbol{S}}_{\gg}p^{N}_{j},\partial^{+}\widehat{\boldsymbol{S}}) + \Diss(\widehat{\boldsymbol{S}}).
    \end{align}
    \item[(ii)] We have
    \begin{align}\label{TIDiscLowerEnergy}
        \Ecal(t^{N}_{j},u^{N}_{j},p^{N}_{j},\boldsymbol{T}^{N}_{j}) \leq \Ecal(0,u_{0},p_{0},\boldsymbol{T}_{0})-\sum_{i=0}^{j} \Diss(\boldsymbol{S}^{N}_{i})-\sum_{i=1}^{j}\int_{t^{N}_{i-1}}^{t^{N}_{i}}\langle \dot{f}(s),u^{N}_{i-1}\rangle \dd s.
    \end{align}
\end{enumerate}
\end{proposition}

\begin{proof} For $i \in \{0,\dots,N\}$, we shall use the following abbreviations:
\begin{align*}
    e^{N}_{i} &:= \Ecal(t^{N}_{i},u^{N}_{i},p^{N}_{i},\boldsymbol{T}^{N}_{i}), & d^{N}_{i} &:= \Diss(\boldsymbol{S}^{N}_{i}).
\end{align*}
Take $\hat{u} \in \BV(\Omega;\R^{3})$ and $\widehat{\boldsymbol{S}} \in \Slip_{1/n}(\boldsymbol{T}^{N}_{j})$ satisfying \eqref{TIDiscStabilityConds}. We test \eqref{itm:IP-ii} at $i=j$ with $\hat{u}$ and $\widehat{\boldsymbol{S}} \circ \boldsymbol{S}^{N}_{j}$ to obtain
\begin{align*}
    e^{N}_{j}+d^{N}_{j} &\leq \Ecal(t^{N}_{j},\hat{u},(\widehat{\boldsymbol{S}} \circ \boldsymbol{S}^{N}_{j})_{\gg}p^{N}_{j-1},\partial^{+}(\widehat{\boldsymbol{S}} \circ \boldsymbol{S}^{N}_{j})) + \Diss(\widehat{\boldsymbol{S}} \circ \boldsymbol{S}^{N}_{j}), \\
    &\leq \Ecal(t^{N}_{j},\hat{u},\widehat{\boldsymbol{S}}_{\gg}p^{N}_{j},\partial^{+}\widehat{\boldsymbol{S}}) + \Diss(\widehat{\boldsymbol{S}}) + d^{N}_{j},
\end{align*}
from which \eqref{TIDiscStability} follows.

Now fix $j \in \{1,\dotsc,N\}$. For $i \in \{1,\dots,j\}$, we test with $(u^{N}_{i-1},[\![0,1]\!] \times \boldsymbol{T}^{N}_{i-1})$ \eqref{itm:IP-ii} to obtain
\begin{align*}
    e^{N}_{i}+d^{N}_{i} \leq \Ecal(t^{N}_{i},u^{N}_{i-1},p^{N}_{i-1},\boldsymbol{T}^{N}_{i-1}) = e^{N}_{i-1}-\int^{t^{N}_{i}}_{t^{N}_{i-1}}\langle \dot{f}(s),u^{N}_{i-1}\rangle \dd s.
\end{align*}
Summing these estimates for $i \in \{1,\dots,j\}$ and using the fact that $\Diss(\boldsymbol{S}^{N}_{0}) = 0$ yields \eqref{TIDiscLowerEnergy}.
\end{proof}

\subsection{Construction of Interpolants}\label{s:Interpolants} In what follows, we define interpolants corresponding to the previously constructed time-incremental solutions. Let $(u_{0},p_{0},\boldsymbol{T}_{0}) \in \Qcal_{1/n}$ and take $M > 0$ with $M \geq \Wcal_{c}(\boldsymbol{T}_{0})$. Take $N \in \N$ and let 
\begin{align*}
    \{(u^{N}_{j},p^{N}_{j},\boldsymbol{T}^{N}_{j},\boldsymbol{S}^{N}_{j})\}_{j=0}^{N}
\end{align*}
be a $\Delta T^{N}$-time-incremental solution starting from $(u_{0},p_{0},\boldsymbol{T}_{0})$. Define $\{\alpha^{N}_{j}\}_{j=0}^{N}$ as in \eqref{AlphaDefn}. Let $u_{N}:[0,T] \to \BV(\Omega;\R^{3})$ where, for $t \in [0,T]$ and $j \in \{0,\dotsc,N-1\}$,
\begin{align}\label{uNDefn}
    u_{N}(t) &:= u^{N}_{j} \hspace{1em} \text{whenever $t^{N}_{j} \leq t < t^{N}_{j+1}$}.
\end{align}
and $u_{N}(T) := u^{N}_{N}$. Notice that $(u^{N})^{-1}(A)$ is $\mathscr{L}^{1}$-measurable for any $A \subset \BV(\Omega;\R^{3})$ and so $u^{N}$ is weakly$^{*}$ $\mathscr{L}^{1}$-measurable. Similarly, we define
\begin{align}\label{SNDefn}
    \boldsymbol{S}_{N} := \sum_{j=1}^{N}[\![t^{N}_{j-1},t^{N}_{j}]\!] \times \boldsymbol{T}^{N}_{j-1} + \delta_{t^{N}_{j}} \times \pbf_{*}\boldsymbol{S}^{N}_{j}
\end{align}
where, for $j \in \{1,\dotsc,N\}$,
\begin{align*}
    \delta_{t^{N}_{j}} \times \pbf_{*}\boldsymbol{S}^{N}_{j} := (\delta_{t^{N}_{j}} \times \pbf_{*}(S^{N}_{j})^{b})_{b \in \Bcal}.
\end{align*}
Notice that $\boldsymbol{S}_{N} \in \BV([0,T];\Disl_{1/n}(\overline{\Omega}))$ and that $ \partial^{-}\boldsymbol{S}_{N} = \boldsymbol{T}_{0}$. Finally, we define 
\begin{align}\label{pNDefn}
    p_{N} := (p_{0})_{\boldsymbol{S}_{N}}
\end{align}
In particular, we have
\begin{align*}
    p_{N}(t) = p_{0} + \frac{1}{2} \sum_{b \in \Bcal} b \otimes \star\pbf_{*}((S_{N})^{b} \mvert [0,t] \times \R^{3}) \hspace{1em} \text{for $t \in [0,T]$}.
\end{align*}

We shall now discuss some properties of the interpolants that are inherited from the time-incremental solutions. We start by defining
\begin{align}\label{Partitions}
    \Pi_{N} := \{t^{N}_{j}: j \in \{0,\dotsc,N\}\}.
\end{align}
For $t \in (0,T)\backslash \Pi_{N}$, there is a unique $j(t,N) \in \{0,\dotsc,N-1\}$ such that
\begin{align*}
    t^{N}_{j(t,N)} < t < t^{N}_{j(t,N)+1}.
\end{align*}
Fix $t \in (0,T)\backslash \Pi_{N}$. It follows from \eqref{uNDefn}, \eqref{pNDefn}, and \eqref{SNDefn} that
\begin{align}\label{InterpolantSlices}
    u_{N}(t) &= u^{N}_{j(t,N)}, & p_{N}(t) &= p^{N}_{j(t,N)} & \boldsymbol{S}_{N}(t) &= \boldsymbol{T}^{N}_{j(t,N)}
\end{align}
Now additionally suppose that $j(t,N) > 0$ (since $t > 0$, this is true for sufficiently large $N$). It follows directly from the first conclusion of Proposition \ref{pp:TIDiscProperties} that, for $\hat{u} \in \BV(\Omega;\R^{3})$ and $\widehat{\boldsymbol{S}} \in \Slip_{1/n}(\boldsymbol{S}_{N}(t))$ with 
\begin{align*}
    (\hat{u},\widehat{\boldsymbol{S}}_{\gg}p_{N}(t),\partial^{+}\widehat{\boldsymbol{S}}) &\in \Qcal_{1/n}, & \Vert \widehat{\boldsymbol{S}} \Vert_{\Lrm^{\infty}_{c}} &\leq M,
\end{align*}
we have
\begin{align}\label{InterpolantStability}
    \Ecal(t^{N}_{j(t,N)},u_{N}(t),p_{N}(t),\boldsymbol{S}_{N}(t)) \leq \Ecal(t^{N}_{j(t,N)},\hat{u},\widehat{\boldsymbol{S}}_{\gg}p_{N}(t),\partial^{+}\widehat{\boldsymbol{S}}) + \Diss(\widehat{\boldsymbol{S}}).
\end{align}
Similarly, it follows directly from the second conclusion of Proposition \ref{pp:TIDiscProperties} that
\begin{equation}\label{InterpolantLowerEnergy0}
\begin{aligned}
    \Ecal(t^{N}_{j(t,N)},u_{N}(t),p_{N}(t),\boldsymbol{S}_{N}(t)) &- \Ecal(0,u_{0},p_{0},\boldsymbol{T}_{0}) \vphantom{} \\
    &\leq -\sum_{i=0}^{j(t,N)}\Diss(\boldsymbol{S}^{N}_{i}) - \int_{0}^{t^{N}_{j(t,N)}}\langle \dot{f}(s),u_{N}(s)\rangle \dd s.
\end{aligned}
\end{equation}
For $j \in \{0,\dotsc,N\}$, we can use \cite[Lemma 2.5]{BonicattoRindlerTurnbull26-I?} and \cite[Lemma 2.6]{BonicattoRindlerTurnbull26-I?} to deduce that
\begin{align}\label{DissJump}
    \Diss(\delta_{t^{N}_{j}} \times \pbf_{*}\boldsymbol{S}^{N}_{j}) &\leq \Diss(\boldsymbol{S}^{N}_{j}).
\end{align}
We then use \eqref{SNDefn}, \eqref{DissJump}, and the fact that $\Diss(\boldsymbol{S}^{N}_{0})=0$ to deduce that
\begin{align}\label{DissInterpolants}
    \Diss(\boldsymbol{S}_{N};[0,t]) &\leq \sum_{i=0}^{j(t,N)}\Diss(\boldsymbol{S}_{i}^{N}) \hspace{1em} \text{for $t \in (0,T)\backslash \Pi_{N}$}.
\end{align}
Again, suppose that $j(t,N) > 0$. Combining the dissipation estimate above with \eqref{InterpolantLowerEnergy0} yields
\begin{equation}\label{InterpolantLowerEnergy}
\begin{aligned}
    \Ecal(t^{N}_{j(t,N)},u_{N}(t),p_{N}(t),\boldsymbol{S}_{N}(t)) &- \Ecal(0,u_{0},p_{0},\boldsymbol{T}_{0}) \vphantom{} \\
    &\leq -\Diss(\boldsymbol{S}_{N};[0,t]) - \int_{0}^{t^{N}_{j(t,N)}}\langle \dot{f}(s),u_{N}(s)\rangle \dd s.
\end{aligned}
\end{equation}

We conclude this section with uniform estimates on $u_{N}$, $p_{N}$, and $\boldsymbol{S}_{N}$:

\begin{lemma}\label{lm:InterpolantAP} There is a constant $\overline{C} > 0$, depending only on the data in Section \ref{s:Assumptions} and the constant $\lambda > 0$ from \eqref{A:3}, such that, for $N \in \N$,
\begin{align}\label{cc:InterpolantAP1}
    \Var(p_{N}) + \Diss(\boldsymbol{S}_{N}) + \Wcal_{c}(\partial^{-}\boldsymbol{S}_{N}) + \Wcal_{c}(\partial^{+}\boldsymbol{S}_{N}) \leq \overline{C}\alpha_{0}
\end{align}
and, for all $t \in (0,T)\backslash\Pi_{N}$,
\begin{align}\label{cc:InterpolantAP2}
    \Vert u_{N}(t) \Vert_{\BV} + \Vert \symfr[Du_{N}(t)-p_{N}(t)]\Vert_{\Lrm^{2}} + \Vert p_{N}(t) \Vert(\overline{\Omega}) + \Wcal_{c}(\boldsymbol{S}_{N}(t)) &\leq \overline{C}\alpha_{0}
\end{align}
where $\alpha_{0}$ is defined as in \eqref{Alpha0Defn}.
\end{lemma}

\begin{proof} Fix $N \in \N$. It follows from Lemma \ref{lm:PlasticPathBV} and \eqref{DissInterpolants} that
\begin{align}\label{InterpolantAP0}
    \Var(p_{N}) + \Diss(\boldsymbol{S}_{N}) \leq C \cdot \sum_{j=0}^{N} \Diss(\boldsymbol{S}^{N}_{j}).
\end{align}
For $i \in \{1,\dotsc,N\}$, we define $\alpha^{N}_{i}$ as in \eqref{AlphaDefn}; notice that $\alpha^{N}_{0} = \alpha_{0}$. Take $j \in \{0,\dotsc,N\}$. We apply Corollary \ref{cr:MainCoercivity} with
\begin{align*}
    t &:= t^{N}_{j}, & (u,p,\boldsymbol{T}) &:= (u^{N}_{j},p^{N}_{j},\boldsymbol{T}^{N}_{j}), & \boldsymbol{S} &:= (\boldsymbol{S}_{j}^{N} \circ \cdots \circ \boldsymbol{S}_{0}^{N})^{-1},
\end{align*}
to obtain
\begin{equation}\label{InterpolantAP1}
\begin{aligned}
    \Vert u^{N}_{j} \Vert_{\BV} + \Vert \symfr[Du^{N}_{j}-p^{N}_{j}&] \Vert_{\Lrm^{2}} + \Vert p_{0}^{N} \Vert(\overline{\Omega}) + \Wcal_{c}(\boldsymbol{T}^{N}_{0}) \\
    &+ \sum_{i=0}^{j} \Diss(\boldsymbol{S}_{i}^{N})+\Wcal_{c}(\boldsymbol{T}^{N}_{j}) \leq C_{\lambda} \alpha^{N}_{j}.
\end{aligned}
\end{equation}
where $C_{\lambda} > 0$ is given by Corollary \ref{cr:MainCoercivity}. It follows from Lemma \ref{lm:AlphaAP} that 
\begin{align}\label{InterpolantAP2}
    \frac{\alpha^{N}_{i}-\alpha^{N}_{i-1}}{\Delta T^{N}} &\leq C^{*} \cdot \alpha^{N}_{i-1} \hspace{1em} \text{for $i \in \{1,\dotsc,j\}$}.
\end{align}
Iterating \eqref{InterpolantAP2} yields
\begin{align}\label{InterpolantAP3}
    \alpha^{N}_{j} \leq (1+C^{*}\Delta T^{N})^{j}\alpha^{N}_{0} \leq e^{C^{*}T}\alpha^{N}_{0}.
\end{align}
Since $\partial^{-}\boldsymbol{S}_{N} = \boldsymbol{T}_{0}^{N}$ and $\partial^{+}\boldsymbol{S}_{N} = \boldsymbol{T}_{N}^{N}$, we obtain \eqref{cc:InterpolantAP1} by combining \eqref{InterpolantAP0}, \eqref{InterpolantAP1} with $j=N$, and \eqref{InterpolantAP3}. We can also use Lemma \ref{lm:PlasticPathBV} to deduce that
\begin{align}\label{InterpolantAP4}
    \Vert p^{N}_{j} \Vert(\overline{\Omega}) \leq \Vert p_{0} \Vert(\overline{\Omega}) + C\sum_{i=0}^{j}\Diss(\boldsymbol{S}^{N}_{i}) \hspace{1em} \text{for $j \in \{0,\dotsc,N\}$}.
\end{align}
Take $t \in (0,T)\backslash \Pi_{N}$. The estimate \eqref{cc:InterpolantAP2} follows from \eqref{InterpolantSlices} after combining \eqref{InterpolantAP4} and \eqref{InterpolantAP1} (at $j=j(t,N)$) with \eqref{InterpolantAP3}.
\end{proof}

\subsection{An Euler-Lagrange Equation for the Displacement} We now consider the dependence of the minimising displacement of $\Ecal$ on the corresponding plastic distortion and dislocation system.

\begin{lemma}\label{lm:uEL} Fix $t \in [0,T]$ and take $(u_{*},p,\boldsymbol{T}) \in \Qcal$. Then $u_{*} \in \BV(\Omega;\R^{3})$ is a minimiser of $\hat{u} \mapsto \Ecal(t,\hat{u},p,\boldsymbol{T})$ over all $\hat{u} \in \BV(\Omega;\R^{3})$ with $(\hat{u},p,\boldsymbol{T}) \in \Qcal$ if and only if
\begin{align}\label{ELu}
    \int_{\Omega} \Ebb(\fr[Du_{*}-p]) : D \varphi \dd x= \int_{\Omega} f(t) \cdot \varphi \dd x
\end{align}
for all $\varphi \in \Wrm^{1,2}(\Omega;\R^{3})$ with
\begin{align}\label{uELTestConds}
    [\varphi]_{H} &=0, & \skew \int_{\Omega} \chi \dd D\varphi &= 0.
\end{align}
\end{lemma}

\begin{proof} First suppose that $u_{*} \in \BV(\Omega;\R^{3})$ is a minimiser of $\hat{u} \mapsto \Ecal(t,\hat{u},p,\boldsymbol{T})$ over all $\hat{u} \in \BV(\Omega;\R^{3})$ with $(\hat{u},p,\boldsymbol{T}) \in \Qcal$. Take $\varphi \in \Wrm^{1,2}(\Omega;\R^{3})$ satisfying \eqref{uELTestConds} and let $h > 0$. First notice that
\begin{align*}
    \fr[D(u_{*}+h\varphi)-p] = \fr[Du_{*}-p]+hD\varphi.
\end{align*}
Since $\varphi \in \Wrm^{1,2}(\Omega;\R^{3})$ and $\symfr[Du_{*}-p] \in \Lrm^{2}(\Omega;\R^{3 \times 3})$, we have 
\begin{align*}
    \symfr[D(u_{*}+h\varphi)-p] \in \Lrm^{2}(\Omega;\R^{3 \times 3}).
\end{align*}
Furthermore, since \eqref{uELTestConds} holds and $(p,\boldsymbol{T})$ is consistent, we have
\begin{align*}
    (u_{*}+h\varphi,p,\boldsymbol{T}) \in \Qcal.
\end{align*}
We therefore have $\Ecal(t,u_{*},p,\boldsymbol{T}) \leq \Ecal(t,u_{*}+h\varphi,p,\boldsymbol{T})$ and so
\begin{align*}
    0 &\leq \frac{1}{2}\int_{\Omega} \vert \fr[Du_{*}-p]+hD\varphi\vert_{\Ebb}^{2} - \vert \fr[Du_{*}-p]\vert_{\Ebb}^{2} \dd x - \int_{\Omega} h f(t) \cdot \varphi \dd x\\
    &= h \int_{\Omega} \Ebb(\fr[Du_{*}-p]):D\varphi + \frac{h}{2} \vert D\varphi \vert_{\Ebb}^{2} - f(t) \cdot \varphi \dd x
\end{align*}
We then divide by $h$, take $h \downarrow 0$, and use $\pm \varphi$ in place of $\varphi$ to obtain \eqref{ELu}. 

For the converse, take $\hat{u} \in \BV(\Omega;\R^{3})$ with $(\hat{u},p,\boldsymbol{T}) \in \Qcal$, and set $\varphi := \hat{u} - u_{*}$. Then \eqref{uELTestConds} holds and
\begin{align*}
    \sym\hspace{0.2em}D\varphi = \symfr[D\hat{u}-p] - \symfr[Du_{*}-p] \in \Lrm^{2}(\Omega;\R^{3 \times 3}).
\end{align*}
We can therefore use Lemma \ref{lm:Korn} to deduce that $\varphi \in \Wrm^{1,2}(\Omega;\R^{3})$ and so
\begin{align}\label{ELu1}
    \int_{\Omega} f(t) \cdot (\hat{u} - u_{*}) \dd x= \int_{\Omega} f(t) \cdot \varphi \dd x = \int_{\Omega} \Ebb(\fr[Du_{*}-p]) : D \varphi \dd x.
\end{align}
Finally, notice that
\begin{align}\label{ELu2}
    \int_{\Omega} \Ebb(\fr[Du_{*}-p]) : D \varphi \dd x\leq \frac{1}{2}\int_{\Omega} \vert \fr[D\hat{u}-p]\vert^{2}_{\Ebb}-\vert \fr[Du_{*}-p]\vert^{2}_{\Ebb} \dd x.
\end{align}
The result follows by combining \eqref{ELu1} and \eqref{ELu2}.
\end{proof}

\begin{corollary}\label{cr:uUnique} Fix $t \in [0,T]$ and take $(u_{*},p,\boldsymbol{T}) \in \Qcal$. Suppose that $u_{*} \in \BV(\Omega;\R^{3})$ is a minimiser of $\hat{u} \mapsto \Ecal(t,\hat{u},p,\boldsymbol{T})$ over all $\hat{u} \in \BV(\Omega;\R^{3})$ with $(\hat{u},p,\boldsymbol{T}) \in \Qcal$. Then $u_{*}$ is unique (if it exists).
\end{corollary}

\begin{proof} Take $u_{*}^{1},u_{*}^{2} \in \BV(\Omega;\R^{3})$ with $(u_{*}^{i},p,\boldsymbol{T}) \in \Qcal$ for $i \in \{1,2\}$ and suppose that both $u_{*}^{1}$ and $u_{*}^{2}$ are minimisers of $\hat{u} \mapsto \Ecal(t,\hat{u},p,\boldsymbol{T})$ over all $\hat{u} \in \BV(\Omega;\R^{3})$ with $(\hat{u},p,\boldsymbol{T}) \in \Qcal$. Set $\varphi := u_{*}^{2}-u_{*}^{1}$ and notice that \eqref{uELTestConds} holds. One can work as in the proof of Lemma \ref{lm:uEL} to show that $\varphi \in \Wrm^{1,2}(\Omega;\R^{3})$. Subtracting \eqref{ELu} with $u_{*} := u^{1}_{*}$ from \eqref{ELu} with $u_{*} := u^{2}_{*}$ yields
\begin{align*}
    \int_{\Omega} \Ebb(D\varphi):D\varphi \dd x = 0.
\end{align*}
We therefore have $\sym\hspace{0.2em}D\varphi = 0$ on $\Omega$ and so we can use Lemma \ref{lm:Korn} to deduce that $D\varphi = 0$. It then follows from \eqref{Poincare} that $\varphi = 0$.
\end{proof}

Given $t \in [0,T]$, $p \in \Mcal_{\curl}(\overline{\Omega};\R^{3 \times 3})$, and $\boldsymbol{T} \in \overline{\Disl}(\overline{\Omega})$, we shall use
\begin{align*}
    u_{*}(t,p,\boldsymbol{T})
\end{align*} 
to denote the minimiser of $\hat{u} \mapsto \Ecal(t,\hat{u},p,\boldsymbol{T})$ over all $\hat{u} \in \BV(\Omega;\R^{3})$ with $(\hat{u},p,\boldsymbol{T}) \in \Qcal$ if it exists.

\begin{corollary}\label{cr:uCty} Let $(t_{j})_{j} \subset [0,T]$ and take  $((u_{j},p_{j},\boldsymbol{T}_{j}))_{j} \subset \Qcal$ with $u_{j} = u_{*}(t_{j},p_{j},\boldsymbol{T}_{j})$ for $j \in \N$. Suppose that there is $M > 0$ with
\begin{align}\label{uCtyBounds}
    \Vert u_{j} \Vert_{\BV} + \Vert \sym\fr[Du_{j}-p_{j}] \Vert_{\Lrm^{2}} + \Wcal_{c}(\boldsymbol{T}_{j}) &\leq M \hspace{1em} \text{for $j \in \N$.}
\end{align}
If there are $t \in [0,T]$, $p \in \Mcal(\overline{\Omega};\R^{3 \times 3})$, and $\boldsymbol{T} \in \overline{\Disl}(\overline{\Omega})$ with
\begin{align}\label{uCtyConv}
    t_{j} &\to t, & p_{j} &\wkto p, & \boldsymbol{T}_{j} \wkto \boldsymbol{T},
\end{align}
then $u_{*}(t,p,\boldsymbol{T})$ exists and $u_{j} \wkto u_{*}(t,p,\boldsymbol{T})$ in $\BV(\Omega;\R^{3})$.
\end{corollary}

\begin{proof} Take a (not relabelled) subsequence of $(u_{j})_{j}$. We first appeal to \cite[Theorem 3.23]{AmbrosioFuscoPallara00book} to find $u \in \BV(\Omega;\R^{3})$ and a (not relabelled) subsequence of $(u_{j})_{j}$ with
\begin{align*}
    u_{j} \wkto u \hspace{1em} \text{in $\BV(\Omega;\R^{3})$}.
\end{align*}
It follows from \eqref{uCtyBounds}, \eqref{uCtyConv}, and Lemma \ref{lm:QClosure} that $(u,p,\boldsymbol{T}) \in \Qcal$. Furthermore, we have $\fr[Du_{j}-p_{j}] \wkto \fr[Du-p]$ in $\Mcal(\overline{\Omega};\R^{3 \times 3})$. We can then use \eqref{uCtyBounds} to deduce that
\begin{align}\label{uCty1}
    \symfr[Du_{j}-p_{j}] \rightharpoonup \symfr[Du-p] \hspace{1em} \text{in $\Lrm^{2}(\Omega;\R^{3 \times 3})$}.
\end{align}
Take $\varphi \in \Wrm^{1,2}(\Omega;\R^{3})$ satisfying \eqref{uELTestConds}. It follows from Lemma \ref{lm:uEL} that
\begin{align}\label{uCty2}
    \int_{\Omega} \Ebb(\fr[Du_{j}-p_{j}]) : D \varphi \dd x= \int_{\Omega} f(t_{j}) \cdot \varphi \dd x
\end{align}
Taking $j \to \infty$ in \eqref{uCty2}, using \eqref{uCty1} and \eqref{A:2-iii}, yields 
\begin{align*}
    \int_{\Omega} \Ebb(\fr[Du-p]) : D \varphi \dd x= \int_{\Omega} f(t) \cdot \varphi \dd x
\end{align*}
We can then appeal to Lemma \ref{lm:uEL} to deduce that $u$ is a minimiser of $\hat{u} \mapsto \Ecal(t,\hat{u},p,\boldsymbol{T})$ over all $\hat{u} \in \BV(\Omega;\R^{3})$ with $(\hat{u},p,\boldsymbol{T}) \in \Qcal$. It follows from Corollary \ref{cr:uUnique} that $u=u_{*}(t,p,\boldsymbol{T})$ and the result follows since the initial choice of subsequence was arbitrary.
\end{proof}

Take $p \in \Mcal_{\curl}(\overline{\Omega};\R^{3 \times 3})$ and let $\boldsymbol{T} \in \Disl_{1/n}(\overline{\Omega})$. Then, for all $u \in \BV(\Omega;\R^{3})$, we have $(u,p,\boldsymbol{T}) \in \Qcal$ if and only if $(u,p,\boldsymbol{T}) \in \Qcal_{1/n}$ and so we can replace $\Qcal$ with $\Qcal_{1/n}$ in the results above in this case.

\subsection{The Limit \texorpdfstring{$N \to \infty$}{}}\label{s:Ninfty} Let $(u_{0},p_{0},\boldsymbol{T}_{0}) \in \Qcal_{1/n}$ and take $M > 0$ with $M \geq \Wcal_{c}(\boldsymbol{T}_{0})$. For $N \in \N$, we can work as in Section \ref{s:Interpolants} to define
\begin{align*}
    u_{N} &: [0,T] \to \BV(\Omega;\R^{3}), \\
    p_{N} &: [0,T] \to \Mcal(\overline{\Omega};\R^{3 \times 3}), \\
    \boldsymbol{S}_{N} &\in \BV([0,T];\Disl_{1/n}(\overline{\Omega})).
\end{align*}
Let $\overline{C} > 0$ be the constant from Lemma \ref{lm:InterpolantAP}. Since $\overline{C}$ does not depend on $M$, we can choose $M > 0$ so that
\begin{align}\label{MCond}
     M \geq \max\{\Wcal_{c}(\boldsymbol{T}_{0}), (C+1)\cdot\overline{C}\alpha_{0}\},
\end{align}
where $C > 0$ is the constant from Proposition \ref{pp:AppliedIntegralSTC} and $\alpha_{0}$ is defined in \eqref{Alpha0Defn}. Recall from Lemma \ref{lm:InterpolantAP} that
\begin{align*}
   \Var(p_{N}) + \Vert \boldsymbol{S}_{N} \Vert_{\Lrm^{\infty}_{c}} + \Diss(\boldsymbol{S}_{N}) + \Wcal_{c}(\partial^{-}\boldsymbol{S}_{N}) + \Wcal_{c}(\partial^{+}\boldsymbol{S}_{N}) \leq 2\overline{C}\alpha_{0} \hspace{1em} \text{for $N \in \N$}.
\end{align*}
We can therefore appeal to Proposition \ref{pp:BVHelly} and Proposition \ref{pp:BVCompactness} to obtain (not relabelled) subsequences of $(p_{N})_{N}$ and $(\boldsymbol{S}_{N})_{N}$ and
\begin{align*}
    p &\in \BV([0,T];\Mcal(\overline{\Omega};\R^{3 \times 3})), \\
    \boldsymbol{S} &\in \BV([0,T];\Disl_{1/n}(\overline{\Omega})),
\end{align*}
such that
\begin{align}\label{NinftySpConv}
    p_{N}(t) &\wkto p(t) \hspace{1em} \text{for $t \in [0,T]$}, & \boldsymbol{S}_{N} &\wkto \boldsymbol{S},
\end{align}
and
\begin{align}\label{NinftyDissVarLSC}
    \Var(p) &\leq \liminf_{N \to \infty} \Var(p_{N}), & \Diss(\boldsymbol{S}) &\leq \liminf_{N \to \infty} \Diss(\boldsymbol{S}_{N}).
\end{align}
Furthermore, for $\mathscr{L}^{1}$-a.e.\ $t \in [0,T]$, we have
\begin{align}\label{NinftySliceConv}
    \tbf_{\#}\Vert \boldsymbol{S} \Vert(\{t\}) &= 0, & \boldsymbol{S}_{N}(t) &\wkto \boldsymbol{S}(t).
\end{align}
and
\begin{align}\label{NinftySRestrictConv}
    \boldsymbol{S}_{N} \mvert [0,t] &\wkto \boldsymbol{S} \mvert [0,t], & \Diss(\boldsymbol{S};[0,t]) &\leq \liminf_{N \to \infty} \Diss(\boldsymbol{S}_{N};[0,t]).
\end{align}
Since $\boldsymbol{S}_{N} \wkto \boldsymbol{S}$, we can use Lemma \ref{lm:EndpointsCty} to deduce that $\partial^{\pm}\boldsymbol{S}_{N} \wkto \partial^{\pm}\boldsymbol{S}$. It then follows from \eqref{CoreLSC} that
\begin{align}\label{NinftyEndpointsCoreLSC}
    \Wcal_{c}(\partial^{\pm}\boldsymbol{S}) &\leq \liminf_{N \to \infty} \Wcal_{c}(\partial^{\pm}\boldsymbol{S}_{N}).
\end{align}
Furthermore, since $\partial^{-}{\boldsymbol{S}}_{N} = \boldsymbol{T}_{0}$ for $N \in \N$, we have $\partial^{-}\boldsymbol{S} = \boldsymbol{T}_{0}$. Similarly, since $p_{N}(0) = p_{0}$ for $N \in \N$, we have $p(0) = p_{0}$. Let $J \subset [0,T]$ be the null set on which \eqref{NinftySliceConv} and \eqref{NinftySRestrictConv} fail. After possibly enlarging $J$, we can assume that
\begin{align*}
    \Pi_{N} \cap [0,T] \subset J \hspace{1em} \text{for $N \in \N$},
\end{align*}
where $\Pi_{N}$ is defined in \eqref{Partitions}. Fix $t \in [0,T]\backslash J$ and take $N \in \N$ with $j(t,N) > 0$. It follows from Lemma \ref{lm:InterpolantAP} that
\begin{align*}
    \Vert [\![0,1]\!] \times \boldsymbol{S}_{N}(t) \Vert_{\Lrm^{\infty}_{c}} = \Wcal_{c}(\boldsymbol{S}_{N}(t)) \leq \overline{C}\alpha_{0} \leq M,
\end{align*}
so we can test with $\widehat{\boldsymbol{S}} := [\![0,1]\!] \times \boldsymbol{S}_{N}(t)$ in \eqref{InterpolantStability} and deduce from Corollary \ref{cr:uUnique} that
\begin{align*}
    u_{N}(t) = u_{*}(t^{N}_{j(t,N)},p_{N}(t),\boldsymbol{S}_{N}(t)).
\end{align*}
We can therefore appeal to the convergence of minimisers expressed in Corollary \ref{cr:uCty} (taking into account Lemma \ref{lm:InterpolantAP}) to deduce that $u(t) := u_{*}(t,p(t),\boldsymbol{S}(t))$ exists and that $u_{N}(t) \wkto u(t)$. In particular, we have defined
\begin{align*}
    u: [0,T]\backslash J \to \BV(\Omega;\R^{3})
\end{align*}
and we have
\begin{align}\label{NinftyuConv}
    u_{N}(t) \wkto u(t) \hspace{1em} \text{for $t \in [0,T]\backslash J$}.
\end{align}
It follows from Lemma \ref{lm:MeasurableCty} that $u$ is weakly$^{*}$ $\mathscr{L}^{1}$-measurable. 

Together with Lemma \ref{lm:InterpolantAP}, the conclusions above allow us to establish the following estimates:

\begin{lemma}\label{lm:NinftyAP} Let $\overline{C} > 0$ be the ($M$-independent) constant from Lemma \ref{lm:InterpolantAP}. Then
\begin{align}\label{cc:NinftyAP1}
    \Var(p) + \Diss(\boldsymbol{S}) + \Wcal_{c}(\partial^{-}\boldsymbol{S}) + \Wcal_{c}(\partial^{+}\boldsymbol{S}) \leq \overline{C}\alpha_{0}
\end{align}
and, for all $t \in [0,T] \backslash J$,
\begin{equation}\label{cc:NinftyAP2}
\begin{aligned}
    \Vert u(t) \Vert_{\BV} + \Vert \symfr[Du(t)-p(t)]\Vert_{\Lrm^{2}} + \Vert p(t) \Vert(\overline{\Omega}) + \Wcal_{c}(\boldsymbol{S}(t)) &\leq \overline{C}\alpha_{0}
\end{aligned}
\end{equation}
where $\alpha_{0}$ is defined as in \eqref{Alpha0Defn}.
\end{lemma}

It follows from Lemma \ref{lm:NinftyAP} that $u \in \Lrm^{\infty}([0,T];\BV(\Omega;\R^{3}))$. We shall now show that the triple $(u,p,\boldsymbol{S})$ enjoys several properties that are weaker analogues of the conditions of Definition \ref{df:DiscreteEnergetic}.

\begin{proposition}\label{pp:NinftyProperties} Fix $t \in (0,T)\backslash J$. Then the following conclusions hold:
\begin{enumerate}
    \item[(i)] We have
    \begin{align*}
        (u(t),p(t),\boldsymbol{S}(t)) \in \Qcal_{1/n}.
    \end{align*}
    \item[(ii)] For all $\hat{u} \in \BV(\Omega;\R^{3})$ and all $\widehat{\boldsymbol{S}} \in \Slip_{1/n}(\boldsymbol{S}(t))$ with
    \begin{align}\label{PartialStabilityConds}
        (\hat{u},\widehat{\boldsymbol{S}}_{\gg}p(t),\partial^{+}\widehat{\boldsymbol{S}}) &\in \Qcal_{1/n}, & \Vert \widehat{\boldsymbol{S}} \Vert_{\Lrm^{\infty}_{c}} &\leq M,
    \end{align}
    we have
    \begin{align*}
        \Ecal(t,u(t),p(t),\boldsymbol{S}(t)) \leq \Ecal(t,\hat{u},\widehat{\boldsymbol{S}}_{\gg}p(t),\partial^{+}\widehat{\boldsymbol{S}}) + \Diss(\widehat{\boldsymbol{S}}).
    \end{align*}
    \item[(iii)]  We have
    \begin{align}\label{NinftyLowerEnergy}
        \Ecal(t,&u(t),p(t),\boldsymbol{S}(t)) - \Ecal(0,u_{0},p_{0},\boldsymbol{T}_{0}) + \Diss(\boldsymbol{S};[0,t]) +\int_{0}^{t}\langle \dot{f}(s),u(s)\rangle \dd s \leq 0.
    \end{align}
    \item[(iv)] We have
    \begin{align*}
        p(t) = p_{0} + \frac{1}{2}\sum_{b \in \Bcal}b \otimes \star\pbf_{*}(S^{b} \mvert [0,t] \times \R^{3}).
    \end{align*}
\end{enumerate}
\end{proposition}

We remark in passing that conclusions (i)-(iv) of Proposition \ref{pp:NinftyProperties} may also hold for some $t \in J$; in particular, a jump point for an element of the sequence $((u_{N},p_{N},\boldsymbol{S}_{N}))_{N}$ may not be a jump point for the limit $(u,p,\boldsymbol{S})$.

\begin{proof}

\begin{subproof}[(i)] This follows immediately from Lemma \ref{lm:QClosure} (taking into account Lemma \ref{lm:InterpolantAP}).
\end{subproof}

\begin{subproof}[(ii)] Since $\boldsymbol{S}_{N}(t) \wkto \boldsymbol{S}(t)$ and $\Wcal_{c}(\boldsymbol{S}_{N}(t)) \leq \overline{C}\alpha_{0}$ for $N \in \N$, we can appeal to Proposition \ref{pp:AppliedIntegralSTC} to yield, for all $N \in \N$ along a (not relabelled) subsequence, a
\begin{align*}
    \boldsymbol{R}_{N} \in \Slip_{1/n}(\boldsymbol{S}_{N}(t))
\end{align*}
with $\partial^{+}\boldsymbol{R}_{N} = \boldsymbol{S}(t)$ and
\begin{align}\label{NinftyConnectivity}
    \Diss(\boldsymbol{R}_{N}) &\to 0, & \limsup_{N \to \infty}  \Vert \boldsymbol{R}_{N} \Vert_{\Lrm^{\infty}_{c}} &\leq C \cdot \limsup_{N \to \infty} \Wcal_{c}(\boldsymbol{S}_{N}(t)).
\end{align}
It follows from Lemma \ref{lm:InterpolantAP} that 
\begin{align*}
    \limsup_{N \to \infty} \Vert \boldsymbol{R}_{N} \Vert_{\Lrm^{\infty}_{c}} \leq C \cdot \overline{C}\alpha_{0}.
\end{align*}
We can therefore ensure that $\Vert \boldsymbol{R}_{N} \Vert_{\Lrm^{\infty}_{c}} \leq (C+1) \cdot \overline{C}\alpha_{0} \leq M$ (see~\eqref{MCond}) along the subsequence without loss of generality. We shall additionally ensure that $j(t,N) > 0$ (so that we can apply \eqref{InterpolantStability}) along the subsequence. Take $\hat{u} \in \BV(\Omega;\R^{3})$ and $\widehat{\boldsymbol{S}} \in \Slip_{1/n}(\boldsymbol{S}(t))$ satisfying \eqref{PartialStabilityConds}. For $N \in \N$ along the subsequence, we define 
\begin{align*}
    \widehat{\boldsymbol{S}}_{N} := \widehat{\boldsymbol{S}} \circ \boldsymbol{R}_{N}
\end{align*}
and then appeal to Lemma \ref{lm:Corrector} to find $\hat{u}_{N} \in \BV_{h_{0},\chi}(\Omega;\R^{3})$ with 
\begin{align}\label{PartialStability1}
    \symfr[D\hat{u}_{N}-(\widehat{\boldsymbol{S}}_{N})_{\gg}p_{N}(t)] &= \symfr[D\hat{u} - \widehat{\boldsymbol{S}}_{\gg}p(t)].
\end{align}
Notice that $\widehat{\boldsymbol{S}}_{N} \in \Slip_{1/n}(\boldsymbol{S}_{N}(t))$. It follows from \eqref{PartialStability1} and Lemma \ref{lm:FlowConsistent} that
\begin{align*}
(\hat{u}_{N},(\widehat{\boldsymbol{S}}_{N})_{\gg}p_{N}(t),\partial^{+}\widehat{\boldsymbol{S}}_{N}) \in \Qcal_{1/n}.
\end{align*}
We also have
\begin{align*}
    \Vert \widehat{\boldsymbol{S}}_{N} \Vert_{\Lrm^{\infty}_{c}} \leq \max\{\Vert \widehat{\boldsymbol{S}}\Vert_{\Lrm^{\infty}_{c}}, \Vert \boldsymbol{R}_{N}\Vert_{\Lrm^{\infty}_{c}}\} \leq M.
\end{align*}
We can therefore test with $(\hat{u}_{N},\widehat{\boldsymbol{S}}_{N})$ in \eqref{InterpolantStability} to obtain
\begin{equation}\label{PartialStability3}
\begin{aligned}
    \Ecal(t^{N}_{j(t,N)},u_{N}(t),p_{N}(t),\boldsymbol{S}_{N}(t)) 
    \leq \Ecal(t^{N}_{j(t,N)},\hat{u}_{N},(\widehat{\boldsymbol{S}}_{N})_{\gg}p_{N}(t),\partial^{+}\widehat{\boldsymbol{S}}_{N}) + \Diss(\widehat{\boldsymbol{S}}_{N}).
\end{aligned}
\end{equation}

Next, notice that $\partial^{+}\widehat{\boldsymbol{S}}_{N} = \partial^{+}\widehat{\boldsymbol{S}}$; combining this observation with \eqref{PartialStability1} and \eqref{EDefn} yields
\begin{equation}\label{PartialStability4}
\begin{aligned}
    \Ecal(t^{N}_{j(t,N)},&\hat{u}_{N},(\widehat{\boldsymbol{S}}_{N})_{\gg}p_{N}(t),\partial^{+}\widehat{\boldsymbol{S}}_{N}) \\
    &= \Ecal(t,\hat{u},\widehat{\boldsymbol{S}}_{\gg}p(t),\partial^{+}\widehat{\boldsymbol{S}}) + \langle f(t), \hat{u} \rangle - \langle f(t^{N}_{j(t,N)}),\hat{u}_{N} \rangle.
\end{aligned}
\end{equation}
Now, since $\Diss(\boldsymbol{R}_{N}) \to 0$ and $p_{N}(t) \wkto p(t)$, we can use Lemma \ref{lm:pCtyVanishingVar} to deduce that $ (\boldsymbol{R}_{N})_{\gg}p_{N}(t) \wkto p(t)$. It then follows from Lemma \ref{lm:pCty} and Lemma \ref{lm:PlasticConcatenation} that
\begin{align*}
    (\widehat{\boldsymbol{S}}_{N})_{\gg}p_{N}(t) \wkto \widehat{\boldsymbol{S}}_{\gg}p(t).
\end{align*}
Furthermore, since $((\widehat{\boldsymbol{S}}_{N})_{\gg}p_{N}(t),\partial^{+}\widehat{\boldsymbol{S}}_{N})$ and $(\widehat{\boldsymbol{S}}_{\gg}p(t),\partial^{+}\widehat{\boldsymbol{S}})$ are consistent, we have 
\begin{align*}
    \curl(\widehat{\boldsymbol{S}}_{N})_{\gg}p_{N}(t) = \curl\hspace{0.2em}\widehat{\boldsymbol{S}}_{\gg}p(t)
\end{align*}
and so we can appeal to Corollary \ref{cr:CorrectorCty} to deduce that $\hat{u}_{N} \wkto \hat{u}$ in $\BV(\Omega;\R^{3})$. Since we also have $t^{N}_{j(t,N)} \to t$, it follows from \eqref{PartialStability4} and \eqref{A:2-iii} that
\begin{align}\label{PartialStability5}
    \Ecal(t^{N}_{j(t,N)},\hat{u}_{N},(\widehat{\boldsymbol{S}}_{N})_{\gg}p_{N}(t),\partial^{+}\widehat{\boldsymbol{S}}_{N}) \to \Ecal(t,\hat{u},\widehat{\boldsymbol{S}}_{\gg}p(t),\partial^{+}\widehat{\boldsymbol{S}}).
\end{align}
Furthermore, we can use \eqref{NinftyConnectivity} and Lemma \ref{lm:Concatenation} to deduce that
\begin{align}\label{PartialStability6}
    \Diss(\widehat{\boldsymbol{S}}_{N}) \to \Diss(\widehat{\boldsymbol{S}}).
\end{align}
Finally, it follows from Lemma \ref{lm:ELSC} (taking into account Lemma \ref{lm:InterpolantAP}) that
\begin{align}\label{PartialStability7}
    \Ecal(t,u(t),p(t),\boldsymbol{S}(t)) \leq \liminf_{N \to \infty} \Ecal(t^{N}_{j(t,N)},u_{N}(t),p_{N}(t),\boldsymbol{S}_{N}(t)).
\end{align}
The stability estimate follows by taking a lower limit in \eqref{PartialStability3} and appealing to \eqref{PartialStability5}, \eqref{PartialStability6}, and \eqref{PartialStability7}.
\end{subproof}

\begin{subproof}[(iii)] After taking a lower limit of \eqref{InterpolantLowerEnergy}, appealing to \eqref{PartialStability7} for the first term, \eqref{NinftySRestrictConv} for the third term, and the dominated convergence theorem for the fourth term, we obtain \eqref{NinftyLowerEnergy}. 
\end{subproof}

\begin{subproof}[(iv)] This follows immediately from \eqref{NinftySpConv}, \eqref{NinftySRestrictConv}, and \eqref{pNDefn}. \qedhere
\end{subproof}
\end{proof}

\subsection{Proof of Theorem \ref{tm:DiscreteExist}} Take $(u_{0},p_{0},\boldsymbol{T}_{0}) \in \Qcal_{1/n}$ satisfying \eqref{DiscreteInitialStability} (notice that we did not need to assume this in the previous sections) and let $\overline{C} > 0$ be the constant from Lemma \ref{lm:NinftyAP}. For each $M \in \N$ satisfying \eqref{MCond}, we can proceed as in Section \ref{s:Ninfty} to construct a triple $(u_{M},p_{M},\boldsymbol{S}_{M})$ where
\begin{align*}
    u_{M} &\in \Lrm^{\infty}([0,T];\BV(\Omega;\R^{3})), \\
    p_{M} &\in \BV([0,T];\Mcal(\overline{\Omega};\R^{3 \times 3})), \\
    \boldsymbol{S}_{M} &\in \BV([0,T];\Disl_{1/n}(\overline{\Omega}))
\end{align*}
and a $J_{M} \subset [0,T]$ with $\mathscr{L}^{1}(J_{M}) = 0$ such that conclusions of Lemma \ref{lm:NinftyAP} and Proposition \ref{pp:NinftyProperties} hold,
\begin{align*}
    u_{M}(t) &= u_{*}(t,p_{M}(t),\boldsymbol{S}_{M}(t)) \hspace{1em} \text{for $t \in [0,T]\backslash J_{M}$},
\end{align*}
and
\begin{align*}
    \partial^{-}\boldsymbol{S}_{M} &= \boldsymbol{T}_{0}, & p_{M}(0) &= p_{0}.
\end{align*}
For $M \in \N$ satisfying \eqref{MCond}, recall from Lemma \ref{lm:NinftyAP} that
\begin{align*}
   \Var(p_{M}) + \Vert \boldsymbol{S}_{M} \Vert_{\Lrm^{\infty}_{c}} + \Diss(\boldsymbol{S}_{M}) + \Wcal_{c}(\partial^{-}\boldsymbol{S}_{M}) + \Wcal_{c}(\partial^{+}\boldsymbol{S}_{M}) \leq 2\overline{C}\alpha_{0}.
\end{align*}
We can therefore appeal again to Proposition \ref{pp:BVHelly} and Proposition \ref{pp:BVCompactness} to obtain (not relabelled) subsequences of $(p_{M})_{M}$ and $(\boldsymbol{S}_{M})_{M}$ and 
\begin{align*}
    p &\in \BV([0,T];\Mcal(\overline{\Omega};\R^{3 \times 3})), \\
    \boldsymbol{S} &\in \BV([0,T];\Disl_{1/n}(\overline{\Omega})),
\end{align*}
such that
\begin{align}\label{DiscreteExistSpConv}
    p_{M}(t) &\wkto p(t) \hspace{1em} \text{for $t \in [0,T]$}, & \boldsymbol{S}_{M} &\wkto \boldsymbol{S}, 
\end{align}
and
\begin{align}\label{DiscreteExistDissVarLSC}
    \Var(p) &\leq \liminf_{M \to \infty} \Var(p_{M}), & \Diss(\boldsymbol{S}) &\leq \liminf_{M \to \infty} \Diss(\boldsymbol{S}_{M}).
\end{align}
Furthermore, for $\mathscr{L}^{1}$-a.e.\ $t \in [0,T]$, we have
\begin{align}\label{DiscreteExistSliceConv}
    \tbf_{\#}\Vert \boldsymbol{S} \Vert(\{t\}) &= 0, & \boldsymbol{S}_{M}(t) &\wkto \boldsymbol{S}(t).
\end{align}
and
\begin{align}\label{DiscreteExistRestrictConv}
    \boldsymbol{S}_{M} \mvert [0,t] &\wkto \boldsymbol{S} \mvert [0,t], & \Diss(\boldsymbol{S};[0,t]) &\leq \liminf_{M \to \infty} \Diss(\boldsymbol{S}_{M};[0,t]). 
\end{align}
Since $\boldsymbol{S}_{M} \wkto \boldsymbol{S}$, we can use Lemma \ref{lm:EndpointsCty} to deduce that $\partial^{\pm}\boldsymbol{S}_{M} \wkto \partial^{\pm}\boldsymbol{S}$. It then follows from \eqref{CoreLSC} that
\begin{align}\label{DiscreteExistEndpointsCoreLSC}
    \Wcal_{c}(\partial^{\pm}\boldsymbol{S}) &\leq \liminf_{M \to \infty} \Wcal_{c}(\partial^{\pm}\boldsymbol{S}_{M}).
\end{align}
Furthermore, since $\partial^{-}{\boldsymbol{S}}_{M} = \boldsymbol{T}_{0}$ for $M \in \N$ satisfying \eqref{MCond}, we have $\partial^{-}\boldsymbol{S} = \boldsymbol{T}_{0}$. Similarly, since $p_{M}(0) = p_{0}$ for $M \in \N$ satisfying \eqref{MCond}, we have $p(0) = p_{0}$. Let $J \subset [0,T]$ be the null set on which \eqref{DiscreteExistSliceConv} and \eqref{DiscreteExistRestrictConv} fail. After possibly enlarging $J$, we can assume that
\begin{align*}
    J_{M} \subset J \hspace{1em} \text{for $M \in \N$ satisfying \eqref{MCond}}.
\end{align*}
Take $t \in [0,T]\backslash J$. Since $u_{M}(t) = u_{*}(t,p_{M}(t),\boldsymbol{S}_{M}(t))$ for $M \in \N$ satisfying \eqref{MCond}, we can appeal to Corollary \ref{cr:uCty} (taking into account Lemma \ref{lm:NinftyAP}) to deduce that $u(t) := u_{*}(t,p(t),\boldsymbol{S}(t))$ exists and $u_{M}(t) \wkto u(t)$. In particular, we have defined 
\begin{align*}
    u: [0,T]\backslash J \to \BV(\Omega;\R^{3})
\end{align*}
and we have
\begin{align*}
    u_{M}(t) \wkto u(t) \hspace{1em} \text{for $t \in [0,T]\backslash J$}.
\end{align*}
It follows from Lemma \ref{lm:MeasurableCty} that $u$ is weakly$^{*}$ $\mathscr{L}^{1}$-measurable. We can now proceed as in the proof of Lemma \ref{lm:NinftyAP} to deduce that the bounds \eqref{DiscreteExistBounds1} and \eqref{DiscreteExistBounds2} hold, and we note in passing that $u \in \Lrm^{\infty}([0,T];\BV(\Omega;\R^{3}))$.

\begin{claim} The triple $(u,p,\boldsymbol{S})$ is an approximate $n$-discrete energetic solution with initial condition $(u_{0},p_{0},\boldsymbol{T}_{0})$ and error parameters $(0,\delta)$.
\end{claim}

\begin{proof} We have already shown that $\partial^{-}\boldsymbol{S} = \boldsymbol{T}_{0}$ and that $p(0) = p_{0}$. Fix $t \in (0,T)\backslash J$.

\begin{subproof}[\textit{Consistency}] This is immediate from Lemma \ref{lm:QClosure} (taking into account Lemma \ref{lm:NinftyAP}).
\end{subproof}

\begin{subproof}[\textit{Stability}] Since $\boldsymbol{S}_{M}(t) \wkto \boldsymbol{S}(t)$ and $\Wcal_{c}(\boldsymbol{S}_{M}(t)) \leq \overline{C}\alpha_{0}$ for $M \in \N$ satisfying \eqref{MCond}, we can appeal to Proposition \ref{pp:AppliedIntegralSTC} to yield, for all $M \in \N$ along a (not relabelled) subsequence, a $\boldsymbol{R}_{M} \in \Slip_{1/n}(\boldsymbol{S}_{M}(t))$ with $\partial^{+}\boldsymbol{R}_{M} = \boldsymbol{S}(t)$ and
\begin{align*}
    \Diss(\boldsymbol{R}_{M}) &\to 0, & 
    \Vert \boldsymbol{R}_{M} \Vert_{\Lrm^{\infty}_{c}} \leq (C+1) \cdot \overline{C}\alpha_{0}.
\end{align*}
Take $\hat{u} \in \BV(\Omega;\R^{3})$ and $\widehat{\boldsymbol{S}} \in \Slip_{1/n}(\boldsymbol{S}(t))$ with 
\begin{align*}
    (\hat{u},\widehat{\boldsymbol{S}}_{\gg}p(t),\partial^{+}\widehat{\boldsymbol{S}}) \in \Qcal_{1/n}.
\end{align*}
For $M \in \N$ along the subsequence, we define 
\begin{align*}
    \widehat{\boldsymbol{S}}_{M} := \widehat{\boldsymbol{S}} \circ \boldsymbol{R}_{M}
\end{align*}
and then appeal to Lemma \ref{lm:Corrector} to find $\hat{u}_{M} \in \BV_{h_{0},\chi}(\Omega;\R^{3})$ with 
\begin{align*}
    \symfr[D\hat{u}_{M}-(\widehat{\boldsymbol{S}}_{M})_{\gg}p_{M}(t)] &= \symfr[D\hat{u} - \widehat{\boldsymbol{S}}_{\gg}p(t)].
\end{align*}
Notice that $\widehat{\boldsymbol{S}}_{M} \in \Slip_{1/n}(\boldsymbol{S}_{M}(t))$ and that $\partial^{+}\widehat{\boldsymbol{S}}_{M} = \partial^{+}\widehat{\boldsymbol{S}}$. We can work as in the proof of Proposition \ref{pp:NinftyProperties} (ii) to deduce that
\begin{align*}
    (\hat{u}_{M},(\widehat{\boldsymbol{S}}_{M})_{\gg}p_{M}(t),\partial^{+}\widehat{\boldsymbol{S}}_{M}) &\in \Qcal_{1/n}, & \Vert \widehat{\boldsymbol{S}}_{M} \Vert_{\Lrm^{\infty}_{c}} &\leq \max\{\Vert \widehat{\boldsymbol{S}}\Vert_{\Lrm^{\infty}_{c}}, (C+1)\cdot\overline{C}\alpha_{0}\},
\end{align*}
and
\begin{align}\label{DiscreteExistStability1}
    \hat{u}_{M} &\wkto \hat{u} \hspace{1em} \text{in $\BV(\Omega;\R^{3})$}, & \Diss(\widehat{\boldsymbol{S}}_{M}) &\to \Diss(\widehat{\boldsymbol{S}}).
\end{align}
For $M \in \N$ far enough along the subsequence, we test with $(\hat{u}_{M},\widehat{\boldsymbol{S}}_{M})$ in Proposition \ref{pp:NinftyProperties} (ii) to deduce that 
\begin{align}\label{DiscreteExistStability2}
    \Ecal(t,u_{M}(t),p_{M}(t),\boldsymbol{S}_{M}(t)) \leq \Ecal(t,\hat{u}_{M},(\widehat{\boldsymbol{S}}_{M})_{\gg}p_{M}(t),\partial^{+}\widehat{\boldsymbol{S}}_{M}) + \Diss(\widehat{\boldsymbol{S}}_{M}).
\end{align}
Notice that 
\begin{align*}
    \Ecal(t,\hat{u}_{M},(\widehat{\boldsymbol{S}}_{M})_{\gg}p_{M}(t),\partial^{+}\widehat{\boldsymbol{S}}_{M}) = \Ecal(t,\hat{u},\widehat{\boldsymbol{S}}_{\gg}p(t),\partial^{+}\widehat{\boldsymbol{S}}) + \langle f(t),\hat{u}-\hat{u}_{M} \rangle.
\end{align*}
Since $\hat{u}_{M} \wkto \hat{u}$ in $\BV(\Omega;\R^{3})$, we have
\begin{align}\label{DiscreteExistStability3}
     \Ecal(t,\hat{u}_{M},(\widehat{\boldsymbol{S}}_{M})_{\gg}p_{M}(t),\partial^{+}\widehat{\boldsymbol{S}}_{M}) \to \Ecal(t,\hat{u},\widehat{\boldsymbol{S}}_{\gg}p(t),\partial^{+}\widehat{\boldsymbol{S}}).
\end{align}
Furthermore, it follows from Lemma \ref{lm:ELSC} (taking into account Lemma \ref{lm:NinftyAP}) that
\begin{align}\label{DiscreteExistStability4}
    \Ecal(t,u(t),p(t),\boldsymbol{S}(t)) &\leq \liminf_{M \to \infty} \Ecal(t,u_{M}(t),p_{M}(t),\boldsymbol{S}_{M}(t)).
\end{align}
The stability estimate now follows after taking a lower limit in \eqref{DiscreteExistStability2} and appealing to \eqref{DiscreteExistStability3}, \eqref{DiscreteExistStability1}, and \eqref{DiscreteExistStability4}.
\end{subproof}

\begin{subproof}[\textit{Plastic Flow}] This follows immediately from \eqref{DiscreteExistSpConv}, \eqref{DiscreteExistRestrictConv}, and Proposition \ref{pp:NinftyProperties} (iv).
\renewcommand{\qed}{}
\end{subproof}

\begin{subproof}[\textit{Energy Balance}] It follows from Proposition \ref{pp:NinftyProperties} (iii) that, for $M \in \N$ satisfying \eqref{MCond}, we have 
\begin{equation}\label{DiscreteExistMLowerEnergy}
\begin{aligned}
    \Ecal(t,&u_{M}(t),p_{M}(t),\boldsymbol{S}_{M}(t)) - \Ecal(0,u_{0},p_{0},\boldsymbol{T}_{0}) \vphantom{\int_{t_{1}}^{t_{2}}} \\&+ \Diss(\boldsymbol{S}_{M};[0,t]) +\int_{0}^{t}\langle \dot{f}(s),u_{M}(s)\rangle \dd s \leq 0.
\end{aligned}
\end{equation}
By taking a lower limit of \eqref{DiscreteExistMLowerEnergy}, appealing to \eqref{DiscreteExistStability4} for the first term, \eqref{DiscreteExistRestrictConv} for the third term, and the dominated convergence theorem for the fourth term, we obtain 
\begin{align}\label{DiscreteExistLowerEnergy}
    \Ecal(t,u(t),p(t),\boldsymbol{S}(t)) - \Ecal(0,u_{0},p_{0},\boldsymbol{T}_{0}) + \Diss(\boldsymbol{S};[0,t]) +\int_{0}^{t}\langle \dot{f}(s),u(s)\rangle \dd s \leq 0.
\end{align}

To establish the reverse inequality, we shall use the stability estimate established above. Take $\sigma,\tau \in [0,T]\backslash J$ with $\sigma < \tau$ and notice that
\begin{align*}
    \boldsymbol{S} \mvert [\sigma,\tau] \in \BV([\sigma,\tau];\Disl_{1/n}(\overline{\Omega})).
\end{align*}
Since $\tbf_{\#}\Vert S \Vert(\{\sigma,\tau\}) = 0$, we can use \eqref{RestrictionEndpoints} to deduce that
\begin{align*}
    \partial^{-}(\boldsymbol{S} \mvert [\sigma,\tau]) &= \boldsymbol{S}(\sigma), & \partial^{+}(\boldsymbol{S} \mvert [\sigma,\tau]) &= \boldsymbol{S}(\tau).
\end{align*}
and we can use the plastic flow condition to deduce that 
\begin{align}\label{DiscreteExistEnergy1}
    p(\tau) = (\boldsymbol{S} \mvert [\sigma,\tau])_{\gg}p(\sigma).
\end{align}
If $\boldsymbol{S} \mvert [\sigma,\tau] \in \Slip_{1/n}(\boldsymbol{S}(\sigma))$, then we could test with $u(\tau)$ and $\boldsymbol{S} \mvert [\sigma,\tau]$ in the stability estimate for $(u,p,\boldsymbol{S})$ at $\sigma$ to obtain
\begin{align}\label{DiscreteExistEnergy1.1}
    \Ecal(\sigma,u(\sigma),p(\sigma),\boldsymbol{S}(\sigma)) \leq \Ecal(\sigma,u(\tau),p(\tau),\boldsymbol{S}(\tau)) + \Diss(\boldsymbol{S};[\sigma,\tau]).
\end{align}
It would then follow from \eqref{EDefn} that
\begin{equation}\label{DiscreteExistEnergy1.2}
\begin{aligned}
    \Ecal(\tau,u(\tau),p(\tau),\boldsymbol{S}(\tau)) &-\Ecal(\sigma,u(\sigma),p(\sigma),\boldsymbol{S}(\sigma)) \\
    &\geq - \Diss(\boldsymbol{S};[\sigma,\tau])-\int_{\sigma}^{\tau} \langle \dot{f}(s),u(\tau)\rangle \dd s.
\end{aligned}
\end{equation}
However, in general, we can only conclude that $\boldsymbol{S} \mvert [\sigma,\tau] \in \BV([\sigma,\tau];\Disl_{1/n}(\overline{\Omega}))$. Instead, we appeal to Proposition \ref{pp:SlipLipApprox} to obtain a sequence $(\widehat{\boldsymbol{S}}_{j})_{j} \subset \Slip_{1/n}(\boldsymbol{S}(\sigma);[\sigma,\tau])$ with $\partial^{+}\widehat{\boldsymbol{S}}_{j} = \boldsymbol{S}(\tau)$ for $j \in \N$ and
\begin{align}\label{DiscreteExistEnergy2}
    \widehat{\boldsymbol{S}}_{j} &\wkto \boldsymbol{S} \mvert [\sigma,\tau], & \Diss(\widehat{\boldsymbol{S}}_{j}) &\to \Diss(\boldsymbol{S};[\sigma,\tau]), & \sup_{j \in \N} \int \Wcal_{c}(\widehat{\boldsymbol{S}}_{j}(t)) \dd t &< \infty.
\end{align}
For $j \in \N$, we then appeal to Lemma \ref{lm:Corrector} to find $\hat{u}_{j} \in \BV_{h_{0},\chi}(\Omega;\R^{3})$ with
\begin{align}\label{DiscreteExistEnergy3}
    \symfr[D\hat{u}_{j}-(\widehat{\boldsymbol{S}}_{j})_{\gg}p(\sigma)] = \symfr[Du(\tau)-p(\tau)].
\end{align}
Fix $j \in \N$. Since $(p(\sigma),\boldsymbol{S}(\sigma))$ is consistent, we can use Lemma \ref{lm:FlowConsistent} to deduce that $((\widehat{\boldsymbol{S}}_{j})_{\gg}p(\sigma),\partial^{+}\widehat{\boldsymbol{S}}_{j})$ is consistent. It then follows from \eqref{DiscreteExistEnergy3} that
\begin{align*}
    (\hat{u}_{j},(\widehat{\boldsymbol{S}}_{j})_{\gg}p(\sigma),\partial^{+}\widehat{\boldsymbol{S}}_{j}) \in \Qcal_{1/n}
\end{align*}
We shall now rescale the sequence $(\widehat{\boldsymbol{S}}_{j})_{j}$ in time to obtain competitors in the stability estimate for $(u,p,\boldsymbol{S})$ at $\sigma$. Let $a: [\sigma,\tau] \to [0,1]$ where 
\begin{align*}
    a(t) = \frac{t-\sigma}{\tau-\sigma} \hspace{1em} \text{for $t \in [\sigma,\tau]$}.
\end{align*}
It then follows from Proposition \ref{pp:DissMaxCoreRescale} that
\begin{align}\label{DiscreteExistEnergy5}
    \Diss(a_{*}\widehat{\boldsymbol{S}}_{j}) &= \Diss(\widehat{\boldsymbol{S}}_{j}), & a_{*}\widehat{\boldsymbol{S}}_{j} &\in \Slip_{1/n}(\boldsymbol{S}(\sigma)),
\end{align} 
and we can use \eqref{EndpointsRescale} and Lemma \ref{lm:PlasticRescale} to deduce that
\begin{align}\label{DiscreteExistEnergy6}
    \partial^{+}(a_{*}\widehat{\boldsymbol{S}}_{j}) &= \partial^{+}\widehat{\boldsymbol{S}}_{j}, &(a_{*}\widehat{\boldsymbol{S}}_{j})_{\gg}p(\sigma) &= (\widehat{\boldsymbol{S}}_{j})_{\gg}p(\sigma).
\end{align}
By testing with $(\hat{u}_{j},a_{*}\widehat{\boldsymbol{S}}_{j})$ in the stability estimate for $(u,p,\boldsymbol{S})$ at $\sigma$, taking into account \eqref{DiscreteExistEnergy5} and \eqref{DiscreteExistEnergy6}, we obtain
\begin{align}\label{DiscreteExistEnergy7}
    \Ecal(\sigma,u(\sigma),p(\sigma),\boldsymbol{S}(\sigma)) \leq \Ecal(\sigma,\hat{u}_{j},(\widehat{\boldsymbol{S}}_{j})_{\gg}p(\sigma),\partial^{+}\widehat{\boldsymbol{S}}_{j}) + \Diss(\widehat{\boldsymbol{S}}_{j}).
\end{align}
We can then use \eqref{DiscreteExistEnergy3} and the fact that $\partial^{+}\widehat{\boldsymbol{S}}_{j} = \boldsymbol{S}(\tau)$ to deduce that
\begin{align*}
    \Ecal(\sigma,\hat{u}_{j},(\widehat{\boldsymbol{S}}_{j})_{\gg}p(\sigma),\partial^{+}\widehat{\boldsymbol{S}}_{j}) = \Ecal(\sigma,u(\tau),p(\tau),\boldsymbol{S}(\tau)) + \langle f(\sigma), u(\tau) - \hat{u}_{j} \rangle.
\end{align*}
Since $\widehat{\boldsymbol{S}}_{j} \wkto \boldsymbol{S} \mvert [\sigma,\tau]$, we can use Lemma \ref{lm:pCty} and \eqref{DiscreteExistEnergy1} to deduce that
\begin{align*}
    (\widehat{\boldsymbol{S}}_{j})_{\gg}p(\sigma) \wkto p(\tau).
\end{align*}
Furthermore, since $((\widehat{\boldsymbol{S}}_{j})_{\gg}p(\sigma),\partial^{+}\widehat{\boldsymbol{S}}_{j})$ and $(p(\tau),\boldsymbol{S}(\tau))$ are consistent, we have 
\begin{align*}
    \curl (\widehat{\boldsymbol{S}}_{j})_{\gg}p(\sigma) = \curl \hspace{0.2em} p(\tau)
\end{align*}
and so we can appeal to Corollary \ref{cr:CorrectorCty} to obtain the convergence $\hat{u}_{j} \wkto u(\tau)$ in $\BV(\Omega;\R^{3})$, whereby
\begin{align}\label{DiscreteExistEnergy8}
    \Ecal(\sigma,\hat{u}_{j},(\widehat{\boldsymbol{S}}_{j})_{\gg}p(\sigma),\partial^{+}\widehat{\boldsymbol{S}}_{j}) \to \Ecal(\sigma,u(\tau),p(\tau),\boldsymbol{S}(\tau)).
\end{align}
We now take $j \to \infty$ in \eqref{DiscreteExistEnergy7}. Appealing to \eqref{DiscreteExistEnergy8} and \eqref{DiscreteExistEnergy2} yields \eqref{DiscreteExistEnergy1.1}, and it follows that \eqref{DiscreteExistEnergy1.2} holds.

Now consider the interval $[0,\tau]$. Notice that 
\begin{align*}
    \boldsymbol{S} \mvert [0,\tau] \in \BV([0,\tau];\Disl_{1/n}(\overline{\Omega})).
\end{align*}
It follows from \eqref{RestrictionEndpoints} that
\begin{align*}
    \partial^{-}(\boldsymbol{S}\mvert[0,\tau]) &= \partial^{-}\boldsymbol{S} = \boldsymbol{T}_{0}, & \partial^{+}(\boldsymbol{S}\mvert[0,\tau]) &= \boldsymbol{S}(\tau).
\end{align*}
Since the plastic flow condition holds, we have
\begin{align*}
    p(\tau) = (\boldsymbol{S} \mvert [0,\tau])_{\gg}p_{0}.
\end{align*}
By repeating the arguments made above on the interval $[\sigma,\tau]$, using \eqref{DiscreteInitialStability} in place of the stability estimate at $\sigma$, we obtain
\begin{equation}\label{DiscreteExistEnergy11}
\begin{aligned}
    \Ecal(\tau,u(\tau),p(\tau),\boldsymbol{S}(\tau)) &-\Ecal(0,u_{0},p_{0},\boldsymbol{T}_{0}) \\
    &\geq - \Diss(\boldsymbol{S};[0,\tau])-\int_{0}^{\tau} \langle \dot{f}(s),u(\tau)\rangle \dd s - \delta.
\end{aligned}
\end{equation}
Now take $N \in \N$ and let $0 = t_{0} < \cdots < t_{N} = t$ be a subdivision of $[0,t]$ with $\{t_{i}\}_{i=1}^{N} \subset (0,T)\backslash J$. By combining the estimates \eqref{DiscreteExistEnergy1.2} (applied with $[\sigma,\tau] := [t_{i},t_{i+1}]$ for $i \in \{1,\dotsc,N-1\}$) and \eqref{DiscreteExistEnergy11} (applied with $\tau=t_{1}$), we obtain
\begin{equation}\label{DiscreteExistEnergy11.5}
\begin{aligned}
    \Ecal(t,u(t),p(t),\boldsymbol{S}(t)) &-\Ecal(0,u_{0},p_{0},\boldsymbol{T}_{0}) \\
    &\geq - \Diss(\boldsymbol{S};[0,t])-\sum_{i=0}^{N-1}\int_{t_{i}}^{t_{i+1}} \langle \dot{f}(s),u(t_{i+1})\rangle \dd s - \delta.
\end{aligned}
\end{equation}

Finally, we shall use the regularity of $f$ (see \eqref{A:2-iii}) to rectify the integral term in \eqref{DiscreteExistEnergy11.5}.
To this end, take $\varepsilon > 0$. We shall show that there is $N \in \N$ and a subdivision $0 = t_{0} \leq \cdots \leq t_{N} = t$ of $[0,t]$ with the property that
\begin{align}\label{DiscreteExistEnergy12}
    \bigg\vert \sum_{i=0}^{N-1}\int_{t_{i}}^{t_{i+1}} \langle \dot{f}(s),u(t_{i+1})\rangle \dd s - \int_{0}^{t} \langle \dot{f}(s),u(s) \rangle \dd s \bigg\vert \leq \varepsilon.
\end{align}
It follows from \eqref{A:2-iii} that $\dot{f}: [0,T] \to \Lrm^{\infty}(\Omega;\R^{3})$ is uniformly continuous and so there is $\eta > 0$ such that, for any $\sigma,\tau \in [0,T]$,
\begin{align*}
    \vert \tau - \sigma \vert \leq \eta \Rightarrow \Vert \dot{f}(\tau)-\dot{f}(\sigma)\Vert_{\Lrm^{\infty}} \leq \frac{\varepsilon}{2\overline{C}\alpha_{0}\cdot T}.
\end{align*}
Now take $\sigma,\tau \in [0,T]$ with $\vert \tau - \sigma \vert \leq \eta$. Then, for $v \in \Lrm^{1}(\Omega;\R^{3})$, we have 
\begin{align*}
    \bigg\vert \int_{\sigma}^{\tau}\langle \dot{f}(s),v\rangle \dd s - (\tau-\sigma)\langle \dot{f}(\tau),v\rangle \bigg\vert \leq \vert \tau - \sigma \vert \cdot \Vert v \Vert_{\Lrm^{1}} \cdot \frac{\varepsilon}{2\overline{C}\alpha_{0} \cdot T}.
\end{align*}
It follows (taking into account the uniform estimate \eqref{DiscreteExistBounds2}) that, if $N \in \N$ and $0 = t_{0} \leq \cdots \leq t_{N} = t$ is a subdivision of $[0,t]$ such that $\vert t_{i+1}-t_{i} \vert \leq \eta$ and $t_{i+1} \in [0,t]\backslash J$ for $i \in \{0,\dotsc,N-1\}$, then 
\begin{align}\label{DiscreteExistEnergy13}
    \bigg\vert \sum_{i=0}^{N-1} \int_{t_{i}}^{t_{i+1}} \langle \dot{f}(s),u(t_{i+1})\rangle \dd s - \sum_{i=0}^{N-1} (t_{i+1}-t_{i})\langle \dot{f}(t_{i+1}),u(t_{i+1})\rangle \bigg\vert \leq \frac{\varepsilon}{2}.
\end{align}
Finally, since the map $s \mapsto \langle \dot{f}(s),u(s)\rangle$ is $\mathscr{L}^{1}$-measurable and uniformly bounded (see Remark \ref{rm:EnergyIntegralWellDefined}), we can appeal to Lemma \ref{lm:Hahn} to find a $N \in \N$ and a subdivision $0 = t_{0} \leq t_{1} \leq \cdots \leq t_{N} = t$ of $[0,t]$ with $\vert t_{i+1}-t_{i} \vert \leq \eta$ and $t_{i+1} \in [0,t]\backslash J$ for $i \in \{0,\dotsc,N-1\}$ such that
\begin{align}\label{DiscreteExistEnergy14}
    \bigg\vert \sum_{i=0}^{N-1} (t_{i+1}-t_{i})\langle \dot{f}(t_{i+1}),u(t_{i+1})\rangle - \int_{0}^{t} \langle \dot{f}(s),u(s) \rangle \dd s\bigg\vert \leq \frac{\varepsilon}{2}.
\end{align}
The estimate \eqref{DiscreteExistEnergy12} follows from \eqref{DiscreteExistEnergy13} and \eqref{DiscreteExistEnergy14}. Since the choice of $\varepsilon > 0$ was arbitrary and \eqref{DiscreteExistEnergy11.5} holds, we have
\begin{equation}\label{DiscreteExistEnergy15}
\begin{aligned}
    \Ecal(t,u(t),p(t),\boldsymbol{S}(t)) &-\Ecal(0,u_{0},p_{0},\boldsymbol{T}_{0}) \\
    &\geq - \Diss(\boldsymbol{S};[0,t])-\int_{0}^{t} \langle \dot{f}(s),u(s)\rangle \dd s - \delta.
\end{aligned}
\end{equation}
Combining \eqref{DiscreteExistLowerEnergy} and \eqref{DiscreteExistEnergy15} yields the result. \qedhere
\end{subproof}
\end{proof}

\section{Homogenisation of Approximate Discrete Energetic Solutions}

\subsection{Proof of Theorem \ref{tm:Forwards}} After possibly enlarging $J \subset [0,T]$, we can ensure that $J$ contains the jump sets of $(u_{j},p_{j},\boldsymbol{S}_{j})$ for all $j \in \N$. Since $\partial^{-}\boldsymbol{S}_{j} = \boldsymbol{T}_{0}^{j}$ for $j \in \N$, $\boldsymbol{S}_{j} \wkto \boldsymbol{S}$, and $\boldsymbol{T}_{0}^{j} \wkto \boldsymbol{T}_{0}$, we can use Lemma \ref{lm:EndpointsCty} to deduce that $\partial^{-}\boldsymbol{S} = \boldsymbol{T}_{0}$. Similarly, since $p_{j}(0)=p_{0}^{j}$ for $j \in \N$, $p_{0}^{j} \wkto p_{0}$, and $p_{j}(0) \wkto p(0)$, we have $p(0) = p_{0}$. Fix $t \in (0,T)\backslash J$.

\begin{subproof}[\textit{Consistency}] This follows immediately from Lemma \ref{lm:QClosure} (taking into account \eqref{ForwardsBounds}). 
\end{subproof}

\begin{subproof}[\textit{Stability}] Take $\hat{u} \in \BV(\Omega;\R^{3})$ and $\widehat{\boldsymbol{S}} \in \overline{\Slip}(\boldsymbol{S}(t))$ with $(\hat{u},\widehat{\boldsymbol{S}}_{\gg}p(t),\partial^{+}\widehat{\boldsymbol{S}}) \in \Qcal$. Since $\boldsymbol{S}_{j}(t) \wkto \boldsymbol{S}(t)$ and $\Wcal_{c}(\boldsymbol{S}_{j}(t)) \leq \overline{C}$ for $j \in \N$, we can use Proposition \ref{pp:SlipPrescribedApprox} to yield, for all $j \in \N$ along a (not relabelled) subsequence, an 
\begin{align*}
    \widehat{\boldsymbol{S}}_{j} \in \Slip_{1/n(j)}(\boldsymbol{S}_{j}(t))
\end{align*}
with $\Wcal_{c}(\partial^{+}\widehat{\boldsymbol{S}}_{j}) \to \Wcal_{c}(\partial^{+}\widehat{\boldsymbol{S}})$ and 
\begin{align*}
    \widehat{\boldsymbol{S}}_{j} &\wkto \widehat{\boldsymbol{S}}, & \Diss(\widehat{\boldsymbol{S}}_{j}) &\to \Diss(\widehat{\boldsymbol{S}}), & \sup_{j \in \N} \int \Wcal_{c}(\widehat{\boldsymbol{S}}_{j}(t)) \dd t &< \infty.
\end{align*}
For $j \in \N$, we then appeal to Lemma \ref{lm:Corrector} to find $\hat{u}_{j} \in \BV_{h_{0},\chi}(\Omega;\R^{3})$ with 
\begin{align}\label{ForwardsStability1}
    \symfr[D\hat{u}_{j}-(\widehat{\boldsymbol{S}}_{j})_{\gg}p_{j}(t)] = \symfr[D\hat{u}-\widehat{\boldsymbol{S}}_{\gg}p(t)].
\end{align}
It follows from \eqref{ForwardsStability1} and Lemma \ref{lm:FlowConsistent} that
\begin{align*}
    (\hat{u}_{j},(\widehat{\boldsymbol{S}}_{j})_{\gg}p_{j}(t),\partial^{+}\widehat{\boldsymbol{S}}_{j}) \in \Qcal_{1/n(j)}.
\end{align*}
We can therefore test with $(\hat{u}_{j},\widehat{\boldsymbol{S}}_{j})$ in the (approximate) stability estimate for $(u_{j},p_{j},\boldsymbol{S}_{j})$ to obtain
\begin{align}\label{ForwardsStability2}
    \Ecal(t,u_{j}(t),p_{j}(t),\boldsymbol{S}_{j}(t)) \leq \Ecal(t,\hat{u}_{j},(\widehat{\boldsymbol{S}}_{j})_{\gg}p_{j}(t),\partial^{+}\widehat{\boldsymbol{S}}_{j}) + \Diss(\widehat{\boldsymbol{S}}_{j}) + \delta_{j}(t). 
\end{align}
It follows from \eqref{ForwardsStability1} and \eqref{EDefn} that
\begin{equation}\label{ForwardsStability3}
\begin{aligned}
    \Ecal(t,\hat{u}_{j},(\widehat{\boldsymbol{S}}_{j})_{\gg}p_{j}(t),\partial^{+}\widehat{\boldsymbol{S}}_{j}) &= \Ecal(t,\hat{u},\widehat{\boldsymbol{S}}_{\gg}p(t),\partial^{+}\widehat{\boldsymbol{S}}) \\
    &+ h(\Wcal_{c}(\partial^{+}\widehat{\boldsymbol{S}}_{j})) - h(\Wcal_{c}(\partial^{+}\widehat{\boldsymbol{S}})) + \langle f(t),\hat{u}-\hat{u}_{j}\rangle.
\end{aligned}
\end{equation}
Since $p_{j}(t) \wkto p(t)$ and $\widehat{\boldsymbol{S}}_{j} \wkto \widehat{\boldsymbol{S}}$, we can appeal to Lemma \ref{lm:pCty} to deduce that
\begin{align*}
    (\widehat{\boldsymbol{S}}_{j})_{\gg}p_{j}(t) \wkto \widehat{\boldsymbol{S}}_{\gg}p(t).
\end{align*}
We can use Lemma \ref{lm:EndpointsCty} to deduce that $\partial^{+}\widehat{\boldsymbol{S}}_{j} \wkto \partial^{+}\widehat{\boldsymbol{S}}$. Since $((\widehat{\boldsymbol{S}}_{j})_{\gg}p_{j}(t),\partial^{+}\widehat{\boldsymbol{S}}_{j})$ is consistent for $j \in \N$, $(\widehat{\boldsymbol{S}}_{\gg}p(t),\partial^{+}\widehat{\boldsymbol{S}})$ is consistent, and $(\Wcal_{c}(\partial^{+}\widehat{\boldsymbol{S}}_{j}))_{j}$ is uniformly bounded, we must have
\begin{align*}
    \curl (\widehat{\boldsymbol{S}}_{j})_{\gg}p_{j}(t) \wkto \curl\hspace{0.2em} \widehat{\boldsymbol{S}}_{\gg}p(t)
\end{align*}
We can therefore invoke Corollary \ref{cr:CorrectorCty} to deduce that $\hat{u}_{j} \wkto \hat{u}$. Moreover, since $h$ is continuous, we have $h(\Wcal_{c}(\partial^{+}\widehat{\boldsymbol{S}}_{j})) \to h(\Wcal_{c}(\partial^{+}\widehat{\boldsymbol{S}}))$. It now follows from \eqref{ForwardsStability3} that
\begin{align}\label{ForwardsStability4}
     \Ecal(t,\hat{u}_{j},(\widehat{\boldsymbol{S}}_{j})_{\gg}p_{j}(t),\partial^{+}\widehat{\boldsymbol{S}}_{j}) \to \Ecal(t,\hat{u},\widehat{\boldsymbol{S}}_{\gg}p(t),\partial^{+}\widehat{\boldsymbol{S}}).
\end{align}
Furthermore, we can use Lemma \ref{lm:ELSC} (taking into account \eqref{ForwardsBounds}) to deduce that 
\begin{align}\label{ForwardsStability5}
    \Ecal(t,u(t),p(t),\boldsymbol{S}(t)) \leq \liminf_{j \to \infty}\Ecal(t,u_{j}(t),p_{j}(t),\boldsymbol{S}_{j}(t))
\end{align}
Finally, since $\delta_{j}(t) \to 0$, $\Diss(\widehat{\boldsymbol{S}}_{j}) \to \Diss(\widehat{\boldsymbol{S}})$, and \eqref{ForwardsStability4} and \eqref{ForwardsStability5} hold, we can take a lower limit in \eqref{ForwardsStability2} to yield
\begin{align*}
    \Ecal(t,u(t),p(t),\boldsymbol{S}(t)) \leq \Ecal(t,\hat{u},\widehat{\boldsymbol{S}}_{\gg}p(t),\partial^{+}\widehat{\boldsymbol{S}}) + \Diss(\widehat{\boldsymbol{S}}). 
\end{align*}

\end{subproof}

\begin{subproof}[\textit{Energy Balance}] First notice that
\begin{equation}\label{ForwardsEnergy1}
\begin{aligned}
    \Ecal(t,&u_{j}(t),p_{j}(t),\boldsymbol{S}_{j}(t)) - \Ecal(0,u_{0}^{j},p_{0}^{j},\boldsymbol{T}_{0}^{j}) \\
    &+ \Diss(\boldsymbol{S}_{j};[0,t]) +\int_{0}^{t}\langle \dot{f}(s),u_{j}(s)\rangle \dd s \leq \eta_{j}(t).
\end{aligned}
\end{equation}
After taking a lower limit of \eqref{ForwardsEnergy1}, using the fact that $\eta_{j}(t) \to 0$ and appealing to \eqref{ForwardsStability5} for the first term, \eqref{ForwardsInitialConv} for the second term, \eqref{ForwardsSRestrictConv} for the third term, together with the dominated convergence theorem for the fourth term, we obtain
\begin{align*}
    \Ecal(t,u(t),p(t),\boldsymbol{S}(t)) - \Ecal(0,u_{0},p_{0},\boldsymbol{T}_{0}) + \Diss(\boldsymbol{S};[0,t]) +\int_{0}^{t}\langle \dot{f}(s),u(s)\rangle \dd s \leq 0.
\end{align*}
One can then work as in the proof of Theorem \ref{tm:DiscreteExist} to establish the reverse inequality. 
\end{subproof}

\begin{subproof}[\textit{Plastic Flow}] This follows immediately from \eqref{ForwardsInitialConv}, \eqref{ForwardsSRestrictConv}, and the fact that the plastic flow condition holds for $j \in \N$.
\end{subproof}

\subsection{Proof of Theorem \ref{tm:FieldExist}} The proof proceeds by approximating the initial values $(u_{0},p_{0},\boldsymbol{T}_{0})$ in such a way that we can use Theorem \ref{tm:DiscreteExist} to construct $n: \N \to \N$ and a suitable sequence $((u_{j},p_{j},\boldsymbol{S}_{j}))_{j}$ of approximate $n(j)$-discrete energetic solutions to which we can apply Theorem \ref{tm:Forwards}.

We start by applying Proposition \ref{pp:DislApprox} to $\boldsymbol{T}_{0}$ to obtain $n: \N \to \N$ and a sequence 
\begin{align*}
    (\boldsymbol{T}_{0}^{j})_{j} \subset \overline{\Disl}(\overline{\Omega})
\end{align*}
such that $\boldsymbol{T}_{0}^{j} \in \Disl_{1/n(j)}(\overline{\Omega})$ for $j \in \N$ and 
\begin{align}\label{FieldExist1}
    \boldsymbol{T}_{0}^{j} &\wkto \boldsymbol{T}_{0}, & \Wcal_{c}(\boldsymbol{T}_{0}^{j}) &\to \Wcal_{c}(\boldsymbol{T}_{0}).
\end{align}
For $j \in \N$, we replace $n(j)$ with 
\begin{align*}
    j! \cdot \prod_{i=1}^{j} n(i)
\end{align*}
so that $n$ is strictly increasing and, for $m, N \in \N$, there is $j \geq N$ with $m \vert n(j)$. We appeal to Proposition \ref{pp:AppliedNormalSTC} to find, for $j \in \N$, an 
\begin{align*}
    \boldsymbol{R}^{j}_{0} \in \overline{\Slip}(\boldsymbol{T}^{j}_{0})
\end{align*}
with $\partial^{+}\boldsymbol{R}_{0}^{j} = \boldsymbol{T}_{0}$ and $\Diss(\boldsymbol{R}_{0}^{j}) \to 0$. For $j \in \N$, we set
\begin{align*}
    p^{j}_{0} := (\boldsymbol{R}_{0}^{j})^{-1}_{\gg}p_{0}
\end{align*}
and appeal to Lemma \ref{lm:Corrector} to find $u^{j}_{0} \in \BV_{h_{0},\chi}(\Omega;\R^{3})$ with
\begin{align*}
    \symfr[Du_{0}^{j}-p_{0}^{j}] &= \symfr[Du_{0}-p_{0}].
\end{align*}
It follows from Lemma \ref{lm:FlowConsistent} that $(p_{0}^{j},\boldsymbol{T}^{j}_{0})$ is consistent for $j \in \N$ and so $(u_{0}^{j},p_{0}^{j},\boldsymbol{T}_{0}^{j}) \in \Qcal_{1/n(j)}$ for $j \in \N$. Notice that
\begin{align}\label{FieldExist2}
    \Ecal(0,u_{0}^{j},p_{0}^{j},\boldsymbol{T}_{0}^{j}) = \Ecal(0,u_{0},p_{0},\boldsymbol{T}_{0}) + \langle f(0),u_{0}-u_{0}^{j} \rangle + h(\Wcal_{c}(\boldsymbol{T}_{0}^{j})) - h(\Wcal_{c}(\boldsymbol{T}_{0})).
\end{align}
We can use Lemma \ref{lm:PlasticPathBV} and the fact that $\Diss(\boldsymbol{R}^{j}_{0}) \to 0$ to yield 
\begin{align}\label{FieldExist3}
    \Vert p_{0}^{j} - p_{0} \Vert(\overline{\Omega}) \to 0.
\end{align}
Since $(p_{0}^{j},\boldsymbol{T}^{j}_{0})$ is consistent for $j \in \N$, $(p_{0},\boldsymbol{T}_{0})$ is consistent, and \eqref{FieldExist1} holds, we must have
\begin{align*}
    \curl \hspace{0.2em} p_{0}^{j} \wkto \curl \hspace{0.2em} p_{0}
\end{align*}
and so we can use Corollary \ref{cr:CorrectorCty} to deduce that $u^{j}_{0} \wkto u_{0}$. Moreover, since $h$ is continuous, we have $h(\Wcal_{c}(\boldsymbol{T}_{0}^{j})) \to h(\Wcal_{c}(\boldsymbol{T}_{0}))$ and so it follows from \eqref{FieldExist2} that
\begin{align}\label{FieldExist4}
    \Ecal(0,u_{0}^{j},p_{0}^{j},\boldsymbol{T}_{0}^{j}) \to \Ecal(0,u_{0},p_{0},\boldsymbol{T}_{0}).
\end{align}

Fix $j \in \N$. Take $\hat{u} \in \BV(\Omega;\R^{3})$ and $\widehat{\boldsymbol{S}} \in \Slip_{1/n(j)}(\boldsymbol{T}_{0}^{j})$ with $(\hat{u},\widehat{\boldsymbol{S}}_{\gg}p^{j}_{0},\partial^{+}\widehat{\boldsymbol{S}}) \in \Qcal_{1/n(j)}$. It follows from Lemma \ref{lm:PlasticConcatenation} that 
\begin{align}\label{FieldExist5}
    \widehat{\boldsymbol{S}}_{\gg}p_{0}^{j} = (\widehat{\boldsymbol{S}} \circ (\boldsymbol{R}^{j}_{0})^{-1})_{\gg}p_{0}.
\end{align}
Set $\widehat{\boldsymbol{S}}_{j} := \widehat{\boldsymbol{S}} \circ (\boldsymbol{R}^{j}_{0})^{-1}$ and notice that $\widehat{\boldsymbol{S}}_{j} \in \overline{\Slip}(\boldsymbol{T}_{0})$. Since $(\hat{u},(\widehat{\boldsymbol{S}}_{j})_{\gg}p_{0},\partial^{+}\widehat{\boldsymbol{S}}_{j}) \in \Qcal$, we can test with $(\hat{u},\widehat{\boldsymbol{S}}_{j})$ in the stability estimate \eqref{FieldInitialStability} and deduce that
\begin{align*}
    \Ecal(0,u_{0},p_{0},\boldsymbol{T}_{0}) \leq \Ecal(0,\hat{u},(\widehat{\boldsymbol{S}}_{j})_{\gg}p_{0},\partial^{+}\widehat{\boldsymbol{S}}_{j}) + \Diss(\widehat{\boldsymbol{S}}_{j}).
\end{align*}
We can then use \eqref{FieldExist5}, Lemma \ref{lm:Concatenation}, Lemma \ref{lm:Reversal}, and the fact that $\partial^{+}\widehat{\boldsymbol{S}}_{j} = \partial^{+}\widehat{\boldsymbol{S}}$ to deduce that
\begin{align*}
    \Ecal(0,u_{0},p_{0},\boldsymbol{T}_{0}) \leq \Ecal(0,\hat{u},\widehat{\boldsymbol{S}}_{\gg}p_{0}^{j},\partial^{+}\widehat{\boldsymbol{S}}) + \Diss(\widehat{\boldsymbol{S}}) + \Diss(\boldsymbol{R}_{0}^{j}).
\end{align*}
Now define
\begin{align*}
    \delta_{0}^{j} := \Ecal(0,u_{0}^{j},p_{0}^{j},\boldsymbol{T}_{0}^{j}) - \Ecal(0,u_{0},p_{0},\boldsymbol{T}_{0}) + \Diss(\boldsymbol{R}^{j}_{0})
\end{align*}
so that $\delta_{0}^{j} \to 0$ and 
\begin{align*}
    \Ecal(0,u_{0}^{j},p_{0}^{j},\boldsymbol{T}_{0}^{j}) \leq \Ecal(0,\hat{u},\widehat{\boldsymbol{S}}_{\gg}p_{0}^{j},\partial^{+}\widehat{\boldsymbol{S}}) + \Diss(\widehat{\boldsymbol{S}}) + \delta_{0}^{j}.
\end{align*}

We can now use Theorem \ref{tm:DiscreteExist} to deduce that there is a $J_{j} \subset [0,T]$ with $\mathscr{L}^{1}(J_{j}) = 0$, and an approximate $n(j)$-discrete energetic solution $(u_{j},p_{j},\boldsymbol{S}_{j})$ with jump set $J_{j}$, error parameters $(0,\delta_{j}^{0})$, and initial condition $(u_{0}^{j},p_{0}^{j},\boldsymbol{T}_{0}^{j})$. For $j \in \N$, define
\begin{align*}
    \alpha_{0}^{j} := \Ecal(0,u_{0}^{j},p_{0}^{j},\boldsymbol{T}_{0}^{j})+C_{\lambda}(1+\Vert p_{0}^{j} \Vert(\overline{\Omega})+\Wcal_{c}(\boldsymbol{T}_{0}^{j})).
\end{align*}
It follows from Theorem \ref{tm:DiscreteExist} that
\begin{align}\label{FieldExist6}
    \Var(p_{j}) + \Diss(\boldsymbol{S}_{j})+\Wcal_{c}(\partial^{-}\boldsymbol{S}_{j})+\Wcal_{c}(\partial^{+}\boldsymbol{S}_{j}) \leq \overline{C}\alpha_{0}^{j}
\end{align}
and 
\begin{equation}\label{FieldExist7}
\begin{aligned}
    \Vert u_{j}(t) \Vert_{\BV} + \Vert \symfr[Du_{j}(t)-p_{j}(t)]\Vert_{\Lrm^{2}} + \Vert p_{j}(t) \Vert(\overline{\Omega}) + \Wcal_{c}(\boldsymbol{S}_{j}(t)) &\leq \overline{C}\alpha_{0}^{j}
\end{aligned}
\end{equation}
for all $t \in [0,T]\backslash J_{j}$. Here, $\overline{C} > 0$ depends only on the data in Section \ref{s:Assumptions} and the constant $\lambda > 0$ from \eqref{A:3}.
Since \eqref{FieldExist1}, \eqref{FieldExist3}, and \eqref{FieldExist4} hold, we can find $\alpha_{0}$ with 
\begin{align*}
    \alpha_{0}^{j} \leq \alpha_{0} \hspace{1em} \text{for $j \in \N$}.
\end{align*}
We can therefore appeal to Proposition \ref{pp:BVHelly} and Proposition \ref{pp:BVCompactness} to obtain (not relabelled) subsequences of $(p_{j})_{j}$ and $(\boldsymbol{S}_{j})_{j}$ and
\begin{align*}
    p &\in \BV([0,T];\Mcal(\overline{\Omega};\R^{3 \times 3})), \\
    \boldsymbol{S} &\in \BV([0,T];\overline{\Disl}(\overline{\Omega})),
\end{align*}
such that
\begin{align*}
    p_{j}(t) &\wkto p(t) \hspace{1em} \text{for $t \in [0,T]$}, & \boldsymbol{S}_{j} &\wkto \boldsymbol{S},
\end{align*}
and
\begin{align*}
    \Var(p) &\leq \liminf_{j \to \infty} \Var(p_{j}), &
    \Diss(\boldsymbol{S}) &\leq \liminf_{j \to \infty} \Diss(\boldsymbol{S}_{j}).
\end{align*}
Furthermore, for $\mathscr{L}^{1}$-a.e.\ $t \in [0,T]$, we have
\begin{align}\label{FieldExistSliceConv}
    \tbf_{\#}\Vert \boldsymbol{S} \Vert(\{t\}) &= 0, & \boldsymbol{S}_{j}(t) &\wkto \boldsymbol{S}(t),
\end{align}
and
\begin{align}\label{FieldExistSRestrictConv}
    \boldsymbol{S}_{j} \mvert [0,t] &\wkto \boldsymbol{S} \mvert [0,t], & \Diss(\boldsymbol{S};[0,t]) &\leq \liminf_{j \to \infty} \Diss(\boldsymbol{S}_{j};[0,t]). 
\end{align}
Let $J \subset [0,T]$ be the null set on which \eqref{FieldExistSliceConv} and \eqref{FieldExistSRestrictConv} fail. After possibly enlarging $J$, we can assume that
\begin{align*}
    J_{j} \cap [0,T] \subset J \hspace{1em} \text{for $j \in \N$ along the subsequence.}
\end{align*}
Fix $t \in (0,T)\backslash J$. For $j \in \N$, we can test with $\widehat{\boldsymbol{S}} := [\![0,1]\!] \times \boldsymbol{S}_{j}(t)$ in the stability estimate for $(u_{j}(t),p_{j}(t),\boldsymbol{S}_{j}(t))$ and deduce from Corollary \ref{cr:uUnique} that $u_{j}(t) = u_{*}(t,p_{j}(t),\boldsymbol{S}_{j}(t))$. We can therefore appeal to Corollary \ref{cr:uCty} to deduce that $u(t) := u_{*}(t,p(t),\boldsymbol{S}(t))$ exists and $u_{j}(t) \wkto u(t)$. It follows from Lemma \ref{lm:MeasurableCty} that $u$ is weakly$^{*}$ $\mathscr{L}^{1}$-measurable. In view of the uniform bound on $\Vert u_{j}(t) \Vert_{\BV}$ given above, this procedure defines
\begin{align*}
    u \in \Lrm^{\infty}([0,T];\BV(\Omega;\R^{3})).
\end{align*}
The result now follows from Theorem \ref{tm:Forwards}. \hfill \qed

\subsection{Proof of Theorem \ref{tm:Backwards}} We start by applying Proposition \ref{pp:SlipRationalApprox} to $\boldsymbol{S}$ to obtain an $n: \N \to \N$ and, for $j \in \N$, a 
\begin{align*}
    \boldsymbol{S}_{j} \in \BV([0,T];\Disl_{1/n(j)}(\overline{\Omega}))
\end{align*}
such that
\begin{align*}
    \boldsymbol{S}_{j} &\wkto \boldsymbol{S}, & \Diss(\boldsymbol{S}_{j}) &\to \Diss(\boldsymbol{S}), & \Wcal_{c}(\partial^{\pm}\boldsymbol{S}_{j}) &\to \Wcal_{c}(\partial^{\pm}\boldsymbol{S}).
\end{align*}
Furthermore, for $\mathscr{L}^{1}$-a.e.\ $t \in [0,T]$, we have
\begin{align}\label{Backwards1}
    \Wcal_{c}(\boldsymbol{S}_{j}(t)) &\to \Wcal_{c}(\boldsymbol{S}(t)), & \Wcal_{c}(\boldsymbol{S}_{j}(t)) &\leq C \cdot \Vert \boldsymbol{S} \Vert_{\Lrm^{\infty}_{c}} + 1
\end{align}
and
\begin{align}\label{Backwards2}
    \boldsymbol{S}_{j} \mvert [0,t] &\wkto \boldsymbol{S} \mvert [0,t], & \Diss(\boldsymbol{S}_{j};[0,t]) &\to \Diss(\boldsymbol{S};[0,t]).
\end{align}
Let $J \subset [0,T]$ be the null set on which \eqref{Backwards1} and \eqref{Backwards2} fail. After possibly enlarging $J$, we can assume that $J$ contains the jump set of $(u,p,\boldsymbol{S})$, that $\{0,T\} \subset J$, and that, for $t \in [0,T]\backslash J$,
\begin{align}\label{Backwards2.5}
    \tbf_{\#}\Vert \boldsymbol{S}_{j} \Vert(\{t\}) &= 0 \hspace{1em} \text{for $j \in \N$}, & \tbf_{\#}\Vert \boldsymbol{S} \Vert(\{t\}) &= 0, & \Wcal_{c}(\boldsymbol{S}(t)) &\leq \Vert \boldsymbol{S} \Vert_{\Lrm^{\infty}_{c}}.
\end{align}
For $j \in \N$, set $\boldsymbol{T}_{0}^{j} := \partial^{-}\boldsymbol{S}_{j}$. Notice that $\partial^{-}\boldsymbol{S} = \boldsymbol{T}_{0}$ and so we can use Lemma \ref{lm:EndpointsCty} to deduce that $\boldsymbol{T}_{0}^{j} \wkto \boldsymbol{T}_{0}$. We then appeal to Proposition \ref{pp:AppliedNormalSTC} to find, for $j \in \N$, an
\begin{align*}
    \boldsymbol{R}^{j}_{0} \in \overline{\Slip}(\boldsymbol{T}_{0}^{j})
\end{align*}
with $\partial^{+}\boldsymbol{R}_{0}^{j} = \boldsymbol{T}_{0}$ and $\Diss(\boldsymbol{R}_{0}^{j}) \to 0$. For $j \in \N$, define
\begin{align*}
    p_{0}^{j} := (\boldsymbol{R}_{0}^{j})^{-1}_{\gg}p_{0}.
\end{align*}
Now use Lemma \ref{lm:Corrector} to find $u_{0}^{j} \in \BV_{h_{0},\chi}(\Omega;\R^{3})$ with
\begin{align}\label{Backwards3}
    \symfr[Du_{0}^{j}-p_{0}^{j}] &= \symfr[Du_{0}-p_{0}].
\end{align}
It follows from Lemma \ref{lm:FlowConsistent} that $(p_{0}^{j},\boldsymbol{T}^{j}_{0})$ is consistent for $j \in \N$ and so $(u_{0}^{j},p_{0}^{j},\boldsymbol{T}_{0}^{j}) \in \Qcal_{1/n(j)}$ for $j \in \N$. Notice that
\begin{align}\label{Backwards4}
    \Ecal(0,u_{0}^{j},p_{0}^{j},\boldsymbol{T}_{0}^{j}) = \Ecal(0,u_{0},p_{0},\boldsymbol{T}_{0}) + \langle f(0),u_{0}-u_{0}^{j} \rangle + h(\Wcal_{c}(\boldsymbol{T}_{0}^{j})) - h(\Wcal_{c}(\boldsymbol{T}_{0})).
\end{align}
We can use Lemma \ref{lm:PlasticPathBV} to deduce that 
\begin{align*}
    \Vert p_{0}^{j} - p_{0} \Vert(\overline{\Omega}) \to 0.
\end{align*}
Since $(p_{0}^{j},\boldsymbol{T}_{0}^{j})$ is consistent for $j \in \N$, $(p_{0},\boldsymbol{T}_{0})$ is consistent, $\boldsymbol{T}^{j}_{0} \wkto \boldsymbol{T}_{0}$, and $(\Wcal_{c}(\boldsymbol{T}^{j}_{0}))_{j}$ is uniformly bounded,
we must have
\begin{align*}
    \curl \hspace{0.2em} p_{0}^{j} \wkto \curl \hspace{0.2em} p_{0}
\end{align*}
and so we can use Corollary \ref{cr:CorrectorCty} to deduce that $u_{0}^{j} \wkto u_{0}$. Moreover, since $h$ is continuous, we have $h(\Wcal_{c}(\boldsymbol{T}^{j}_{0})) \to h(\Wcal_{c}(\boldsymbol{T}_{0}))$ and so it follows from \eqref{Backwards4} that 
\begin{align}\label{Backwards5}
    \Ecal(0,u_{0}^{j},p_{0}^{j},\boldsymbol{T}_{0}^{j}) \to \Ecal(0,u_{0},p_{0},\boldsymbol{T}_{0}).
\end{align}

Now define $p_{j}: [0,T] \to \Mcal(\overline{\Omega};\R^{3 \times 3})$ where $p_{j}(0) = p_{0}^{j}$ and
\begin{align}\label{Backwards6}
    p_{j}(t) = p_{0}^{j} + \frac{1}{2}\sum_{b \in \Bcal} b \otimes \star \pbf_{*}(S_{j}^{b} \mvert [0,t] \times \R^{3}) \hspace{1em} \text{for $0 < t \leq T$}.
\end{align}
One can work as in the proof of Lemma \ref{lm:PlasticPathBV} to deduce that $p_{j} \in \BV([0,T];\Mcal(\overline{\Omega};\R^{3 \times 3}))$ and that
\begin{align*}
    \Var(p_{j})+\Diss(\boldsymbol{S}_{j}) \leq C \cdot \Diss(\boldsymbol{S}_{j}).
\end{align*}
Since $\Diss(\boldsymbol{S}_{j}) \to \Diss(\boldsymbol{S})$, there is $\overline{C} > 0$ such that \eqref{BackwardsBound1} holds for $j \in \N$. Notice that
\begin{align*}
    \Vert p_{j}(t) \Vert \leq \Vert p_{0}^{j} \Vert(\overline{\Omega}) + \Var(p_{j}) \hspace{1em} \text{for $j \in \N$, $0 < t \leq T$}.
\end{align*}
Since $\Vert p_{0}^{j} - p_{0} \Vert(\overline{\Omega}) \to 0$ and \eqref{BackwardsBound1} holds, we must have
\begin{align}\label{Backwards7}
    \Vert p_{j}(t) \Vert(\overline{\Omega}) \leq \overline{C} \hspace{1em} \text{for $j \in \N$, $t \in [0,T]$}.
\end{align}
after possibly adjusting $\overline{C}$. 

Fix $j \in \N$ and take $t \in [0,T]\backslash J$. Since $\tbf_{\#}\Vert \boldsymbol{S}_{j} \Vert(\{t\}) = 0$, we can use \eqref{RestrictionEndpoints} to deduce that 
\begin{align*}
     \partial^{+}(\boldsymbol{S}_{j} \mvert [0,t]) = \boldsymbol{S}_{j}(t).
\end{align*}
It then follows from Lemma \ref{lm:FlowConsistent} that $(p_{j}(t),\boldsymbol{S}_{j}(t))$ is consistent. We now appeal to Lemma \ref{lm:Corrector} to find $u_{j}(t) \in \BV_{h_{0},\chi}(\Omega;\R^{3})$ with
\begin{align}\label{Backwards8}
    \symfr[Du_{j}(t)-p_{j}(t)] &= \symfr[Du(t)-p(t)].
\end{align}
It follows that $(u_{j}(t),p_{j}(t),\boldsymbol{S}_{j}(t)) \in \Qcal_{1/n(j)}$. We can use \eqref{CorrectorEst} to deduce that
\begin{align}\label{Backwards8.1}
    \Vert u(t) - u_{j}(t) \Vert_{\BV} &\leq C(\Vert p(t)-p_{j}(t) \Vert(\overline{\Omega}) + \Vert \curl(p(t)-p_{j}(t))\Vert(\overline{\Omega})).
\end{align}
It follows from \eqref{CurlDistByCore} that
\begin{align*}
    \Vert \curl(p(t)-p_{j}(t))\Vert(\overline{\Omega}) \leq C(\Wcal_{c}(\boldsymbol{S}(t)) + \Wcal_{c}(\boldsymbol{S}_{j}(t))).
\end{align*}
We can then use \eqref{Backwards1} and \eqref{Backwards2.5} to estimate
\begin{align}\label{Backwards8.2}
    \Vert \curl(p(t)-p_{j}(t))\Vert(\overline{\Omega}) \leq C(\Vert \boldsymbol{S} \Vert_{\Lrm^{\infty}_{c}} + 1).
\end{align}
Combining \eqref{Backwards7}, \eqref{Backwards8.1}, and \eqref{Backwards8.2} yields
\begin{align}\label{Backwards8.3}
    \Vert u_{j}(t) \Vert_{\BV} \leq \overline{C}
\end{align}
after possibly adjusting $\overline{C}$. The estimate \eqref{BackwardsBound2} now follows from \eqref{Backwards1}, \eqref{Backwards7}, and \eqref{Backwards8.3} after adjusting $\overline{C}$ once more. It follows from Lemma \ref{lm:PlasticPathMeasurable} that $p_{j}$ and $p$ are $\mathscr{L}^{1}$-measurable on $[0,T]\backslash J$ with respect to the topology $\Tcal^{*}_{\mathrm{curl}}$ on $\Mcal_{\mathrm{curl}}(\overline{\Omega};\R^{3 \times 3})$. Since $u$ is weakly$^{*}$ $\mathscr{L}^{1}$-measurable by assumption, we can use Corollary \ref{cr:CorrectorCty} together with the bounds \eqref{Backwards7}, \eqref{Backwards8.2}, and \eqref{Backwards8.3} to deduce that $u_{j}: [0,T]\backslash J \to \BV(\Omega;\R^{3})$ is weakly$^{*}$ $\mathscr{L}^{1}$-measurable. We therefore have
\begin{align*}
    u_{j} \in \Lrm^{\infty}([0,T];\BV(\Omega;\R^{3})).
\end{align*}
Fix $t \in [0,T]\backslash J$ and take $j \in \N$. It follows from \eqref{Backwards8} that
\begin{equation}\label{Backwards9}
\begin{aligned}
    \Ecal(t,u_{j}(t),p_{j}(t),\boldsymbol{S}_{j}(t)) &= \Ecal(t,u(t),p(t),\boldsymbol{S}(t)) + \langle f(t),u(t)-u_{j}(t)\rangle \\
    &\qquad + h(\Wcal_{c}(\boldsymbol{S}_{j}(t))) - h(\Wcal_{c}(\boldsymbol{S}(t))).
\end{aligned}
\end{equation}
Since $p_{0}^{j} \wkto p_{0}$, $\boldsymbol{S}_{j} \mvert [0,t] \wkto \boldsymbol{S} \mvert [0,t]$, \eqref{Backwards6} holds, and $(u,p,\boldsymbol{S})$ satisfies the plastic flow condition at $t$, we have $p_{j}(t) \wkto p(t)$. Since $(p_{j}(t),\boldsymbol{S}_{j}(t))$ is consistent for $j \in \N$, $(p(t),\boldsymbol{S}(t))$ is consistent, $\boldsymbol{S}_{j}(t) \wkto \boldsymbol{S}(t)$, and \eqref{Backwards1} holds, we must have
\begin{align*}
    \curl \hspace{0.2em} p_{j}(t) \wkto \curl \hspace{0.2em} p(t)
\end{align*}
and so we can appeal to Corollary \ref{cr:CorrectorCty} to deduce that $u_{j}(t) \wkto u(t)$. Since $h$ is continuous, we have $h(\Wcal_{c}(\boldsymbol{S}_{j}(t))) \to h(\Wcal_{c}(\boldsymbol{S}(t)))$ and so it follows from \eqref{Backwards9} that
\begin{align}\label{Backwards10}
    \Ecal(t,u_{j}(t),p_{j}(t),\boldsymbol{S}_{j}(t)) \to \Ecal(t,u(t),p(t),\boldsymbol{S}(t)).
\end{align}

So far, we have established that, for $j \in \N$, the triple $(u_{j},p_{j},\boldsymbol{S}_{j})$ satisfies the conclusions (ii)-(v) of Theorem \ref{tm:Backwards}, the consistency and plastic flow conditions of Definition \ref{df:ApproxDiscreteEnergetic}, and
\begin{align*}
    p_{j}(0) &= p_{0}^{j}, & \partial^{-}\boldsymbol{S}_{j} &= \boldsymbol{T}^{j}_{0}.
\end{align*}
To prove the following claim, it therefore suffices to establish the stability and energy balance conditions of Definition \ref{df:ApproxDiscreteEnergetic} with error parameters to be determined in the course of the proof:

\begin{claim} For $j \in \N$, the triple $(u_{j},p_{j},\boldsymbol{S}_{j})$ is an approximate $n(j)$-discrete energetic solution with initial condition $(u_{0}^{j},p_{0}^{j},\boldsymbol{T}_{0}^{j})$.
\end{claim}

In proving the claim, we shall also establish conclusion (i) and the result will follow. Fix $t \in (0,T)\backslash J$.

\begin{proof}

\begin{subproof}[\textit{Stability}] Since $\boldsymbol{S}_{j}(t) \wkto \boldsymbol{S}(t)$ and \eqref{Backwards1} holds, we can appeal to Proposition \ref{pp:AppliedNormalSTC} to find, for $j \in \N$, an
\begin{align*}
    \boldsymbol{R}_{j} \in \overline{\Slip}(\boldsymbol{S}_{j}(t))
\end{align*}
with $\partial^{+}\boldsymbol{R}_{j} = \boldsymbol{S}(t)$ and $\Diss(\boldsymbol{R}_{j}) \to 0$. Fix $j \in \N$. Take $\hat{u} \in \BV(\Omega;\R^{3})$ and $\widehat{\boldsymbol{S}} \in \Slip_{1/n(j)}(\boldsymbol{S}_{j}(t))$ with 
\begin{align*}
    (\hat{u},\widehat{\boldsymbol{S}}_{\gg}p_{j}(t),\partial^{+}\widehat{\boldsymbol{S}}) \in \Qcal_{1/n(j)}.
\end{align*}
We define
\begin{align*}
    \widehat{\boldsymbol{S}}_{j} := \widehat{\boldsymbol{S}} \circ \boldsymbol{R}_{j}^{-1},
\end{align*}
and then appeal to Lemma \ref{lm:Corrector} to find $\hat{u}_{j} \in \BV_{h_{0},\chi}(\Omega;\R^{3})$ with
\begin{align}\label{BackwardsStability1}
    \symfr[D\hat{u}_{j}-(\widehat{\boldsymbol{S}}_{j})_{\gg}p(t)] = \symfr[D\hat{u} - \widehat{\boldsymbol{S}}_{\gg}p_{j}(t)].
\end{align}
Notice that $\widehat{\boldsymbol{S}}_{j} \in \overline{\Slip}(\boldsymbol{S}(t))$. It follows from \eqref{BackwardsStability1} and Lemma \ref{lm:FlowConsistent} that 
\begin{align*}
    (\hat{u}_{j},(\widehat{\boldsymbol{S}}_{j})_{\gg}p(t),\partial^{+}\widehat{\boldsymbol{S}}_{j}) \in \Qcal.
\end{align*}
We can therefore test with $(\hat{u}_{j},\widehat{\boldsymbol{S}}_{j})$ in the stability estimate for $(u,p,\boldsymbol{S})$ at $t$ to obtain
\begin{align}\label{BackwardsStability3}
    \Ecal(t,u(t),p(t),\boldsymbol{S}(t)) \leq \Ecal(t,\hat{u}_{j},(\widehat{\boldsymbol{S}}_{j})_{\gg}p(t),\partial^{+}\widehat{\boldsymbol{S}}_{j})+\Diss(\widehat{\boldsymbol{S}}_{j})
\end{align}
Notice that $\partial^{+}\widehat{\boldsymbol{S}}_{j} = \partial^{+}\widehat{\boldsymbol{S}}$; combining this observation with \eqref{BackwardsStability1} yields
\begin{align}\label{BackwardsStability4}
    \Ecal(t,\hat{u}_{j},(\widehat{\boldsymbol{S}}_{j})_{\gg}p(t),\partial^{+}\widehat{\boldsymbol{S}}_{j}) = \Ecal(t,\hat{u},\widehat{\boldsymbol{S}}_{\gg}p_{j}(t),\partial^{+}\widehat{\boldsymbol{S}}) + \langle f(t),\hat{u}-\hat{u}_{j} \rangle.
\end{align}
Define 
\begin{align}\label{BackwardsStability5}
    w_{j}(t):= \hat{u}-\hat{u}_{j} \in \BV(\Omega;\R^{3})
\end{align}
and notice that
\begin{equation}\label{BackwardsStability6}
\begin{aligned}
    D(w_{j}(t)) &= \fr[D\hat{u}-\widehat{\boldsymbol{S}}_{\gg}p_{j}(t)] + \widehat{\boldsymbol{S}}_{\gg}p_{j}(t) \mvert \Omega - \beta[\curl\hspace{0.2em} \widehat{\boldsymbol{S}}_{\gg}p_{j}(t)] \\
    &\qquad -\fr[D\hat{u}_{j}-(\widehat{\boldsymbol{S}}_{j})_{\gg}p(t)] - (\widehat{\boldsymbol{S}}_{j})_{\gg}p(t) \mvert \Omega + \beta[\curl (\widehat{\boldsymbol{S}}_{j})_{\gg}p(t)].
\end{aligned}
\end{equation}
Since $((\widehat{\boldsymbol{S}}_{j})_{\gg}p(t),\partial^{+}\widehat{\boldsymbol{S}}_{j})$ and $(\widehat{\boldsymbol{S}}_{\gg}p_{j}(t),\partial^{+}\widehat{\boldsymbol{S}})$ are consistent, we must have
\begin{align}\label{BackwardsStability7}
    \curl\hspace{0.2em}\widehat{\boldsymbol{S}}_{\gg}p_{j}(t) = \curl (\widehat{\boldsymbol{S}}_{j})_{\gg}p(t).
\end{align}
It follows from \eqref{BackwardsStability1}, \eqref{BackwardsStability6}, and \eqref{BackwardsStability7} that
\begin{align*}
    \sym\hspace{0.2em}D(w_{j}(t)) = \sym(\widehat{\boldsymbol{S}}_{\gg}p_{j}(t) - (\widehat{\boldsymbol{S}}_{j})_{\gg}p(t)) \mvert \Omega.
\end{align*}
We now use \eqref{PlasticForward} and \eqref{ConcatenationDefn} to deduce that
\begin{align*}
    \widehat{\boldsymbol{S}}_{\gg}p_{j}(t) - (\widehat{\boldsymbol{S}}_{j})_{\gg}p(t) &= p_{j}(t) - p(t) + \frac{1}{2}\sum_{b \in \Bcal} b \otimes \star\pbf_{*}(\hat{S}^{b}-\hat{S}^{b}_{j}) \\
    &=p_{j}(t) - p(t) - \frac{1}{2}\sum_{b \in \Bcal} b \otimes \star\pbf_{*}(R^{-1}_{j})^{b}.
\end{align*}
It follows that $\sym\hspace{0.2em}D(w_{j}(t))$ is independent of $(\hat{u},\widehat{\boldsymbol{S}})$. Furthermore, since
\begin{align*}
    \skew \int \chi \dd D(w_{j}(t)) &= 0, & [w_{j}(t)]_{H} = 0,
\end{align*}
we see (via Lemma \ref{lm:Korn} and \eqref{Poincare}) that $w_{j}(t)$ is independent of $(\hat{u},\widehat{\boldsymbol{S}})$. Notice that we can estimate as in the proof of Lemma \ref{lm:PlasticPathBV} and use Lemma \ref{lm:Reversal} to obtain
\begin{align*}
    \bigg\Vert \frac{1}{2}\sum_{b \in \Bcal} b \otimes \star\pbf_{*}(R^{-1}_{j})^{b} \bigg\Vert(\overline{\Omega}) \leq C\cdot \Diss(\boldsymbol{R}_{j}^{-1}) = C \cdot \Diss(\boldsymbol{R}_{j}).
\end{align*}
Since $\Diss(\boldsymbol{R}_{j}) \to 0$ and $p_{j}(t) \wkto p(t)$, it follows that $\sym\hspace{0.2em}D(w_{j}(t)) \wkto 0$ and so $w_{j}(t) \to 0$ in $\Lrm^{1}(\Omega;\R^{3})$. Returning to \eqref{BackwardsStability3}, notice that we can use Lemma \ref{lm:Concatenation} and Lemma \ref{lm:Reversal} to deduce that
\begin{align}\label{BackwardsStability8}
    \Diss(\widehat{\boldsymbol{S}}_{j}) = \Diss(\widehat{\boldsymbol{S}})+\Diss(\boldsymbol{R}_{j}). 
\end{align}
We now define 
\begin{align*}
    \delta_{j}(t) &:= \Ecal(t,u_{j}(t),p_{j}(t),\boldsymbol{S}_{j}(t)) - \Ecal(t,u(t),p(t),\boldsymbol{S}(t)) \\
    &\qquad + \Diss(\boldsymbol{R}_{j}) + \langle f(t),w_{j}(t)\rangle.
\end{align*}
Since $\Diss(\boldsymbol{R}_{j}) \to 0$, $w_{j}(t) \to 0$ in $\Lrm^{1}(\Omega;\R^{3})$, and \eqref{Backwards10} holds, we have $\delta_{j}(t) \to 0$. Finally, by combining \eqref{BackwardsStability3}, \eqref{BackwardsStability4}, \eqref{BackwardsStability5}, and \eqref{BackwardsStability8}, we obtain
\begin{align*}
    \Ecal(t,u_{j}(t),p_{j}(t),\boldsymbol{S}_{j}(t)) \leq \Ecal(t,\hat{u},\widehat{\boldsymbol{S}}_{\gg}p_{j}(t),\partial^{+}\widehat{\boldsymbol{S}}) + \Diss(\widehat{\boldsymbol{S}}) + \delta_{j}(t).
\end{align*}
\end{subproof}

\begin{subproof}[\textit{Energy Balance}] Define
\begin{align*}
    \eta_{j}(t) := \bigg\vert \Ecal(t,&u_{j}(t),p_{j}(t),\boldsymbol{S}_{j}(t)) - \Ecal(0,u_{0}^{j},p_{0}^{j},\boldsymbol{T}_{0}^{j}) \\
    &+ \Diss(\boldsymbol{S}_{j};[0,t]) + \int_{0}^{t} \langle \dot{f}(s),u_{j}(s) \rangle \dd s \bigg\vert.
\end{align*}
We can then use \eqref{Backwards10}, \eqref{Backwards5}, \eqref{Backwards2}, the dominated convergence theorem, and the energy balance for $(u,p,\boldsymbol{S})$ at $t$ to deduce that $\eta_{j}(t) \to 0$. This finishes the proof of the claim. \qedhere
\end{subproof}
\end{proof}

\appendix

\section{Miscellaneous Results}

In what follows, we collect some miscellaneous results needed throughout the paper.

\subsection{Functions of Bounded \texorpdfstring{$X$-Variation}{}} Fix $I \subset \R$. Let $X$ be a Banach space. Given a function $g: I \to X$ and $\tilde{I} \subset I$, we define the \textbf{$X$-variation} of $g$ on $\tilde{I}$ as follows:
\begin{align*}
    \Var_{X}(g;\tilde{I}) := \sup\bigg\{\sum_{i=1}^{N}\Vert g(t_{i})-g(t_{i-1})\Vert: N \in \N, \, t_{0} \leq t_{1} \leq \cdots \leq t_{N}, \, \{t_{i}\}_{i=0}^{N} \subset \tilde{I}\bigg\}.
\end{align*}
We shall also abbreviate $\Var_{X}(g) := \Var_{X}(g;I)$ and omit the subscript $X$ when the choice of Banach space is clear (for example, in Lemma \ref{lm:PlasticPathBV}, $X = \Mcal(\overline{\Omega};\R^{3 \times 3})$ and $\Vert \mu \Vert := \Vert \mu \Vert(\overline{\Omega})$ for $\mu \in X$). If $\Var_{X}(g) < \infty$, then we say that $g$ has \textbf{bounded $X$-variation}, and we shall use $\BV(I;X)$ to denote the space of such functions.

We shall often need to appeal to the following compactness result for $\BV(I;X)$:

\begin{proposition}\label{pp:BVHelly} Let $Y$ be a separable Banach space. Take a sequence $(g_{j})_{j} \subset \BV([0,T];Y^{*})$ and suppose that there is $M > 0$ with
\begin{align*}
    \Vert g_{j}(0) \Vert + \Var(g_{j}) \leq M \hspace{1em} \text{for $j \in \N$}.
\end{align*}
Then there is a (not relabelled) subsequence $(g_{j})_{j}$ and $g \in \BV([0,T];Y^{*})$ with 
\begin{align*}
    g_{j}(t) \wkto g(t) \hspace{1em} \text{for all $t \in [0,T]$}.
\end{align*}
Furthermore, we have
\begin{align*}
    \Var(g) \leq \liminf_{j \to \infty} \Var(g_{j}).
\end{align*}
\end{proposition}

\begin{proof} Since bounded subsets of $Y^{*}$ are pre-compact and metrisable in the weak$^{*}$ topology of $Y^{*}$, the result follows from \cite[Proposition B.5.13]{MielkeRoubicek15book}.
\end{proof}

\subsection{Measurability}\label{s:Measurability} Again, fix an $\mathscr{L}^{1}$-measurable $I \subset \R$. Let $X$ be a topological space. Recall that a map $g: I \to X$ is said to be \textbf{$\mathscr{L}^{1}$-measurable} if $g^{-1}(A)$ is $\mathscr{L}^{1}$-measurable for all Borel sets $A \subset X$. We shall now define two classes of measurable functions and discuss some of their properties.

Let $X$ be a Banach space. We shall say that $g$ is \textbf{strongly $\mathscr{L}^{1}$-measurable} (with values in $X$) if $g(I)$ is separable and $g$ is $\mathscr{L}^{1}$-measurable with respect to $X$ equipped with the strong topology. If we additionally have $X = Y^{*}$ for some Banach space $Y$, we shall say that $g$ is \textbf{weakly$^{*}$ $\mathscr{L}^{1}$-measurable} (with values in $X$) if it is $\mathscr{L}^{1}$-measurable with respect to $X$ equipped with the weak$^{*}$ topology. 

Suppose that $g$ is weakly$^{*}$ $\mathscr{L}^{1}$-measurable. Since the evaluation functionals are continuous with respect to the weak$^{*}$ topology by definition, the map $t \mapsto \langle g(t),y \rangle$ is $\mathscr{L}^{1}$-measurable for all $y \in Y$. It turns out that the converse is true when $Y$ is separable:

\begin{lemma}\label{lm:WeakToWeakBorel} Let $Y$ be a separable Banach space. Take a map $g: I \to Y^{*}$ and suppose that, for $y \in Y$, the map $t \mapsto \langle g(t),y \rangle$ is $\mathscr{L}^{1}$-measurable. Then $g$ is weakly$^{*}$ $\mathscr{L}^{1}$-measurable.
\end{lemma}

The following elementary lemma is often used throughout the paper to establish the weak$^{*}$ $\mathscr{L}^{1}$-measurability of various functions:

\begin{lemma}\label{lm:MeasurableCty} Let $Y$ be a separable Banach space and, for $j \in \N$, let $g_{j}: I \to Y^{*}$ be weakly$^{*}$ $\mathscr{L}^{1}$-measurable. Suppose that there is a map $g: I \to Y^{*}$ with the property that
\begin{align*}
    g_{j}(t) \wkto g(t) \hspace{1em} \text{for $\mathscr{L}^{1}$-a.e.\ $t \in I$}.
\end{align*}
Then $g$ is weakly$^{*}$ measurable.
\end{lemma}

\begin{proof} Fix $y \in Y$. Since $\langle g_{j}(t),y \rangle \to \langle g(t),y \rangle$ for $\mathscr{L}^{1}$-a.e.\ $t \in I$ and the map $t \mapsto \langle g_{j}(t),y\rangle$ is $\mathscr{L}^{1}$-measurable for $j \in \N$, we can use \cite[2.3.2 (6)]{Federer69book} to deduce that the map $t \mapsto \langle g(t),y \rangle$ is $\mathscr{L}^{1}$-measurable. Since $y \in Y$ was arbitrary, the result follows from Lemma \ref{lm:WeakToWeakBorel}.
\end{proof}

Let $Y$ be a Banach space. We shall use $\Lrm^{\infty}(I;Y^{*})$ to denote the set of all weakly$^{*}$ $\mathscr{L}^{1}$-measurable functions $g: I \to Y^{*}$ with
\begin{align*}
    \Vert g \Vert_{\Lrm^{\infty}(I;Y^{*})} := \esssup_{t \in I} \Vert g \Vert_{Y^{*}} < \infty,
\end{align*}
identifying functions that agree $\mathscr{L}^{1}$-a.e. on $I$ in the usual way.

\subsection{A Hahn Lemma} Finally, we present a version of the so-called \textit{Hahn Lemma} which allows one to approximate Lebesgue integrals by Riemann sums. The result given below is similar to \cite[Theorem A.3.1]{Mainik05PhD}, but here we additionally control the mesh-size of the partition and take the points in the partition outside of a prescribed $\mathscr{L}^{1}$-negligible set. Nevertheless, the proof is almost identical and only included for the sake of completeness.

\begin{lemma}\label{lm:Hahn} Fix $\sigma,\tau \in \R$ with $\sigma < \tau$, let $M > 0$ and take $J \subset (\sigma,\tau)$ with $\mathscr{L}^{1}(J) = 0$. Let $g: [\sigma,\tau] \to \R$ be $\mathscr{L}^{1}$-measurable and suppose that $\vert g(t) \vert \leq M$ for $t \in [\sigma,\tau]\backslash J$. Take $\eta, \varepsilon > 0$. Then there is $N \in \N$ and a subdivision $\sigma = t_{0} \leq \cdots \leq t_{N} = \tau$ of $[\sigma,\tau]$ with $\vert t_{i+1}-t_{i} \vert \leq \eta$ and $t_{i+1} \in [\sigma,\tau]\backslash J$ for $i \in \{0,\dotsc,N-1\}$ such that
\begin{align*}
    \bigg\vert \int_{\sigma}^{\tau} g(t) \dd t - \sum_{j=0}^{N-1} g(t_{j+1})(t_{j+1}-t_{j}) \bigg\vert \leq \varepsilon.
\end{align*}
\end{lemma}

\begin{proof} Fix $\alpha > 0$ (to be chosen later). We apply Lusin's theorem to obtain a continuous $\tilde{g}: [\sigma,\tau] \to \R$ with $\Vert \tilde{g} \Vert_{\infty} \leq M$ and
\begin{align*}
    \mathscr{L}^{1}(\{t \in [\sigma,\tau]: \tilde{g}(t) \neq g(t)\}) \leq \alpha.
\end{align*}
It follows that
\begin{align}\label{Hahn1}
    \bigg\vert \int_{\sigma}^{\tau} g(t)-\tilde{g}(t) \dd t \bigg\vert \leq 2M\alpha.
\end{align}
Since $\tilde{g}$ is continuous and bounded, it is Riemann integrable. We can therefore find $\tilde{N} \in \N$ and a subdivision $\sigma = s_{0} < \dotsc < s_{\tilde{N}} = \tau$ of $[\sigma,\tau]$ such that
\begin{align*}
    \vert s_{j+1} - s_{j} \vert \leq \frac{\eta}{2} \hspace{1em} \text{for $j \in \{0,\dotsc,\tilde{N}-1\}$}
\end{align*}
and, for $\{t_{j+1}\}_{j=0}^{\tilde{N}-1} \subset [\sigma,\tau]$ with $t_{j+1}\in [s_{j},s_{j+1}]$ for $j \in \{0,\dotsc,\tilde{N}-1\}$, we have
\begin{align}\label{Hahn2}
    \bigg\vert \int_{\sigma}^{\tau} \tilde{g}(t) \dd t - \sum_{j=0}^{\tilde{N}-1}\tilde{g}(t_{j+1})(s_{j+1}-s_{j}) \bigg\vert \leq \alpha.
\end{align}

Now define 
\begin{align*}
    B := \{t \in [\sigma,\tau]: \tilde{g}(t) \neq g(t)\} \cup J
\end{align*}
and set
\begin{align*}
    \Lambda := \{j \in \{0,\dotsc,\tilde{N}-1\}: [s_{j},s_{j+1}]\backslash B \neq \varnothing\}.
\end{align*}
Take $j \in \{0,\dotsc,\tilde{N}-1\}$. If $j \in \Lambda$, let $m_{j+1} = \inf\{s_{j+1}-t: t \in [s_{j},s_{j+1}]\backslash B\}$ so that $(s_{j+1}-m_{j+1},s_{j+1}) \subset B \cap [s_{j},s_{j+1}]$. Otherwise, notice that $[s_{j},s_{j+1}] \subset B$. It follows that
\begin{align}\label{Hahn3}
    \sum_{j \in \Lambda} m_{j+1} + \sum_{j \not\in \Lambda} s_{j+1} - s_{j} \leq \mathscr{L}^{1}(B) \leq \alpha.
\end{align}
Again, take $j \in \{0,\dotsc,\tilde{N}-1\}$. If $j \in \Lambda$, take $t_{j+1} \in [s_{j},s_{j+1}]\backslash B$ with 
\begin{align}\label{Hahn4}
    s_{j+1} - t_{j+1} \leq m_{j+1} + \frac{\alpha}{\tilde{N}}.
\end{align}
Otherwise, take $t_{j+1} \in [s_{j},s_{j+1}]\backslash J$, so that
\begin{align}\label{Hahn5}
    s_{j+1}-t_{j+1} \leq s_{j+1}-s_{j}.
\end{align}
Combining \eqref{Hahn3}, \eqref{Hahn4} and \eqref{Hahn5} yields
\begin{align}\label{Hahn6}
    \sum_{j=0}^{\tilde{N}-1} s_{j+1}-t_{j+1} \leq \sum_{j \in \Lambda} \bigg(m_{j+1} + \frac{\alpha}{\tilde{N}}\bigg) + \sum_{j \not\in \Lambda} s_{j+1}-s_{j} \leq 2\alpha.
\end{align}
Furthermore, since $t_{j+1} \in [s_{j},s_{j+1}]\backslash B$ for $j \in \Lambda$ and \eqref{Hahn3} holds, we have
\begin{align}\label{Hahn7}
    \bigg\vert \sum_{j=0}^{\tilde{N}-1} (\tilde{g}(t_{j+1})-g(t_{j+1}))(s_{j+1}-s_{j}) \bigg\vert = \bigg\vert \sum_{j \not\in \Lambda} (\tilde{g}(t_{j+1})-g(t_{j+1}))(s_{j+1}-s_{j}) \bigg\vert \leq 2M\alpha.
\end{align}
Now set $N = \tilde{N}+1$, $t_{0}=\sigma$, and $t_{N}=\tau$. We then use \eqref{Hahn6} and the fact that $t_{0}=s_{0}$ to deduce that
\begin{align}\label{Hahn8}
\begin{aligned}
    \bigg\vert \sum_{j=0}^{\tilde{N}-1} &g(t_{j+1})(s_{j+1}-s_{j}) - \sum_{j=0}^{N-1} g(t_{j+1})(t_{j+1}-t_{j}) \bigg\vert \\
    &\leq M \bigg(\sum_{j=0}^{\tilde{N}-1} s_{j+1}-t_{j+1} + \sum_{j=0}^{\tilde{N}-1} s_{j}-t_{j} + s_{\tilde{N}}-t_{\tilde{N}}\bigg) \\
    &= 2M \sum_{j=0}^{\tilde{N}-1} s_{j+1}-t_{j+1}\leq 4M\alpha
\end{aligned}
\end{align}
Combining \eqref{Hahn1}, \eqref{Hahn2}, \eqref{Hahn7}, and \eqref{Hahn8} yields
\begin{align*}
    \bigg\vert \int_{\sigma}^{\tau} g(t) \dd t - \sum_{j=0}^{N-1} g(t_{j+1})(t_{j+1}-t_{j}) \bigg\vert \leq (8M+1)\alpha.
\end{align*}
The result follows after choosing $\alpha > 0$ so that $(8M+1)\alpha \leq \varepsilon$.
\end{proof}

\bibliography{refs}
\bibliographystyle{plain}

\end{document}